\documentclass{amsart}
\usepackage{a4wide,verbatim}
\usepackage{latexsym}
\usepackage{amsmath}%
\usepackage{amsfonts}%
\usepackage{amssymb}%
\usepackage[francais,english]{babel}
\usepackage[latin1]{inputenc}
\usepackage[T1]{fontenc}
\usepackage[cyr]{aeguill} 
\usepackage{verbatim}

\newtheorem{corollary}{Corollary}[section]

\newtheorem{lemma}{Lemma}[section]

\newtheorem{proposition}{Proposition}[section]
\newtheorem{remark}{Remark}[section]
\newtheorem{theorem}{Theorem}[section]

\allowdisplaybreaks
\newcommand\beq{\begin{equation}}
\newcommand\eeq{\end{equation}}
\renewcommand{\emph}{\textbf}
\newcommand{\brk}[1]{\left(#1\right)}          
\newcommand{\Brk}[1]{\left[#1\right]}          
\newcommand{\BRK}[1]{\left\{#1\right\}}        
\newcommand{\Abs}[1]{\left| #1 \right|}        
\newcommand{\Scal}[2]{\left(#1,#2\right)}      
\newcommand{\Norm}[1]{\left\| #1 \right\|}     

\newcommand{\jump}[1]{[\![#1]\!]}
\newcommand{\xx}{\boldsymbol{x}}
\newcommand{\bolda}{\boldsymbol{a}}
\newcommand{\boldb}{\boldsymbol{b}}
\newcommand{\e}{\varepsilon}
\newcommand{\D}{\mathcal{D}}
\newcommand{\f}{\boldsymbol{f}}
\newcommand{\I}{\boldsymbol{I}}
\renewcommand{\to}{\rightarrow}

\newcommand{\str}{{\boldsymbol{\tau}}}
\newcommand{\strs}{{\boldsymbol{\sigma}}}

\newcommand{\bzero}{{\boldsymbol{0}}}
\newcommand{\bu}{{\boldsymbol{u}}}
\newcommand{\bv}{\boldsymbol{v}}
\newcommand{\bw}{\boldsymbol{w}}
\newcommand{\bz}{\boldsymbol{z}}
\newcommand{\be}{\boldsymbol{e}}
\newcommand{\bg}{\boldsymbol{g}}
\newcommand{\bs}{\boldsymbol{s}}

\newcommand{\gbu}{{\grad\bu}}
\newcommand{\gbv}{{\grad\bv}}

\newcommand{\bphi}{\boldsymbol{\phi}}
\newcommand{\bpsi}{\boldsymbol{\psi}}
\newcommand{\bchi}{\boldsymbol{\chi}}

\newcommand{\bxi}{\boldsymbol{\xi}}
\newcommand{\bvarsigma}{\boldsymbol{\varsigma}}
\newcommand{\buh}{\bu_h}

\newcommand{\gbuh}{\grad\bu_h}
\newcommand{\strh}{\strs_h}

\newcommand{\prd}{p_{\delta}}
\newcommand{\pa}{p_{\alpha}}
\newcommand{\pad}{p_{\alpha,\delta}}

\newcommand{\padorL}{p_{\alpha,\delta}^{(L)}}
\newcommand{\bud}{\bu_\delta}
\newcommand{\gbud}{\grad\bu_\delta}
\newcommand{\strd}{\strs_\delta}
\newcommand{\bua}{\bu_\alpha}
\newcommand{\gbua}{\grad\bu_\alpha}
\newcommand{\stra}{\strs_\alpha}
\newcommand{\buad}{\bu_{\alpha,\delta}}
\newcommand{\buadh}{\bu_{\alpha,\delta,h}}
\newcommand{\gbuad}{\grad\bu_{\alpha,\delta}}

\newcommand{\strad}{\strs_{\alpha,\delta}}
\newcommand{\stradh}{\strs_{\alpha,\delta,h}}

\newcommand{\buaL}{\bu_{\alpha}}

\newcommand{\buadorL}{\bu_{\alpha,\delta}^{(L)}}

\newcommand{\gbuadorL}{\grad\bu_{\alpha,\delta}^{(L)}}

\newcommand{\straL}{\strs_{\alpha}}

\newcommand{\stradorL}{\strs_{\alpha,\delta}^{(L)}}
\newcommand{\varrhoad}{\varrho_{\alpha,\delta}}
\newcommand{\varrhoadh}{\varrho_{\alpha,\delta,h}}
\newcommand{\buhd}{\bu_{\delta,h}}
\newcommand{\gbuhd}{\grad\bu_{\delta,h}}
\newcommand{\strhd}{\strs_{\delta,h}}

\newcommand{\buhdLa}{\bu_{\alpha,\delta,h}}
\newcommand{\gbuhdLa}{\grad\bu_{\alpha,\delta,h}}
\newcommand{\strhdLa}{\strs_{\alpha,\delta,h}}
\newcommand{\varrhohda}{\varrho_{\alpha,\delta,h}}
\newcommand{\buhLa}{\bu_{\alpha,h}}
\newcommand{\gbuhLa}{\grad\bu_{\alpha,h}}
\newcommand{\strhLa}{\strs_{\alpha,h}}

\newcommand{\dd}{\frac{d(d+1)}{2}}
\newcommand{\dt}{\Delta t}
\newcommand{\ddt}{{\rm d} t}
\newcommand{\ddx}{{\rm d} \xx}
\newcommand{\Wi}{{\text{Wi}}}
\renewcommand{\Re}{{\text{Re}}}
\renewcommand{\div}{\operatorname{div}}

\newcommand{\tr}{\operatorname{tr}}

\newcommand{\grad}{\boldsymbol{\nabla}}

\newcommand{\R}{\mathbb{R}}
\newcommand{\RS}{\mathbb{R}^{d \times d}_{\rm S}}
\newcommand{\RSPD}{\mathbb{R}^{d \times d}_{{\rm S},>0}}
\newcommand{\RSPDb}{\mathbb{R}^{d \times d}_{{\rm S},>0,b}}

\newcommand{\PP}{\mathbb{P}}

\newcommand{\pd}[2]{\frac{\partial#1}{\partial#2}}
\newcommand{\deriv}[2]{\frac{{\rm d}#1}{{\rm d}#2}}

\newcommand{\intd}{\int_\D}
\def\Gd{G_\delta}

\def\GdL{{G_{\delta}^L}}
\def\GdorL{{G_{\delta}^{(L)}}}

\def\Hd{H_\delta}

\def\BL{\beta}
\def\Bd{\beta_\delta}
\def\BdL{{\beta_{\delta}^L}}
\def\Bdb{{\beta_{\delta}^b}}
\def\BdorL{{\beta_{\delta}^{(L)}}}

\def\BorL{{\beta^{(L)}}}
\def\Fd{F_{\delta}}
\def\Ad{A_{\delta}}
\def\kd{\kappa_{\delta}}

\def\FdorL{F_{\delta}^{(L)}}

\def\Fdh{F_{\delta,h}}
\def\Fh{F_{h}}
\def\bn{\boldsymbol n}
\def\U{\mathrm{W}}  
\def\Uz{\mathrm{V}}
\def\S{\mathrm{S}}

\newcommand{\SPD}{\S_{>0}}

\newcommand{\SPDb}{\S_{>0,b}}
\def\Uh{\mathrm{W}_h}  
\def\Vhzero{\mathrm{V}_{h}^0}
\def\Vhone{\mathrm{V}_{h}^1}
\def\Sh{\mathrm{S}_h}

\def\Shone{\mathrm{S}_h^1}
\def\ShonePD{\mathrm{S}_{h,{\rm PD}}^1}
\def\Qhone{\mathrm{Q}_h^1}

\begin{document}
\markboth{John W.\ Barrett and S\'ebastien Boyaval}{}

\title{Finite Element Approximation of the FENE-P Model
}
\author{John W. Barrett}
\address{Department of Mathematics, Imperial College London\\ London SW7 2AZ, UK\\
jwb@ic.ac.uk}

\author{S\'ebastien Boyaval}
\address{Laboratoire d'hydraulique Saint-Venant\\ Universit\'e Paris-Est
(Ecole des Ponts ParisTech) EDF R\&D\\ 6 quai Watier, 78401 Chatou Cedex, France
\\ sebastien.boyaval@saint-venant.enpc.fr}




\begin{abstract}
We extend our analysis on the Oldroyd-B  
model in Barrett and Boyaval \cite{barrett-boyaval-09} to
consider the finite element approximation  
of the FENE-P system of equations,
which models a dilute polymeric fluid,
in a bounded domain $\D \subset {\mathbb R}^d$,
$d=2$ or $3$,
subject to no flow boundary conditions.
Our schemes are based on approximating 
the pressure and the   
symmetric conformation tensor 
by either
(a) piecewise constants 
or (b) continuous piecewise linears.  
In case (a) the velocity field is approximated 
by continuous piecewise quadratics ($d=2$) or a reduced version, 
where the tangential component on each simplicial edge ($d=2$) 
or face ($d=3$) is linear.
In case (b) the velocity field is approximated 
by continuous piecewise quadratics or the mini-element. 
We show that both of these types of schemes, 
based on the backward Euler type time discretization, 
satisfy a {\it free energy} bound, which involves the logarithm of 
both the conformation tensor and a linear function of its trace,  
without any constraint on the time step. 
Furthermore, 
for
our approximation (b)
in the presence of an additional dissipative term in the stress equation,
the so-called FENE-P model with stress diffusion,
we show (subsequence) {\it convergence} in the case $d=2$, 
as the spatial and temporal discretization parameters
tend to zero, 
towards global-in-time weak solutions of this
FENE-P system.
Hence, we prove existence of global-in-time weak solutions 
to the FENE-P model with stress diffusion in two spatial dimensions.    
\end{abstract}

\maketitle


\bigskip


\section{Introduction}
\label{sec1}
\setcounter{equation}{0}
\subsection{The FENE-P model}

We consider the standard FENE-P model for a dilute polymeric fluid.
The fluid, confined to an open bounded domain $\D\subset\R^d$ ($d=2$ or $3$)
with a Lipschitz boundary $\partial \D$,
is governed by the following non-dimensionalized system
for a given $b \in \R_{>0}$:\\
\noindent
{\bf (P)} Find 
$\bu : (t,\xx)\in[0,T)\times\D \mapsto \bu(t,\xx)\in\R^d$,  
$p : (t,\xx)\in \D_T := (0,T)\times\D \mapsto p(t,\xx)\in\R$ 
and $\strs : (t,\xx)\in[0,T)\times\D \mapsto\strs(t,\xx)\in \RSPDb$ 
such that 
\begin{subequations}
\begin{alignat}{2}
\label{eq:fene-p-sigma}
\Re \brk{\pd{\bu}{t} + (\bu\cdot\grad)\bu} 
& =  - \grad p + (1-\e) \Delta\bu 
+ \frac{\e}{\Wi} \div\brk{ 
A(\strs)\,\strs} + \f 
\quad 
&&\mbox{on } \D_T, 
\\
\div\bu 
& = 0 
\qquad \qquad 
&&\mbox{on } \D_T, 
\label{eq:fene-p-sigma1}
\\
\pd{\strs}{t}+(\bu\cdot\grad)\strs
& = (\gbu)\strs + \strs(\gbu)^T - \frac{
A(\strs)\,\strs}{\Wi}
\qquad 
&&\mbox{on } \D_T, 
\label{eq:fene-p-sigma2}
\\
\bu(0,\xx) &= \bu^0(\xx) 
\label{eq:initial}
\qquad &&\forall \xx \in \D,\\
\strs(0,\xx) &= \strs^0(\xx) 
\qquad &&\forall \xx \in \D,
\label{eq:initial1}
\\
\label{eq:dirichlet}
\bu &= \bzero 
\qquad \qquad &&\text{on } (0,T) \times \partial\D;
\end{alignat}
\end{subequations}
where
\begin{align}
A(\bphi) := \brk{1-\frac{\tr(\bphi)}{b}}^{-1}\I-\bphi^{-1}
\qquad \forall \bphi \in \RSPDb:= \BRK{\bpsi\in\RSPD \,:\, \tr(\bpsi)<b}.
\label{Adef}
\end{align}
Here $\RS$ denotes the set of symmetric $\R^{d\times d}$ matrices, 
and $\RSPD$  
the set of symmetric 
positive  
definite 
$\R^{d\times d}$ matrices. 
In addition, $\I \in \RSPD$ is the identity, and $\tr(\cdot)$ 
denotes trace.
The unknowns in {\bf (P)} are 
the velocity of the fluid, $\bu$, 
the hydrostatic pressure, $p$,
and the symmetric conformation tensor 
of the polymer molecules, $\strs$. 
The latter is
linked to the symmetric 
polymeric extra-stress tensor $\str$  
through the relation
$\str = \frac\e\Wi\,A(\strs)\,\strs$.   
In addition,
$\f : (t,\xx)\in 
\D_T \mapsto \f(t,\xx)\in\R^d$ is the
given density of body forces acting on the fluid;
and the following given parameters are dimensionless:
the Reynolds number $\Re \in \R_{>0}$, the Weissenberg number $\Wi \in \R_{>0}$,
the elastic-to-viscous viscosity fraction $\e \in (0,1)$,
 and the FENE-P parameter $b>0$ 
(related to a maximal admissible extensibility of the polymer molecules within the fluid).
For the sake of simplicity,
we will limit ourselves to the no flow boundary conditions
(\ref{eq:dirichlet}).
Finally, we denote $\grad \bu (t,\xx) \in \R^{d \times d}$ 
the velocity gradient tensor field with 
$[\grad \bu]_{ij}=\frac{\partial \bu_i}{\partial \xx_j}$, 
and $({\rm div}\, \strs)(t,\xx) \in \R^d$ the vector field
with $[{\rm div}\, \strs]_i = \sum_{j=1}^d \frac{\partial \strs_{ij}}{\partial \xx_j}$.

For data $\f \equiv 0$, divergence free $\bu^0 \in [L^2(\D)]^d$, 
and $\strs^0$, which is symmetric positive definite for a.e.\ $\xx \in \D$,
satisfying $\ln(1-\frac{\tr(\strs^0)}{b}) \in L^1(\D)$,   
then the existence of a global-in-time weak solution 
$\bu \in L^\infty(0,T;[L^2(\D)]^d) \cap L^2(0,T,[H^1_0(\D)]^d)$,
$\strs \in L^\infty(0,T;[L^\infty(\D)]^{d \times d})$ and 
$\str \in L^2(0,T;[L^2(\D)]^{d \times d})$ 
to {\bf (P)}, (\ref{eq:fene-p-sigma}--f), was proved in Masmoudi \cite{Masmoudi}.

In this work, we consider finite element approximations of the FENE-P system {\bf (P)}
and the corresponding model with stress diffusion, {\bf (P$_\alpha$)},
which is obtained by adding the dissipative term $\alpha\,\Delta \strs$ 
for a given $\alpha \in \R_{>0}$
to the right-hand side of (\ref{eq:fene-p-sigma2})
with an additional no flux boundary condition for $\strs$ on $\partial \D$.
This paper extends the results in 
Barrett and Boyaval \cite{barrett-boyaval-09}, where  
finite element approximations of the
corresponding Oldroyd-B models, where $A(\strs)=\I-\strs^{-1}$, 
were introduced and analysed. 
In fact, the convergence proof of the finite element approximation
of the Oldroyd-B model with stress diffusion for $d=2$ in \cite{barrett-boyaval-09} provided 
the first existence proof of global-in-time weak solutions for this system.  
Note that $A(\strs)=\I-\strs^{-1}$
is the formal limit of (\ref{Adef}) for infinite extensibility; that is,  
$b\to\infty$.

The model {\bf (P$_\alpha$)} has  
been considered computationally in Sureshkumar and Beris \cite{sureshkumar-beris-95}.
We recall also  that El-Kareh and Leal \cite{EKL}
showed the existence of a weak solution to a modified stationary FENE-P system of equations,
which included stress diffusion, 
but there an additional regularization was also present in their modified system
and played an essential role in their proof.
We stress that the dissipative term $\alpha\,\Delta \strs$ in {\bf (P$_\alpha$)} 
is not a regularization,
but can be physically motivated through the centre-of-mass 
diffusion in the related microscopic-macroscopic polymer model, 
though with a positive $\alpha \ll 1$,
see Barrett and S\"{u}li \cite{BS}, \cite{BS-2011-density-arxiv},
Schieber \cite{SCHI} and Degond and Liu \cite{DegLiu}.

Barrett and S\"{u}li have 
introduced, and proved the existence of global-in-time weak solutions for $d=2$ and $3$ to, 
microscopic-macroscopic dumbbell models of dilute polymers
with center-of-mass diffusion in the corresponding Fokker--Planck equation for  
a finitely extensible nonlinear elastic (FENE) spring law or a Hookean-type spring law,
see \cite{BS2011-fene} and \cite{BS2010-hookean}.
Recently, Barrett and S\"{u}li \cite{BS-HOOKEANOB} have proved rigorously  
that the macroscopic Oldroyd-B model with stress diffusion
is the exact closure of the microscopic-macroscopic Hookean dumbbell model
with center-of-mass diffusion for $d=2$, when the existence 
of global-in-time weak solutions to both
models can be proved. 
In addition, Barrett and S\"{u}li \cite{BS2011-feafene}
have introduced and analysed a finite element approximation
for the FENE microscopic-macroscopic dumbbell model with 
center-of-mass diffusion.

From a physical viewpoint, the FENE-P model is more realistic than the Oldroyd-B model
because it accounts for the finite-extensibility of the polymer molecules in the fluid
through the non-dimensional parameter $b>0$.
From a mathematical viewpoint,
compared to the Oldroyd-B model where the nonlinear terms are only the material 
derivative terms
(like $(\gbu)\strs$),
the FENE-P model has an additional singular nonlinearity due to
the factor $\brk{1-\frac{\tr(\cdot)}{b}}^{-1}$ in 
the definition of $A(\cdot)$,
which necessitates a careful mathematical treatment.
Hence, this paper is not a trivial extension of \cite{barrett-boyaval-09}.
In fact, the latter additional nonlinearity is exactly what makes the FENE-P model closer
to the physics of polymers than the Oldroyd-B model,
and thus also to many other macroscopic models
based on different constitutive relations
that have been developed by physicists for polymers.
We note the FENE-P system is the approximate macroscopic closure of
the FENE microscopic-macroscopic dumbbell model,
whereas the Oldroyd-B system is the exact 
macroscopic closure of
the Hookean microscopic-macroscopic dumbbell model.
Hence, the microscopic-macroscopic dumbbell models corresponding to Oldroyd-B and FENE-P,
only the spring laws differ;
see e.g.\ Bird et al. \cite{BCAH}
and Renardy \cite{renardy-00} for a more complete review of the differences between 
the Oldroyd-B and the FENE-P models from the physical viewpoint,
and for other macroscopic models with more nonlinear effects than the Oldroyd-B model,
e.g.\ 
the Giesekus model and the Phan--Thien Tanner model.

Similarly to \cite{barrett-boyaval-09}, our analysis in the present paper exploits
the underlying free energy of the system, see Wapperom and Hulsen \cite{WapperomH98} and   
Hu and Leli\`{e}vre \cite{hu-lelievre-07}. In  particular,    
the finite element approximation of {\bf (P$_\alpha$)} has to be constructed extremely 
carefully to inherit this free energy structure, and requires the 
approximation of $\tr(\strs)$ as a new unknown.
It is definitely not our goal to review all the macroscopic models used in rheology,
although similar studies could probably be pursued for other macroscopic models endowed 
with a free energy.
We will point out the main differences with 
Barrett and Boyaval \cite{barrett-boyaval-09},
and we thus hope to sufficiently suggest how our technique 
could be adapted to any nonlinear model with a free energy.
We believe that our approach contributes to a better understanding of the numerical stability
of the models used in computational rheology,
where numerical instabilities 
sometimes termed ``High-Weissenberg Number Problems'', see HWNP in
Owens and Phillips \cite{OP},
still persist.
Indeed, as exposed in Boyaval et al.\ \cite{boyaval-lelievre-mangoubi-09},
our point is that to make progress in this area 
one should identify sufficiently general rules for the derivation of 
good discretizations of macroscopic models 
such that they retain the dissipative structure of weak solutions to the system,
at least in some benchmark flows.

The outline of this paper is as follows.
First, we end this section by introducing our notation and some auxiliary results.
In Section \ref{sec:Penergy} we review the formal free energy bound for the FENE-P 
system {\bf (P)}.
In Section \ref{sec:delta} 
we introduce our regularization $G_{\delta}$ of 
of $G \equiv\ln$, which appears in the definition of the free energy of the 
FENE-P system {\bf (P)}.
We then introduce a regularized problem {\bf (P$_\delta$)},
and show a formal free energy bound for it.
In Section \ref{sec:deltah}, on assuming that $\D$ is a polytope for ease of exposition,
we introduce our finite element approximation of {\bf (P$_\delta$)}, 
namely {\bf (P$^{\Delta t}_{\delta,h}$)},
based on approximating 
the pressure and the   
symmetric conformation tensor 
by piecewise constants;  
and
the velocity field
with continuous piecewise quadratics 
or a reduced version, 
where the tangential component on each simplicial edge ($d=2$) 
or face ($d=3$) is linear.
Using the Brouwer fixed point theorem, we prove existence of a 
solution to {\bf (P$^{\Delta t}_{\delta,h}$)}
and show that it satisfies 
a discrete regularized free energy bound for any choice of time step; 
see Theorem~\ref{dstabthm}.
We conclude by showing that, in the limit $\delta \to 0_+$, 
these solutions of 
{\bf (P$^{\Delta t}_{\delta,h})$}
converge to a solution of {\bf (P$^{\Delta t}_{h}$)} 
with the approximation of the conformation tensor being positive definite 
and having a trace strictly less than $b$.
Moreover, this solution of {\bf (P$^{\Delta t}_{h}$)} 
satisfies a discrete free energy bound;  see Theorem~\ref{dconthm}.
Next, in Section \ref{sec:Palpha} we introduce the FENE-P system with stress diffusion, 
{\bf (P$_\alpha$)}, where  
the dissipative term $\alpha\,\Delta \strs$  
has been added to the right-hand side of (\ref{eq:fene-p-sigma2}).
We then introduce the corresponding 
regularized version {\bf (P$_{\alpha,\delta}$)},
and show a formal free energy bound for it.
In Section \ref{sec:Palphah}
we introduce our finite element approximation of 
{\bf (P$_{\alpha,\delta}$)}, 
namely {\bf (P$^{\Delta t}_{\alpha,\delta,h}$)},
based on approximating the velocity field with continuous piecewise quadratics
or the mini element, and the
pressure, the symmetric conformation tensor 
and its trace by continuous piecewise linears.
Here we assume that $\D$ is a convex polytope
and that the finite element mesh
consists of quasi-uniform non-obtuse simplices.
Using the Brouwer fixed point theorem, we prove existence of a 
solution to {\bf (P$^{\Delta t}_{\alpha,\delta,h}$)}
and show that it satisfies a discrete regularized free energy bound 
for any choice of time step; see Theorem~\ref{dstabthmaorL}.
In Section~\ref{sec:convergence}
we prove, in the case $d=2$, (subsequence) convergence of the solutions 
of {\bf (P$_{\alpha,\delta,h}^{\Delta t}$)}, as the 
regularization parameter, $\delta$, and the
spatial, $h$, and temporal, $\Delta t$,
discretization parameters tend to zero,  
to global-in-time weak solutions of {\bf (P$_\alpha$)}; see Theorem~\ref{exPathm}.
This existence result for {\bf (P$_\alpha$)} is new to the literature.

\subsection{Notation and auxiliary results}

The absolute value and the negative part of a real number $s\in\R$ 
are denoted by $\Abs{s}:=\max\{s,-s\}$ and $[s]_{-}=\min\{s,0\}$, respectively.
We adopt the following notation for inner products 
\begin{subequations}
\begin{alignat}{2}
 \bv\cdot\bw &:=\sum_{i=1}^d \bv_i \bw_i
 \equiv \bv^T \bw = \bw^T \bv  \qquad &&\forall \bv,\bw \in {\mathbb R}^d,
 \label{ipvec}
 \\
 \bphi : \bpsi &:=\sum_{i=1}^d \sum_{j=1}^d \bphi_{ij} \bpsi_{ij}
 \equiv \tr\brk{\bphi^T \bpsi}=\tr\brk{\bpsi^T \bphi} \qquad 
 &&\forall \bphi,\bpsi \in \R^{d\times d},
 \label{ipmat}
 \\
 \grad \bphi :: \grad \bpsi &:=\sum_{i=1}^d \sum_{j=1}^d \grad \bphi_{ij} 
 \cdot \grad \bpsi_{ij}\qquad 
 &&\forall \bphi,\bpsi \in \R^{d\times d};
 \label{ipgrad}
\end{alignat}
\end{subequations}
where $\cdot^T$ and $\tr\brk{\cdot}$ denote transposition and trace, respectively.
The corresponding norms are
\begin{subequations}
\begin{alignat}{3}
 \Norm{\bv}&:= (\bv \cdot \bv)^{\frac{1}{2}},
 \qquad &&\Norm{\grad \bv}:= (\grad \bv : \grad \bv)^{\frac{1}{2}}
 \qquad &&\forall \bv \in \R^d;
 \label{normv}
 \\
 \Norm{\bphi} &:= (\bphi : \bphi)^{\frac{1}{2}}, 
 \qquad &&\Norm{\grad \bphi} := (\grad \bphi :: \grad \bphi)^{\frac{1}{2}} 
 \qquad &&\forall \bphi \in \R^{d\times d}.
 \label{normphi}
\end{alignat}
\end{subequations}
We will use on several occasions that $\tr(\bphi)=\tr(\bphi^T)$ and
$\tr(\bphi\bpsi)=\tr(\bpsi\bphi)$ for all $\bphi, \bpsi \in \R^{d \times d}$, and
\begin{subequations}
\begin{alignat}{2} \label{eq:symmetric-tr}
 \bphi\bchi^T:\bpsi 
 &= \bchi\bphi:\bpsi = \bchi : \bpsi \bphi 
 \qquad &&\forall \bphi,\bpsi\in \RS, 
 \ \bchi\in \R^{d\times d}
 \,,\\
 \Norm{\bpsi\bphi} &\leq \Norm{\bpsi} \Norm{\bphi}
\qquad &&\forall \bphi,\bpsi\in\R^{d\times d}, 
\label{normprod}\\
\Norm{\bphi\,\bv} &\leq \Norm{\bphi} \Norm{\bv}
\qquad &&\forall \bphi \in\R^{d\times d},\,\bv \in \R^d. 
\label{vectprod}
\end{alignat}
\end{subequations}

For any $\bphi \in \RS$, there exists a  
decomposition
\beq \label{eq:diagonal-decomposition} 
\bphi= \mathbf{O}^T \mathbf{D} \mathbf{O} 
\qquad \Rightarrow \qquad \tr\brk{\bphi} = \tr\brk{\mathbf{D}},
\eeq
where $\mathbf{O}\in \R^{d \times d}$ is 
an orthogonal matrix and $\mathbf{D} \in \R^{d \times d}$ a diagonal matrix. 
Hence, for any $g:\R\to\R$, one can define $g(\bphi) \in \RS$ as
\beq \label{eq:tensor-g}
g(\bphi) := \mathbf{O}^T g({\bf D})\mathbf{O} 
\qquad \Rightarrow \qquad \tr\brk{g(\bphi)} = \tr\brk{g(\mathbf{D})},
\eeq
where $g({\bf D}) \in \RS$ is the diagonal matrix with entries $[g({\bf D})]_{ii} 
= g({\bf D}_{ii})$, $i=1, \ldots, d$.
Although the diagonal decomposition~\eqref{eq:diagonal-decomposition} is not unique,
\eqref{eq:tensor-g} uniquely defines $g(\bphi)$.
We note for later purposes that
\begin{align}
\label{modphisq}
d^{-1} (\tr(|\bphi|))^2 \leq \Norm{\bphi}^2  
\leq (\tr(|\bphi|))^2 
& \qquad \forall \bphi \in \RS.
\end{align}

One can show via diagonalization, see e.g.\ 
\cite{barrett-boyaval-09}
for details, that 
for all concave function $g \in C^1(\R)$, it holds
\begin{equation} \label{gconcave2}
\brk{\bphi-\bpsi}:g'(\bpsi) 
\ge \tr\brk{g(\bphi)-g(\bpsi)} \ge
\brk{\bphi-\bpsi}:g'(\bphi)
\qquad \forall \bphi,\bpsi \in \RS,
\end{equation}
where $g'$ denotes the first derivative of $g$.
If $g \in C^1(\R)$ is convex,
the inequalities in (\ref{gconcave2}) are reversed.
It follows from  
(\ref{gconcave2}) and (\ref{ipmat}) 
that for any $\bphi \in C^1([0,T];\RS)$
and any concave or convex $g \in C^1(\R)$
\begin{align}
\deriv{}{t}\tr\brk{g(\bphi)}
&=
\tr\brk{\deriv{{\bphi}}{t} g'({\bphi})} 
= 
\deriv{\bphi}{t}:g'(\bphi) 
\qquad \forall t \in [0,T].
\label{eq:deriv-tensor-g}
\end{align}
Of course, a similar result holds for spatial derivatives. 
Furthermore, the results (\ref{gconcave2}) and (\ref{eq:deriv-tensor-g}) hold true 
when $C^1(\R)$ and $\RS$ are replaced by $C^1(\R_{>0})$ and $\RSPD$ 
or $C^1(0,b)$ and $\RSPDb$.
Finally, one can show that 
if $g \in C^{0,1}(\mathbb R)$ with Lipschitz constant $g_{\rm Lip}$, then
\begin{align}\label{Lip}
\Norm{g(\bphi)-g(\bpsi)} \leq g_{\rm Lip} \Norm{\bphi-\bpsi}
\qquad \forall \bphi, \bpsi \in \RS.
\end{align}

We adopt the standard notation for Sobolev spaces, e.g.\ 
$H^1(\D) := \{ \eta:\D \mapsto \R \,:\, \intd [\,|\eta|^2 + \|\nabla \eta\|^2\,] \,\ddx< \infty\}$
with $H^1_0(\D)$ being the closure of $C^\infty_0(\D)$ for 
the corresponding norm $\|\cdot\|_{H^1(\D)}$.
We denote the associated semi-norm as $|\cdot|_{H^1(\D)}$.
The topological dual of the Hilbert space $H^1_0(\D)$, with pivot space  $L^2(\D)$,
will be denoted by $H^{-1}(\D)$.
Such function spaces are naturally extended when the range $\R$ is replaced
by $\R^d$, $\R^{d \times d}$ and $\RS$; 
e.g.\ $H^1(\D)$ becomes
$[H^1(\D)]^d$,  $[H^1(\D)]^{d \times d}$ and  
$[H^1(\D)]^{d \times d}_{\rm S}$ , 
respectively. 
For ease of notation, we write the corresponding norms 
and semi-norms as  
$\| \cdot\|_{H^1(\D)}$ and $| \cdot|_{H^1(\D)}$, respectively,
as opposed to e.g.\ $\|\cdot\|_{[H^1(\D)]^d}$
and $|\cdot|_{[H^1(\D)]^d}$, respectively. 
We denote the duality pairing between $H^{-1}(\D)$ and $H^1_0(\D)$
as $\langle \cdot,\cdot\rangle_{H^1_0(\D)}$,
and we similarly write $\langle \cdot,\cdot\rangle_{H^1_0(\D)}$
for the duality pairing between e.g.\  
$[H^{-1}(\D)]^d$ and $[H^1_0(\D)]^d$.
For notational convenience, we introduce also convex sets such as
$[H^1(\D)]^{d \times d}_{{\rm S},>0} := \{ \bphi \in [H^1(\D)]^{d \times d}_{\rm S} :
\bphi \in \RSPD \mbox{ a.e. in } \D \}$,
and $[H^1(\D)]^{d \times d}_{{\rm S},>0,b} := \{ \bphi \in [H^1(\D)]^{d \times d}_{\rm S} :
\bphi \in \RSPDb \mbox{ a.e. in } \D \}$.

In order to analyse {\bf (P)}, we adopt the notation
\begin{align}
\U  &:= [H^1_0(\D)]^d,
\quad {\rm Q}:= L^2(\D),
\quad  \Uz := \BRK{\bv\in\U \,:\,\intd q\,\div\bv\,\ddx =0 \quad \forall q\in{\rm Q}},
\label{spaces}
\\
\S &:= [L^\infty(\D)]^{d \times d}_{\rm S},
\quad
\SPD :=  [L^\infty(\D)]^{d \times d}_{{\rm S},>0} 
\quad 
\mbox{and} \quad
\SPDb := \BRK{\bphi\in\SPD \,:\, \tr(\bphi)<b \quad \mbox{a.e.\ in } \D}.
\nonumber 
\end{align}
Throughout the paper $C$ will denote a generic positive constant independent
of  
the regularization parameter $\delta$
and the mesh parameters $h$ and $\Delta t$.
Finally, we recall the Poincar\'e inequality
\beq \label{eq:poincare}
 \int_\D \|\bv\|^2 \,\ddx\leq C_P \int_\D \|\gbv\|^2 
 \,\ddx \qquad \forall \bv \in \U, 
\eeq
where $C_P \in \R_{>0}$ depends only on $\D$.

\section{Formal free energy bound for the problem {\bf (P)}}
\label{sec:Penergy} 
\setcounter{equation}{0}
In this section we recall from Hu and Leli\`{e}vre \cite{hu-lelievre-07} 
the free energy structure of problem $(\bf P)$.
Let $F(\bu,\strs)$ denote 
the free energy associated with a solution $(\bu,p,\strs)$ to problem $(\bf P)$, 
where we define
\begin{align} \label{eq:free-energy-P}
F(\bv,\bphi) &:= \frac{\Re}{2}\intd\|\bv\|^2\,\ddx
-\frac{\e}{2\Wi}\intd \Brk{ b \ln\brk{1-\frac{\tr(\bphi)}{b}} + \tr\brk{\ln(\bphi)+\I} }
\ddx
 \\
& \hspace{2.8in}
\qquad 
\forall (\bv,\bphi) \in [L^2(\D)]^d \times \S^\star
\nonumber
\end{align}
with $\S^\star \subset \SPDb$ such that $F(\cdot,\cdot)$ is well-defined.
Here
the first term $\frac{\Re}{2}\intd\|\bv\|^2$ corresponds to the usual kinetic energy term,
and the second term, which is nonnegative, 
is a relative entropy term. 
Moreover, on noting that $\ln$ is a concave function on $\R_{>0}$, we observe
\begin{equation} \label{eq:comparison-oldroyd}
F(\bv,\bphi) \ge \frac{\Re}{2}\intd\|\bv\|^2 \,\ddx
+ \frac{\e}{2\Wi}\intd \tr\brk{\bphi-\ln(\bphi)-\I} \,\ddx
\qquad \forall (\bv,\bphi)\in [L^2(\D)]^d \times \S^\star,
\end{equation}
where the right-hand side  
is the free energy of the Oldroyd-B model under the same no flow boundary conditions,
see e.g.\ 
\cite{hu-lelievre-07} and
\cite{barrett-boyaval-09}.
Clearly, diagonalization yields that 
the relative entropy term of this Oldroyd-B model is nonnegative. 
Of course, the $\I$ term in the relative entropy for FENE-P and Oldroyd-B plays no real role,
and just means that the minimum relative entropy for Oldroyd-B is zero 
and is obtained by $\bphi =\I$.
Finally, we note that 
$\tr\brk{\ln(\bphi)}$ is rewritten as $\ln\brk{\det(\bphi)}$
in  
\cite{hu-lelievre-07},
which once again 
is easily deduced
from diagonalization.

\begin{proposition} \label{prop:free-energy-P}  
With $\f \in L^2(0,T;$ $[H^{-1}(\D)]^d)$
let $(\bu,p,\strs)$  
be a sufficiently smooth solution to problem $(\bf P)$,
(\ref{eq:fene-p-sigma}--f),
such that $\strs(t,\cdot) \in \S^\star$ for $t \in (0,T)$.
Then the free energy $F(\bu,\strs)$ satisfies for a.a.\ $t \in (0,T)$
\begin{align} 
\displaystyle 
& 
\deriv{}{t}F(\bu,\strs) + (1-\e) \intd \|\gbu\|^2 \,\ddx
\label{eq:estimate-P} 
+ \frac{\e}{2 {\rm Wi}^2}\intd\tr\brk{
(A(\strs))^2 \strs}\,\ddx
=
\langle \f, \bu \rangle_{H^1_0(\D)},
\end{align}
where the 
third term on the left-hand side is positive, via diagonalization, on recalling (\ref{Adef}).
\end{proposition}
\begin{proof}
Multiplying the Navier-Stokes equation~\eqref{eq:fene-p-sigma} with $\bu$
and the stress equation~\eqref{eq:fene-p-sigma2}
with $\frac{\e}{2\Wi}\, A(\strs)$, 
summing and integrating over $\D$ yields, 
after using integrations by parts, the boundary condition (\ref{eq:dirichlet}) 
and the incompressibility property (\ref{eq:fene-p-sigma1}) in the standard way, that
\begin{align} 
&\intd  \Brk{ \frac{\Re}{2} \pd{\|\bu\|^2}{t} + (1-\e)\|\gbu\|^2 
+ \frac{\e}{\Wi}\brk{1-\frac{\tr(\strs)}{b}}^{-1}\strs :\gbu }\ddx
\label{Pener} \\ 
& \hspace{0.2in} + \frac{\e}{2\Wi} \intd \Bigg[ \brk{\pd{\strs}{t}+(\bu\cdot\grad)\strs}
 + \frac  
 {A(\strs)\,\strs}{\Wi}\Bigg] : A(\strs)\,\ddx
\nonumber
\\
&\hspace{0.2in} - \frac{\e}{2\Wi} \intd 
 \brk{ \brk{\gbu}\strs + \strs\brk{\gbu}^T }:
A(\strs)\,\ddx
= 
\langle \f, \bu\rangle_{H^1_0(\D)}.
\nonumber
\end{align}
It follows from the chain rule
and (\ref{eq:deriv-tensor-g}) that 
\begin{align} 
&\brk{\pd{\strs}{t}+(\bu\cdot\grad)\strs}:
A(\strs)
= \brk{\pd{}{t}+(\bu\cdot\grad)}\brk{-b\ln\brk{1-\frac{\tr(\strs)}{b}}-\tr\brk{\ln(\strs)}}.  
\label{Pener1}
\end{align}
On integrating (\ref{Pener1}) over $\D$, 
the $(\bu \cdot \grad)$ term on the right-hand side  
vanishes as $\bu(t,\cdot) \in \Uz$.
On noting (\ref{Adef}), (\ref{eq:symmetric-tr}), (\ref{ipmat})  
and (\ref{eq:fene-p-sigma1}), we obtain that
\begin{align} 
\brk{\strs\brk{\gbu}^T+\brk{\gbu}\strs}:
A(\strs) 
&= 2\brk{1-\frac{\tr(\strs)}{b}}^{-1}
\strs : \gbu.
\label{Pener2}
\end{align} 
Hence, on combining (\ref{Pener})--(\ref{Pener2}) and noting a trace property,
we obtain the desired free energy equality~\eqref{eq:estimate-P}.
\end{proof}

For later purposes, we note the following.
\begin{remark}\label{remsplit}
The step in the above proof
of testing 
\eqref{eq:fene-p-sigma2}
with $\frac{\e}{\mbox{\footnotesize $2$\rm \Wi}}
\,A(\strs)$
is equivalent to testing \eqref{eq:fene-p-sigma2}
with $-\frac{\e}{\mbox{\footnotesize $2$\rm \Wi}}\strs^{-1}$ 
and testing the corresponding trace equation  
\begin{align}
\pd{\tr(\strs)}{t}+(\bu\cdot\grad)\tr(\strs)
& = 2\gbu:\strs - \frac{\tr\brk{A(\strs)\,\strs}}{\mbox{\rm \Wi}}
\qquad 
\mbox{on } \D_T 
\label{eq:fene-p-sigma2tr}
\end{align}
with $\frac{\e}{\mbox{\footnotesize $2$\rm Wi}}\brk{1-\frac{\tr(\strs)}{b}}^{-1}$, and adding.
\end{remark}

Recall that in the limit $b\to\infty$
the FENE-P model formally converges to the Oldroyd-B model.
It is thus interesting to note that when $b\to\infty$,
the free energy equality~\eqref{eq:estimate-P} formally converges to
the corresponding free energy equality for the Oldroyd-B model, on recalling 
(\ref{eq:comparison-oldroyd}).
Finally, we note the following result.

\begin{corollary}
\label{cor:free-energy-P} 
Under the assumptions of Proposition \ref{prop:free-energy-P} it follows that
\begin{align} 
&\mathop{\sup \,F(\bu(t,\cdot),\strs(t,\cdot))}_{t \in (0,T)\hspace{2.5cm}}
+ \frac{1-\e}{2}\mathop{\int}_{\D_T} 
\|\gbu\|^2 \,\ddx \,\ddt
\label{eq:free-energy-P-bound}
+\frac{\e}{2{\rm Wi}^2}\mathop{\int}_{\D_T}
\tr\brk{ (A(\strs))^2 
\strs}\,\ddx\,\ddt
\\
& \hspace{2in}
\le 2\brk{ F(\bu^0,\strs^0) + \frac{1+C_P}{2(1-\e)} \Norm{\f}_{L^2(0,T;H^{-1}(\D))}^2 }.
\nonumber
\end{align}
\end{corollary}
\begin{proof}  
One can bound the term $\langle \f, \bu \rangle_{H^1_0(\D)}$ in \eqref{eq:estimate-P},
using the Cauchy-Schwarz and Young inequalities for $\nu\in\R_{>0}$,
and the Poincar\'e inequality (\ref{eq:poincare}),  
by
\begin{align} \label{fbound}
\langle \f, \bu \rangle_{H^1_0(\D)}
&\le \Norm{\f}_{H^{-1}(\D)} \Norm{\bu}_{H^1(\D)} 
\le \frac{1}{2\nu^2} \Norm{\f}_{H^{-1}(\D)}^2 + \frac{\nu^2}{2} \Norm{\bu}_{H^1(\D)}^2 
\\
& \le 
\frac{1}{2\nu^2} \Norm{\f}_{H^{-1}(\D)}^2 
+ \frac{\nu^2}{2} (1+C_P) \Norm{\gbu}_{L^2(\D)}^2. 
\nonumber
\end{align}
Combining (\ref{fbound}) and (\ref{eq:estimate-P}) with $\nu^2= (1-\e)/(1+C_P)$,
and integrating in time yields the result (\ref{eq:free-energy-P-bound}).
\end{proof}

\section{Formal free energy bound for a regularized problem {\bf (P$_\delta$)}}
\label{sec:delta}
\setcounter{equation}{0}
\subsection{A regularization}

Let $G:s \in \R_{>0} \mapsto \ln s\in\R$ denote the logarithm function,
whose domain of definition can be 
straightforwardly extended to the set of symmetric positive definite 
matrices using (\ref{eq:diagonal-decomposition}) and (\ref{eq:tensor-g}).
We define the following concave 
$C^{1,1}(\R)$ regularization of $G$ based on a given parameter $\delta \in (0,1)$: 
\begin{align}\label{eq:Gd}
&\Gd : s \in \R \mapsto 
\begin{cases}
G(s) & \forall s \ge \delta, 
\\
\frac{s}{\delta}+G(\delta)-1  & \forall s \leq \delta
\end{cases}
\qquad \Rightarrow \qquad \Gd(s) \geq G(s) \quad \forall s \in \R_{>0}.
\end{align}
We define also the following scalar functions 
\begin{align}\label{eq:BGd}
\Bd(s):=\brk{\Gd'(s)}^{-1}
\qquad \forall s\in \R 
\qquad \mbox{and}
\qquad
\beta(s):=\brk{G'(s)}^{-1} 
\qquad
\forall s\in \R_{>0}.
\end{align}
Hence, we have that
\begin{align} \label{eq:Bd}
\Bd : s \in \R \mapsto 
\max\{s,\delta\}
\qquad \mbox{and} \qquad
\beta : s \in \R_{>0} \mapsto s. 
\end{align}
For later purposes, we note the following results concerning these functions.
\begin{lemma}\label{GLemma}
For any $\bphi, \bpsi \in \RS$, $\eta \in \R$ and
for any $\delta\in(0,1)$
we have that
\begin{subequations}
\begin{align}
\label{eq:inverse-Gd}
\Bd(\bphi)\Gd'(\bphi)=\Gd'(\bphi)\Bd(\bphi) & =\I, 
\\
\label{eq:positive-term}
\tr
\brk{
  \brk{
  \eta\I-\Gd'(\bphi)}^2
  \Bd(\bphi)
} 
&\ge 0, 
\\
\label{Entropy1}
\tr\brk{\bphi-\Gd(\bphi)-\I} & \ge 0,
\\
\label{eq:concavity}
\brk{\bphi-\bpsi}:\Brk{\Gd'(\bpsi)}&\ge\tr\brk{\Gd(\bphi)-\Gd(\bpsi)}, 
\\
\label{matrix:Lipschitz}
- \brk{\bphi-\bpsi}:\Brk{\Gd'(\bphi)-\Gd'(\bpsi)}
& \ge
 \delta^2 \Norm{\Gd'(\bphi)-\Gd'(\bpsi)}^2.
\end{align}
\end{subequations}
In addition, if $\delta \in (0,\frac{1}{2}]$ 
we have that 
\begin{align}
\label{Entropy2}
\tr\brk{\bphi-\Gd(\bphi)}
&\ge \left\{\begin{array}{ll}
\frac{1}{2} \|\bphi\|  
\\[2mm]
\frac{1}{2\delta} \|[\bphi]_{-}\| 
\end{array}\right. 
\quad \mbox{\rm and}
\quad 
\bphi:\brk{\I-\Gd'(\bphi)}
\ge 
\textstyle \frac{1}{2}
\|\bphi\|
-d.
\end{align}
\end{lemma}
\begin{proof}
All the results are proved in Lemma 2.1 in \cite{barrett-boyaval-09},
except (\ref{eq:positive-term}) and this can be easily proved via 
diagonalization.
\end{proof}

We introduce the following regularization of $A$, (\ref{Adef}),
for any $\delta \in (0,\frac{1}{2}]$:
\begin{align}
A_\delta(\bphi,\eta):= \Gd'\brk{1-\frac{\eta}{b}}\I - \Gd'(\bphi) \qquad \forall (\bphi,\eta) 
\in \RS \times \R.
\label{Addef}
\end{align}
In addition to Lemma~\ref{GLemma}  
we will also make use of the following result, which is similar to (\ref{Entropy2}).
\begin{lemma}\label{GLemma2}
For any ${s \in \mathbb R}$, $b \in \R_{>0}$ and $\delta\in(0,\frac{1}{2}]$, we have that
\begin{subequations}
\begin{align}
-b \,\Gd\left(1-\frac{s}{b}\right) - s &\ge \frac{1}{2}\left[|s| -3b\right]_+,
\label{Gdbbelow}\\
\label{eq:FENEPstrs-above-OBstrs}
\brk{\Gd'\brk{1-\frac{s}{b}}-1}s &\ge \left[\, |s| - b\,\right]_+.
\end{align}
\end{subequations}
\end{lemma}
\begin{proof}
On recalling (\ref{eq:Gd}), we first note from the concavity of $\Gd$ that
\beq
\label{eq:log-concave-scalar-d}
-b\,\Gd\brk{1-\frac{s}{b}} \ge 
-b\,\Gd\brk{1} + s\,\Gd'\brk{1} 
= s 
\qquad \forall s\in\R.
\eeq
From the scalar version of (\ref{Entropy2}), we have that
\begin{align}
\left(1-\frac{s}{b}\right) - \Gd\left(1-\frac{s}{b}\right)
\ge \frac{1}{2}\left|1-\frac{s}{b}\right|
\qquad \Rightarrow \qquad 
-b\, \Gd\left(1-\frac{s}{b}\right) -s 
\ge \frac{b}{2}\left|1-\frac{s}{b}\right|-b.
\label{Gdineqa}
\end{align}
We note that
\begin{align}
\frac{b}{2}\left|1-\frac{s}{b}\right|-b = \left\{ \begin{array}{ll}
\frac{1}{2}(|s|-3b) & \qquad \mbox{if }s \ge b, \\[2mm]
-\frac{1}{2}(s+b) \ge \frac{1}{2}(|s|-3b)& \qquad \mbox{if }s \le b.
\end{array}\right.
\label{Gdineqb}
\end{align}
Combining (\ref{eq:log-concave-scalar-d})--(\ref{Gdineqb}) yields the desired result 
(\ref{Gdbbelow}).

We now consider (\ref{eq:FENEPstrs-above-OBstrs}).
If $1-\frac{s}{b}\le \delta$, i.e.\ $s \geq b(1-\delta)$, then
\begin{align}
\left(\Gd'\brk{1-\frac{s}{b}}-1\right) s
=\left(\frac{1}{\delta}-1\right)s \ge |s|.
\label{rhoeq1}
\end{align}
If $1-\frac{s}{b}\ge \delta$, i.e.\ $s \leq b(1-\delta)$, then
\begin{align}
\left(\Gd'\brk{1-\frac{s}{b}}-1\right) s=
\frac{s^2}{b-s} \ge 0
\qquad \mbox{and} \qquad 
\frac{s^2}{b-s} \ge |s|-b.
\label{rhoeq2}
\end{align}
Combining (\ref{rhoeq1}) and (\ref{rhoeq2}) yields the desired result
(\ref{eq:FENEPstrs-above-OBstrs}).
\end{proof}

\subsection{The regularized problem (P$_\delta$)}

Using the regularizations $\Gd,\,\Bd$ and $\Ad$ introduced above,
we consider the following regularization of {\bf (P)} 
for a given $\delta \in (0,\frac{1}{2}]$:

\noindent
{\bf (P$_{\delta}$)} Find 
$\bud : (t,\xx)\in[0,T)\times\D\mapsto\bud(t,\xx)\in\R^d$,  
$\prd : (t,\xx)\in \D_T \mapsto \prd(t,\xx)\in\R$ 
and $\strd : (t,\xx)\in[0,T)\times\D\mapsto\strd(t,\xx)\in \RS$ 
such that
\begin{subequations}
\begin{alignat}{2}
\label{eq:fene-p-sigmadL}
\Re \brk{\pd{\bud}{t} + (\bud\cdot\grad)\bud} 
& = - \grad \prd 
+ (1-\e)\Delta\bud 
+  \frac{\e}{\Wi}\div\brk{
\Ad(\strd,\tr(\strd))\,
\Bd(\strd)} + \f
\quad \;&&\mbox{on }\D_T,
\\
\label{eq:fene-p-sigma1dL}
\div\bud & = 0
\hspace{0.5in}
&&\mbox{on }\D_T,
\\
\label{eq:fene-p-sigma2dL}
\pd{\strd}{t}+(\bud\cdot\grad)\strd & =
(\gbud)\Bd(\strd)+\Bd(\strd)(\gbud)^T
-\frac
{
\Ad(\strd,\tr(\strd))\,
\Bd(\strd)}
{\Wi}
\quad
&&\mbox{on }\D_T,
\\
\label{eq:initial1dL}
\bud(0,\xx) &= \bu^0(\xx) 
\hspace{0.28in}
&&\hspace{-0.7in}\forall\xx\in\D,\\
\strd(0,\xx) &= \strs^0(\xx)
\hspace{0.28in} 
&&\hspace{-0.7in}\forall\xx\in\D,
\\
\label{eq:dirichletdL}
\bud&=\bzero
\qquad \qquad 
\:
&&\hspace{-0.7in}\text{on}\:(0,T)\times\partial\D.
\end{alignat}
\end{subequations}

\subsection{Formal free energy bound for (P$_\delta$)}

In this section, we extend the \textit{formal} energy results (\ref{eq:estimate-P})
and (\ref{eq:free-energy-P-bound})
for {\bf (P)} 
to problem {\bf (P$_{\delta}$)}. 
We will assume throughout that 
\begin{align} \label{freg}
&\f \in L^2\brk{0,T;[H^{-1}(\D)]^d}, \;\; 
\bu^0 \in 
{\rm H} 
:= \{ \bw \in [L^2(\D)]^d : {\rm div} \,\bw =0\; a.e.\mbox{ in } \D, 
\ \bw \cdot \bn_{\partial \D}=0
\mbox{ on } \partial \D\}, 
\\
&\strs^0 \in \SPD 
\quad \mbox{with}
\quad
\sigma_{\rm min}^0\, \|\bxi\|^2 \leq {\bxi}^T \strs^0(\xx) \,{\bxi} 
\leq \sigma_{\rm max}^0\, \|\bxi\|^2
\quad \forall \bxi \in {\mathbb R}^d
\quad \mbox{for } {a.e.} \ \xx  \mbox{ in }\D
\nonumber \\
&\mbox{and} \quad \tr\brk{\strs^0(\xx)}\leq b^\star
\quad \mbox{for } {a.e.} \ \xx  \mbox{ in }\D;
\nonumber 
\end{align}
where $\bn_{\partial \D}$ is normal to $\partial \D$,
$b^\star,\,\sigma_{\rm min}^0,\,\sigma_{\rm max}^0 \in \R_{>0}$ with $b^\star < b$.
Let $\Fd(\bud,\strd,\tr(\strd))$ 
denote the free energy associated with a solution $(\bud,\prd,\strd)$ 
to problem $\bf (P_{\delta})$, where  
we define
\begin{align} \label{eq:free-energy-Pd}
\Fd(\bv,\bphi,\eta) &:= \frac{\Re}{2}\intd\|\bv\|^2 \,\ddx 
- \frac{\e}{2\Wi}\intd
\Brk{ 
b\: \Gd\brk{1-\frac{\eta}{b}}
+
\tr\brk{
\Gd(\bphi)+\I
}
}\ddx
\\
& \hspace{2.2in}
\qquad  
\forall (\bv,\bphi,\eta) \in [L^2(\D)]^d \times \S \times L^1(\D).
\nonumber
\end{align}
Note that the second term in $\Fd$ has been regularized
in comparison with $F$ in~\eqref{eq:free-energy-P}.
Similarly to~\eqref{eq:comparison-oldroyd}, we have, on noting \eqref{eq:log-concave-scalar-d},
the inequality
\beq 
\label{eq:comparison-oldroyd-d}
\Fd(\bv,\bphi,\tr(\bphi)) \ge \frac{\Re}{2}\intd\|\bv\|^2 \,\ddx 
+ \frac{\e}{2\Wi}\intd\tr\brk{\bphi-\Gd(\bphi)-\I} 
\,\ddx
\qquad \forall (\bv,\bphi) \in 
[L^2(\D)]^d\times\S,
\eeq
where the right-hand side in~\eqref{eq:comparison-oldroyd-d} is the 
free energy of the corresponding regularized Oldroyd-B model,
see~\cite{barrett-boyaval-09} and note (\ref{Entropy1}).
It also follows from (\ref{eq:Gd}) and (\ref{freg}) that 
\begin{align}
\Fd(\bu^0,\strs^0,\tr(\strs^0)) \leq F(\bu^0,\strs^0).
\label{Fdin}
\end{align}

\begin{proposition} \label{prop:free-energy-Pd}
Let $\delta \in (0,\frac{1}{2}]$ and $(\bud,\prd,\strd)$ 
be a sufficiently smooth 
solution to problem $(\bf P_{\delta})$, (\ref{eq:fene-p-sigmadL}--f).
Then the free energy $\Fd(\bud,\strd,\tr(\strd))$  
satisfies for a.a.\ $t \in (0,T)$
\begin{align}
\label{eq:estimate-Pd}
&\deriv{}{t}\Fd(\bud,\strd,\tr(\strd)) +(1-\e)\intd \|\gbud\|^2 \,\ddx
\\
 \nonumber
& \hspace{1.5in}
+ \frac{\e}{2{\rm Wi}^2} \intd \tr
\brk{ 
  \brk{\Ad(\strd,\tr(\strd))
  }^2
  \Bd(\strd)}\ddx
= \langle \f, \bud \rangle_{H^1_0(\D)},
\end{align}
where the third term on the left-hand side is nonnegative from 
(\ref{Addef}) and (\ref{eq:positive-term}).
\end{proposition}

\begin{proof} 
Similarly to the proof of Proposition \ref{prop:free-energy-P}, we 
multiply the regularized Navier-Stokes equation (\ref{eq:fene-p-sigmadL})
by $\bud$ and the regularized stress equation (\ref{eq:fene-p-sigma2dL}) with 
$ \frac{\e}{2\Wi} \Ad(\strd,\tr(\strd))
$,
sum and integrate over $\D$,
use integrations by parts,
the boundary condition (\ref{eq:dirichletdL}) and
the incompressibility property (\ref{eq:fene-p-sigma1dL}). 
This yields the desired result (\ref{eq:estimate-Pd}) 
on noting the following analogues of (\ref{Pener1}) and (\ref{Pener2}) 
\begin{align}\label{Pdener1}
&\brk{\pd{\strd}{t}+(\bud\cdot\grad)\strd}: 
\Ad(\strd,\tr(\strd))
\\
&\hspace{1.5in}
=
\nonumber
\brk{\pd{}{t}+(\bud\cdot\grad)}\brk{-b\:\Gd\brk{1-\frac{\tr(\strd)}{b}}-\tr(\Gd(\strd))}
\end{align}
and
\begin{align} \label{Pdener2}
& \brk{\Bd\brk{\strd}\brk{\gbud}^T+\brk{\gbud}\Bd(\strd)}
:\Ad(\strd,\tr(\strd))
= 2\,\Gd'\brk{1-\frac{\tr(\strd)}{b}} \Bd(\strd) : \gbud.
\end{align}
Here we have recalled (\ref{Addef}) and (\ref{eq:deriv-tensor-g}) for (\ref{Pdener1}), 
and (\ref{eq:symmetric-tr}), (\ref{ipmat}), (\ref{eq:inverse-Gd}) and
(\ref{eq:fene-p-sigma1dL}) for (\ref{Pdener2}). 
\end{proof}

Similarly to Remark \ref{remsplit}, we note the following.  
\begin{remark}\label{remsplitd}
The step in the above proof
of testing the regularized stress equation \eqref{eq:fene-p-sigma2dL}
with $\frac{\e}{\mbox{ \footnotesize $2$\rm \Wi}}
\Ad(\strd,\tr(\strd))$
is equivalent to testing \eqref{eq:fene-p-sigma2dL}
with $-\frac{\e}{\mbox{\footnotesize $2$\rm \Wi}}\Gd'(\strd)$ 
and testing the corresponding 
regularized trace equation  
\begin{align}
\pd{\tr(\strd)}{t}+(\bud\cdot\grad)\tr(\strd)
& = 2\gbud:\Bd(\strd) - \frac{\tr(\Ad(\strd,\tr(\strd))\,\Bd(\strd))}{\mbox{\rm \Wi}} 
\qquad 
\mbox{on } \D_T 
\label{eq:fene-p-sigma2dLtr}
\end{align}
with $\frac{\e}{\mbox{\footnotesize $2$\rm Wi}}\Gd'\brk{1-\frac{\tr(\strd)}{b}}$, and adding.
\end{remark}

\begin{corollary} 
\label{cor:free-energy-Pd}
Under the assumptions of Proposition \ref{prop:free-energy-Pd}
it follows that
\begin{align} 
\label{eq:free-energy-bound}
&
\sup_{t \in (0,T)}\Fd(\bud(t,\cdot),\strd(t,\cdot),\tr(\strd(t,\cdot)))
+ \frac{1-\e}{2}
\int_{\D_T} 
\|\gbud\|^2 \,\ddx\,\ddt
\\
\nonumber
& \hspace{.2in} 
+ \frac{\e}{2{\rm Wi}^2}
\int_{\D_T}
\tr\brk{
  \brk{\Ad(\strd,\tr(\strd))
  }^2
  \Bd(\strd)
 }\ddx\,\ddt
\\
\nonumber
& \hspace{2in}
\le 2\brk{ F(\bu^0,\strs^0) + \frac{1+C_P}{2(1-\e)} \Norm{\f}_{L^2(0,T;H^{-1}(\D))}^2 }.
\end{align}
\end{corollary}

\begin{proof}
The proof of~\eqref{eq:free-energy-bound}
follows from (\ref{eq:estimate-Pd})
in the same way as  
\eqref{eq:free-energy-P-bound}
follows from (\ref{eq:estimate-P}), and in addition noting (\ref{Fdin}).
\end{proof}

\section{Finite element approximation of {\bf (P$_\delta$)} and  {\bf (P)}}
\label{sec:deltah}
\setcounter{equation}{0}
\subsection{Finite element discretization}
\label{FEd}

We now introduce a finite element discretization  
of the problem {\bf (P$_\delta$)},
which satisfies a discrete analogue of~\eqref{eq:estimate-Pd}.

The time interval $[0,T)$ is split into intervals $[t^{n-1},t^n)$ 
with $\dt_{n} = t^{n}-t^{n-1}$, $n=1, \ldots, N_T$. We set
$\dt:= \max_{n=1, \ldots, N_T} \dt_n$.
We will assume throughout that the domain $\D$ is a polytope.
We define a regular family of meshes $\{\mathcal{T}_h\}_{h>0}$ with 
discretization parameter $h>0$,
which is built from partitionings  
of the domain $\D$ into regular open simplices  
so that 
$$ 
\overline{\D} = \mathcal{T}_h := \mathop{\cup}_{k=1}^{N_K} \overline{K_k}\qquad
\mbox{with} \qquad \max_{k=1,\ldots,N_K} \frac{h_{k}}{\rho_{k}} \leq C. $$
Here $\rho_{k}$ is the diameter of the largest inscribed ball contained in 
the simplex $K_k$ and $h_{k}$ is the diameter of $K_k$, 
so that $h = \max_{k=1,\ldots,N_K} h_{k}$. 
For each element $K_k$, $k=1,\ldots, N_K$, 
of the mesh $\mathcal{T}_h$ let $\{P^k_i\}_{i=0}^d$ denotes its vertices,
and $\{\bn^k_i\}_{i=0}^d$ the outward unit normals of the edges $(d=2)$ or faces $(d=3)$
with $\bn^k_i$ being that of the edge/face opposite vertex $P^k_i$, $i=0,\ldots,d$.
In addition, let $\{\eta^k_i(\xx)\}_{i=0}^d$ denote the barycentric coordinates 
of $\xx \in K_k$ with respect to the vertices $\{P^k_i\}_{i=0}^d$;
that is, $\eta^k_i \in \PP_1$ and $\eta^k_i(P^k_j)=\delta_{ij}$,
$i,\,j =0 ,\ldots,d$.  
Here $\PP_m$ denote polynomials of maximal degree $m$ in $\xx$,
and $\delta_{ij}$ the Kronecker delta notation.
Finally, we introduce $\partial\mathcal{T}_h := \{E_j\}_{j=1}^{N_E}$ 
as the set of internal edges $E_j$ of triangles in the mesh $\mathcal{T}_h$ when $d=2$, 
or the set of internal faces $E_j$ of tetrahedra when $d=3$.

We approximate the problem {\bf (P$_\delta$)} by the problem {\bf (P$_{\delta,h}^{\dt})$}
based on the finite element spaces $\Uh^0 \times \mathrm{Q}_h^0 \times \S_h^0$.
As is standard, we require the discrete velocity-pressure spaces 
$\Uh^0\times\mathrm{Q}_h^0 \subset \U \times {\rm Q}$  
 satisfy the discrete Ladyshenskaya-Babu\v{s}ka-Brezzi (LBB) inf-sup condition
 \beq \label{eq:LBB}
 \inf_{q\in\mathrm{Q}_h^0} \sup_{\bv\in\Uh^0} 
 \frac{\displaystyle \int_{\D} q\,\div\bv\,\ddx }{\Norm{q}_{L^2(\D)}\Norm{\bv}_{H^1(\D)}} 
 \geq \mu_{\star} 
 > 0,
 \eeq 
 see e.g.\ \cite[p114]{GiraultRaviart}.
 In the following, we set
 \begin{subequations}
 \begin{align}
 \Uh^0&:=\Uh^2  \subset \U \mbox{ if $d=2$}\quad \mbox{or} \quad \Uh^{2,-}
 \subset \U \mbox{ if $d=2$ or 3}
 \label{Vh},\\
 \mathrm{Q}_h^0 &:=\{q \in {\rm Q} \,:\, q \mid_{K_k} \in \PP_0 \quad
 k=1,\ldots, N_K\} \subset {\rm Q}  
 \label{Qh}\\
\mbox{and} \qquad \S_h^0 &:=\{\bphi \in \S \,:\, \bphi \mid_{K_k} 
\in [\PP_0]^{d \times d}_{\rm S} \quad
 k=1,\ldots, N_K\} \subset \S; 
 \label{Sh}
 \end{align}
 \end{subequations}
 where 
 \begin{subequations}
 \begin{align}
 \Uh^2&:=\{\bv \in [C(\overline{\D})]^d\cap \U \,:\, \bv \mid_{K_k} \in [\PP_2]^d \quad
 k=1,\ldots, N_K\} \label{Vh2},\\
  \Uh^{2,-}&:=\{\bv \in [C(\overline{\D})]^d\cap \U \,:\, \bv \mid_{K_k} \in [\PP_1]^d 
  \oplus \mbox{span} \{\bvarsigma^{k}_i\}_{i=0}^d  \quad
 k=1,\ldots, N_K\}\label{Vh2m}.
 \end{align} 
\end{subequations} 
Here, for $k=1,\ldots, N_K$ and $i=0,\ldots,d$
\begin{align}
\bvarsigma^{k}_i(\xx) = \bn^{k}_i \prod_{j=0, j\neq i}^d 
\eta^{k}_j(\xx) \qquad \mbox{for } \xx \in K_k.
\label{bvarsig}
\end{align}
We introduce also  
\begin{align} \Vhzero := \BRK{\bv \in\Uh^0 \,:\, \int_{\D} q \, {\div\bv} \,\ddx= 0 \quad  
\forall q\in\mathrm{Q}_h^0}, 
\label{divh}
\end{align}
which approximates $\Uz$.
 It is well-known that the choices (\ref{Vh},b)   
 satisfy (\ref{eq:LBB}), see e.g.\ \cite[p221]{brezzi-fortin-92} 
 for $\Uh^0 =\Uh^2$ and $d=2$, and Chapter II, Sections 2.1 ($d=2$) and 2.3 
 ($d=3$) in \cite{GiraultRaviart}
 for $\Uh^0 =\Uh^{2,-}$. 
 Moreover, these particular choices of $\S_h^0$ and $\mathrm{Q}_h^0$
 have the desirable property   
 that
 \beq \label{eq:inclusion1}
 \bphi \in \Sh^0 
\ \Rightarrow \
\Ad(\bphi,\tr(\bphi))
\in \Sh^0 
\ \mbox{ and } \
b\:\Gd\brk{1-\frac{\tr(\bphi)}{b}}+\tr\brk{\Gd(\bphi)} \in \mathrm{Q}_h^0,
 \eeq
which makes it a straightforward matter to mimic the free
energy inequality (\ref{eq:estimate-Pd}) at a discrete level.    
Since  
$\Sh^0$ is discontinuous,
we will use the discontinuous Galerkin method to approximate the advection term 
$(\bud\cdot\nabla)\strd$ in the following. 
Then, for the boundary integrals, we will make use of the following definitions
(see e.g.\ 
\cite[p267]{ern-guermond-04}).
Given $\bv \in \Uh^0$,  
then for any $\bphi \in \Sh^0$ (or ${\rm Q}_h^0$) 
and for any point $\xx$ that is in the interior of some $E_j \in \partial\mathcal{T}_h$,
we define the downstream and upstream values of $\bphi$ at $\xx$ by
\begin{align}\label{eq:streams}
\bphi^{+\bv}(\xx) = \lim_{\rho \rightarrow 0^+} \bphi(\xx+\rho\,\bv(\xx)) \qquad
\mbox{and} \qquad 
\bphi^{-\bv}(\xx) = \lim_{\rho \rightarrow 0^-} \bphi(\xx+\rho\,\bv(\xx));
\end{align}
respectively. 
In addition, we denote by
\begin{align}
\label{eq:difsum}
\jump{\bphi}_{\to\bv}(\xx) = \bphi^{+\bv}(\xx) - \bphi^{-\bv}(\xx) 
\qquad \mbox{and}
\qquad  
\BRK{\bphi}^{\bv}(\xx) = \frac{\bphi^{+\bv}(\xx) + \bphi^{-\bv}(\xx)}{2},
\end{align}
the jump and mean value, respectively, of $\bphi$ 
at the point $\xx$ of boundary $E_j$.
From 
(\ref{eq:streams}), 
it is clear that the values of $\bphi^{+\bv}|_{E_j}$ and $\bphi^{-\bv}|_{E_j}$
can change along $E_j \in \partial\mathcal{T}_h$. 
Finally, it is easily deduced that
\begin{align}
\sum_{j=1}^{N_E} \int_{E_j} |\bv\cdot\bn| \jump{q_1}_{\to\bv} \,q_2^{+\bv}
\,{\rm d}{\bs}
&= -\sum_{k=1}^{N_K} \int_{\partial K_k} \brk{\bv \cdot \bn_{K_k}} q_1\,q_2^{+\bv}
\,{\rm d}{\bs}
\qquad \forall \bv \in \Uh^0, \ q_1,\,q_2 \in {\rm Q}_h^0;
\label{eq:jump}
\end{align}
where $\bn\equiv\bn(E_j)$ is a unit normal to $E_j$, 
whose sign is of no importance,
and $\bn_{K_k}$ is the outward unit normal vector 
of boundary $\partial K_k$ of $K_k$. 
We note that similar ideas appear in upwind schemes; 
e.g.\
see Chapter IV, Section 5 in \cite{GiraultRaviart}
for the Navier-Stokes equations.

\subsection{A free energy preserving approximation 
{\bf (P$^{\Delta t}_{\delta,h}$)} of {\bf (P$_\delta$)}}

For any source term $\f\in L^2\brk{0,T;[H^{-1}(\D)]^d}$, we define 
the following piecewise constant function with respect to the time variable
\beq
\label{eq:constant-source}
\f^{\Delta t,+}(t,\cdot)=\f^n(\cdot) 
:=\frac{1}{\dt_n} \int_{t^{n-1}}^{t^{n}} \f(t,\cdot)\, \ddt,
\qquad t\in [t^{n-1},t^n), \qquad n=1,\ldots,N_T.
\eeq
It is easily deduced that for $n=1,\ldots,N_T$ 
\begin{subequations}
\begin{align}
\label{fncont}
&\sum_{m=1}^{n} \dt_m\,\|\f^m\|_{H^{-1}(\D)}^r 
\leq \int_{0}^{t^n} \|\f(t,\cdot)\|_{H^{-1}(\D)}^r \,\ddt
\qquad \mbox{for any } r \in [1,2],
\\
\label{fnconv}
\mbox{and} \qquad
& \f^{\Delta t,+} \rightarrow \f \quad \mbox{strongly in }
L^2(0,T;[H^{-1}(\D)]^d) \mbox{ as } \dt \rightarrow 0_+.
\end{align}
\end{subequations}

Throughout this section we choose $\buh^0 \in \Vhzero$
to be the $L^2$ projection of $\bu^0$ onto $\Vhzero$ 
and $\strh^0 \in \Sh^0$ to be 
the $L^2$ projection of $\strs^0$ onto $\Sh^0$.
Hence, we have that
\begin{subequations}
\begin{align}
\|\buh^0\|_{L^2(\D)} \leq \|\bu^0\|_{L^2(\D)}, \qquad
\strh^0 \mid_{K_k} = \frac{1}{|K_k|} \int_{K_k} \strs^0 \,\ddx, \quad k=1,\ldots,N_K, 
\label{strh0def}
\end{align}
where $|K_k|$
is the measure of $K_k$;
and it immediately follows from (\ref{freg}) that
\begin{align}
\label{strh0pd}
\sigma_{\rm min}^0\, \|\bxi\|^2 \leq {\bxi}^T \strh^0\mid_{K_k} {\bxi} 
\leq \sigma_{\rm max}^0\, \|\bxi\|^2
\qquad \forall \bxi \in {\mathbb R}^d,
\\
\label{strhOtrb}
\tr(\strh^0)\mid_{K_k}  
= \frac{1}{|K_k|} \int_{K_k} \tr(\strs^0) \,\ddx
\leq \| \tr(\strs^0) \|_{L^\infty(K_k)} 
\leq b^\star < b.
\end{align}
\end{subequations} 
We are now ready to introduce our approximation 
{\bf (P$_{\delta,h}^{\dt}$)} of {\bf (P$_\delta$)} for $\delta \in (0,\frac{1}{2}]$:
 
\noindent
{\bf (P$_{\delta,h}^{\Delta t}$)}
Setting $(\buhd^0,\strhd^0)=(\buh^0,\strh^0) \in\Vhzero\times(\Sh^0\cap\SPDb)$
as defined in (\ref{strh0def}),
then for $n = 1, \ldots, N_T$ 
find $(\buhd^{n},\strhd^{n})\in\Vhzero\times\Sh^0$ such that for any test functions 
$(\bv,\bphi)\in\Vhzero\times\Sh^0$
\begin{subequations}
\begin{align} 
\label{eq:Pdh}
&\int_\D \Biggl[ \Re\left(\frac{\buhd^{n}-\buhd^{n-1}}{\dt_{n}}\right)\cdot \bv 
 + \frac{\Re}{2}\Brk{ \left( (\buhd^{n-1}\cdot\nabla)\buhd^{n}\right) \cdot \bv - 
 \buhd^{n} \cdot \left( (\buhd^{n-1}\cdot\nabla)\bv \right)}
\\
& \hspace{0.8in}
 + (1-\e) \gbuhd^{n}:\grad\bv + \frac{\e}{\Wi} \, 
\Ad(\strhd^{n}, \tr(\strhd^{n}))\,
 \Bd(\strhd^{n}) : \grad\bv
\Biggr] \ddx =
\langle \f^n, \bv\rangle_{H^1_0(\D)},
\nonumber 
\\ 
\label{eq:Pdhb}
&\int_\D \Brk{ \left(\frac{\strhd^{n}-\strhd^{n-1}}{\dt_{n}}\right) : \bphi 
 - 2\left( (\gbuhd^{n})\,\Bd(\strhd^{n})\right) :\bphi 
 + \frac{
 \Ad(\strhd^{n}, \tr(\strhd^{n}))\,
 \Bd(\strhd^{n}):\bphi}{\Wi}
}\ddx
\\
& \hspace{0.8in}
 + \sum_{j=1}^{N_E} \int_{E_j}  \Abs{\buhd^{n-1}\cdot\bn} 
 \jump{\strhd^{n}}_{\to\buhd^{n-1}}:\bphi^{+\buhd^{n-1}}\,{\rm d}\bs
= 0.
\nonumber
\end{align}
\end{subequations}

In deriving {\bf (P$_{\delta,h}^{\dt}$)}, we have noted~\eqref{eq:symmetric-tr} 
and that
\begin{align}
\int_\D \bv \cdot [(\bz \cdot\nabla)\bw] \,\ddx= - 
 \int_\D \bw \cdot [(\bz\cdot\nabla)\bv]\,\ddx 
\qquad 
\forall \bz \in \Uz, \quad \forall \bv, \bw \in [H^1(\D)]^d,
\label{conv0c}
\end{align}
and we refer to \cite[p267]{ern-guermond-04} and \cite{boyaval-lelievre-mangoubi-09}
for the consistency of our approximation of the stress advection term.
We note that on replacing $\Ad(\strhd^{n}, \tr(\strhd^{n}))$ with 
$\I-\Gd'(\strhd^{n})$ then {\bf (P$_{\delta,h}^{\dt}$)}, (\ref{eq:Pdh},b), collapses
to the corresponding finite element approximation of Oldroyd-B studied in 
\cite{barrett-boyaval-09}, see (3.12a,b) there. 

Before proving existence of a solution to {\bf (P$^{\Delta t}_{\delta,h}$)},
we first derive a discrete 
analogue of the energy bound~\eqref{eq:estimate-Pd} for
{\bf (P$^{\Delta t}_{\delta,h}$)}, 
which uses the elementary equality
\begin{align}
2s_1 (s_1-s_2)=s_1^2-s_2^2 +(s_1-s_2)^2 \qquad \forall s_1,s_2 \in \R.
\label{elemident}
\end{align}

\subsection{Energy bound for (P$^{\Delta t}_{\delta,h}$)}

\begin{proposition} \label{prop:free-energy-Pdh}
For $n= 1, \ldots, N_T$, a solution $\brk{\buhd^{n},\strhd^{n}}
\in \Vhzero\times\Sh^0$ to {\bf (P$_{\delta,h}^{\Delta t}$)}, (\ref{eq:Pdh},b), 
if it exists, satisfies
\begin{align} 
\label{eq:estimate-Pdh}
&
\frac{F_\delta(\buhd^{n},\strhd^{n},\tr(\strhd^{n}))
-F_\delta(\buhd^{n-1},\strhd^{n-1},\tr(\strhd^{n-1}))}{\dt_{n}} 
+ \frac{ {\rm Re}}{2\dt_{n}}\intd\|\buhd^{n}-\buhd^{n-1}\|^2
\,\ddx
\\
\nonumber 
& 
\ 
+(1-\e)\intd\|\gbuhd^{n}\|^2 \,\ddx
+ \frac{\e}{2{\rm Wi}^2} \intd 
\tr \brk{
  \brk{
   \Ad(\strhd^{n}, \tr(\strhd^{n}))}^2
  \Bd(\strhd^n)}\ddx
\\
\nonumber 
& \hspace{1in}
\le
\langle \f^n, \buhd^n \rangle_{H^1_0(\D)}
\le 
\frac{(1-\e)}{2}\intd\|\gbuhd^{n}\|^2\,\ddx
+ \frac{1+C_P}{2(1-\e)}\, \|\f^{n}\|_{H^{-1}(\D)}^2.
\end{align}
\end{proposition}
\begin{proof}
Similarly to (\ref{Pdener2}), we have that 
\begin{align} \label{Pdener2h}
& \intd \Brk{  
\brk{\gbuhd^n}\Bd(\strhd^n)
:
 \Ad(\strhd^{n}, \tr(\strhd^{n}))}\ddx
= \intd\Gd'\brk{1-\frac{\tr(\strhd^n)}{b}} \Bd(\strhd^n) : \gbuhd^n
\,\ddx,
\end{align}
where we have noted (\ref{Addef}), (\ref{ipmat}), (\ref{eq:inverse-Gd}) and 
(\ref{divh}).
Then, similarly to the proof of  Proposition~\ref{prop:free-energy-Pd},
we choose $\bv=\buhd^{n}\in\Vhzero$ in (\ref{eq:Pdh}) and 
$\bphi=\frac{\e}{2\Wi}
 \Ad(\strhd^{n}, \tr(\strhd^{n}))
\in\Sh^0$ 
in~(\ref{eq:Pdhb}) and obtain, on noting (\ref{elemident}), (\ref{eq:inverse-Gd}), 
(\ref{divh}) and (\ref{Pdener2h}),
that
\begin{align}
\label{eq:free-energy-Pdh-demo1}
\langle \f^n, \buhd^n \rangle_{H^1_0(\D)} 
&\ge
\intd
\Brk{ \frac{\Re}{2} \brk{ \frac{\|\buhd^{n}\|^2-\|\buhd^{n-1}\|^2}{\dt_{n}} 
+ \frac{\|\buhd^{n}-\buhd^{n-1}\|^2}{\dt_{n}} } + (1-\e)\|\gbuhd^{n}\|^2 }
\ddx
\\
&\hspace{0.25in} + \frac{\e}{2\Wi}  \intd 
\brk{ \frac{\strhd^{n}-\strhd^{n-1}}{\dt_{n}} }:
\Ad(\strhd^{n}, \tr(\strhd^{n}))
\,\ddx
\nonumber \\
& \hspace{0.25in}
+ \frac{\e}{2\Wi^2}  \intd 
\tr \brk{
  \brk{\Ad(\strhd^{n}, \tr(\strhd^{n}))
  }^2
  \Bd(\strhd^n)}
  \ddx 
\nonumber
\\
&\hspace{0.25in} 
 +  \frac{\e}{2\Wi}\sum_{j=1}^{N_E} \int_{E_j} 
     \Brk{ \Abs{\buhd^{n-1}\cdot\bn} \jump{\strhd^{n}}_{\to\buhd^{n-1}}:
     \brk{
     \Ad(\strhd^{n}, \tr(\strhd^{n}))
     }^{+\buhd^{n-1}} } \,{\rm d}\bs.
\nonumber
\end{align}
It follows from (\ref{Addef}), (\ref{eq:concavity}) and the concavity of $\Gd$ that
\begin{align}
\label{concaveGd}
&\brk{\strhd^{n}-\strhd^{n-1}}: 
\Ad(\strhd^{n}, \tr(\strhd^{n}))
\\
& \qquad \ge \brk{\tr(\strhd^{n})-\tr(\strhd^{n-1})}\Gd'\brk{1-\frac{\tr(\strhd^n)}{b}}
+\tr(\Gd(\strhd^{n-1}) - \tr(\Gd(\strhd^n))
\nonumber \\
& \qquad \ge
\brk{b\,\Gd\brk{1-\frac{\tr(\strhd^{n-1})}{b}}
+
\tr\brk{
\Gd(\strhd^{n-1})
}}-
\brk{b\,\Gd\brk{1-\frac{\tr(\strhd^{n})}{b}}
+
\tr\brk{
\Gd(\strhd^{n})
}}.
\nonumber
\end{align}
Similarly to (\ref{concaveGd}), we have, on recalling (\ref{Addef}), (\ref{eq:streams}) and
(\ref{eq:difsum}), that
\begin{align}
\label{concaveGdE}
&\jump{\strhd^{n}}_{\to\buhd^{n-1}}:
\brk{  
\Ad(\strhd^{n}, \tr(\strhd^{n}))}^{+\buhd^{n-1}} 
\ge
-\jump{b\,\Gd\brk{1-\frac{\tr(\strhd^{n})}{b}}
+
\tr\brk{
\Gd(\strhd^{n})
}}_{\to\buhd^{n-1}}.
\end{align}
Finally, we note from (\ref{eq:jump})  
and as $\buhd^{n-1} \in \Vhzero$ that for all $q \in {\rm Q}_h^0$
\begin{align} 
\sum_{j=1}^{N_E} \int_{E_j} \Abs{\buhd^{n-1}\cdot\bn} 
\jump{q  }_{\to\buhd^{n-1}} \,{\rm d}\bs
&= 
 -\sum_{k=1}^{N_K} \int_{\partial K_k} \brk{\buhd^{n-1}\cdot\bn_{K_k}} 
q\,{\rm d}\bs
= 
-\sum_{k=1}^{N_K} \int_{K_k} q\, \div \buhd^{n-1}
\,\ddx  
= 0.
\label{edgeterm} 
\end{align}
Combining (\ref{eq:free-energy-Pdh-demo1})--(\ref{edgeterm})
yields the first desired inequality in (\ref{eq:estimate-Pdh}).
The second inequality in (\ref{eq:estimate-Pdh}) follows immediately from (\ref{fbound})
with $\nu^2 = (1-\e)/(1+C_P)$. 
\end{proof}

\subsection{Existence of a solution to (P$^{\Delta t}_{\delta,h}$)}

\begin{proposition}
\label{prop:existence-Pdh}
Let $\delta \in (0,\frac{1}{2}]$, then,
given $(\buhd^{n-1},\strhd^{n-1}) \in \Vhzero \times \Sh^0$
and for any time step $\dt_{n} > 0$,
there exists at least 
one solution $\brk{\buhd^{n},\strhd^{n}} \in \Vhzero\times\Sh^0$ to 
{\bf (P$_{\delta,h}^{\Delta t}$)}, (\ref{eq:Pdh},b).
\end{proposition}
\begin{proof}
We introduce the following inner product on 
the Hilbert space $\Vhzero\times\Sh^0$
\begin{align} 
 \Scal{ (\bw,\bpsi) }{ (\bv,\bphi) }_\D = 
\intd \Brk{ \bw\cdot\bv + \bpsi:\bphi } \,\ddx 
\qquad \forall (\bw,\bpsi),(\bv,\bphi) \in \Vhzero\times\Sh^0.
\label{IPD}
\end{align}
Given $(\buhd^{n-1},\strhd^{n-1})\in\Vhzero\times\Sh^0$, let
$\mathcal{F} : \Vhzero\times\Sh^0 \mapsto \Vhzero\times\Sh^0$ be such that for any
$(\bw,\bpsi) \in \Vhzero\times\Sh^0$
\begin{align} 
& \Scal{\mathcal{F}(\bw,\bpsi)}{(\bv,\bphi)}_\D 
\label{eq:mapping}
\\
&\qquad := \int_\D 
\Biggl[
  \Re \left(\frac{\bw-\buhd^{n-1}}{\dt_{n}}\right) \cdot \bv
  + \frac{\Re}{2} \left[\left((\buhd^{n-1}\cdot\nabla)\bw \right) \cdot \bv 
  - \bw \cdot \left((\buhd^{n-1}\cdot\nabla)\bv \right) \right]  
\nonumber
\\
& \qquad \qquad 
  + (1-\e) \grad\bw :\grad\bv 
  + \frac{\e}{\Wi}\,
\Ad(\bpsi,\tr(\bpsi))\,
  \Bd(\bpsi):\grad\bv 
\nonumber
\\ 
& \qquad \qquad 
  + \left(\frac{\bpsi-\strhd^{n-1}}{\dt_{n}}\right):\bphi 
 - 2\left((\grad\bw) \,\Bd(\bpsi)\right):\bphi 
 + \frac{
 \Ad(\bpsi,\tr(\bpsi))\,
  \Bd(\bpsi):\bphi}{\Wi}
\Biggr] \,\ddx
\nonumber
\\
& \qquad \qquad 
-\langle \f^n, \bv \rangle_{H^1_0(\D)}
+ \sum_{j=1}^{N_E} \int_{E_j}
\Abs{\buhd^{n-1}\cdot\bn} \jump{\bpsi}_{\to\buhd^{n-1}}:\bphi^{+\buhd^{n-1}} 
\,{\rm d}\bs
\qquad 
\forall (\bv,\bphi) \in \Vhzero \times \Sh^0.
\nonumber
\end{align}
We note that a solution $(\buhd^{n},\strhd^{n})$ to (\ref{eq:Pdh},b), 
if it exists, corresponds to a zero of $\mathcal{F}$; that is,
\beq \label{eq:solution}
\Scal{\mathcal{F}(\buhd^{n},\strhd^{n})}{(\bv,\bphi)}_\D = 0 \qquad 
\forall (\bv,\bphi) \in \Vhzero\times\Sh^0. 
\eeq
In addition, it is easily deduced that the mapping $\mathcal{F}$ is continuous.
For any $(\bw,\bpsi) \in \Vhzero\times\Sh^0$, on choosing
$(\bv,\bphi) = \brk{\bw,\frac{\e}{2\Wi}
\Ad(\bpsi,\tr(\bpsi))}$,
we obtain analogously to (\ref{eq:estimate-Pdh}) that
\begin{align} \label{eq:inequality1}
&\Scal{\mathcal{F}(\bw,\bpsi)}{\brk{\bw,\frac{\e}{2\Wi}
\Ad(\bpsi,\tr(\bpsi))}}_\D
\\
\nonumber
& \hspace{0.3in} 
\ge 
\frac{F_\delta(\bw,\bpsi,\tr(\bpsi))-F_\delta(\buhd^{n-1},
\strhd^{n-1},\tr(\strhd^{n-1}))}{\dt_{n}} 
+ \frac{\Re}{2\dt_{n}}\intd\|\bw-\buhd^{n-1}\|^2 \,\ddx
\\
& \hspace{0.4in}
+ \frac{1-\e}{2}
\intd\|\grad\bw\|^2 \,\ddx
+ \frac{\e}{2\Wi^2} \intd \tr\brk{
\brk{
\Ad(\bpsi,\tr(\bpsi))
}^2
\Bd(\bpsi)
}\ddx
-\frac{1+C_P}{2(1-\e)}\Norm{\f^{n}}_{H^{-1}(\D)}^2.
\nonumber
\end{align}

Let us now assume that for any $\gamma \in \R_{>0}$, 
the continuous mapping $\mathcal{F}$ has no zero $(\buhd^{n},\strhd^{n})$ 
satisfying~\eqref{eq:solution}, which lies in the ball
\begin{align}
\mathcal{B}_\gamma := \BRK{ (\bv,\bphi) \in \Vhzero\times\Sh^0 \,: \, 
\Norm{(\bv,\bphi)}_\D\le\gamma };
\label{Bgamma}
\end{align}
where 
\begin{align}
\Norm{(\bv,\bphi)}_D := 
\left[((\bv,\bphi),(\bv,\bphi))_D\right]^{\frac{1}{2}} =
\brk{\intd [\,\|\bv\|^2+\|\bphi\|^2\,]\,\ddx }^\frac12. 
\label{NormD}
\end{align}
Then for such $\gamma$, we can define the continuous mapping $\mathcal{G}_\gamma
: \mathcal{B}_\gamma \mapsto  \mathcal{B}_\gamma$
such that for all $(\bv,\bphi) \in \mathcal{B}_\gamma$
\begin{align}
\mathcal{G}_\gamma(\bv,\bphi) := -\gamma 
\frac{\mathcal{F}(\bv,\bphi)}{\Norm{\mathcal{F}(\bv,\bphi)}_\D}. 
\label{Ggamma}
\end{align}
By the Brouwer fixed point theorem, $\mathcal{G}_\gamma$ has at least 
one fixed point $(\bw_\gamma,\bpsi_\gamma)$ in $\mathcal{B}_\gamma$. 
Hence it satisfies
\begin{equation}\label{eq:fixed-point}
\Norm{(\bw_\gamma,\bpsi_\gamma)}_\D=
 \Norm{\mathcal{G}_\gamma(\bw_\gamma,\bpsi_\gamma)}_\D=\gamma.
\end{equation}

On noting (\ref{Sh}) and (\ref{eq:fixed-point}), we have that
\begin{equation}\label{eq:norm_equivalence} 
\|\bpsi_\gamma\|_{L^\infty(\D)}^2
\leq \frac{1}{\min_{k\in N_K} |K_k|} \int_\D \|\bpsi_\gamma\|^2 \,\ddx
= \mu_h^2 \int_\D \|\bpsi_\gamma\|^2  \,\ddx \leq \mu_h^2\,\gamma^2,
\end{equation}
where $\mu_h := [1/(\min_{k\in N_K} |K_k|)]^{\frac{1}{2}}$. 
Then 
(\ref{eq:free-energy-Pd}), (\ref{eq:comparison-oldroyd-d}),
(\ref{Entropy2}),
(\ref{eq:norm_equivalence})
and (\ref{eq:fixed-point}) yield that
\begin{align} 
\label{bb1} 
&F_\delta(\bw_\gamma,\bpsi_\gamma,\tr(\bpsi_\gamma)) 
\\
& \qquad \qquad =
\frac{\Re}{2} \intd \|\bw_\gamma\|^2 \,\ddx+ \frac{\e}{2\Wi}
\intd
\Brk{ -b\:\Gd\brk{1-\frac{\tr(\bpsi_\gamma)}{b}}-\tr\brk{\Gd(\bpsi_\gamma)+\I} }
\ddx
\nonumber
\\
& \qquad \qquad \ge
\frac{\Re}{2} \intd\|\bw_\gamma\|^2 \,\ddx 
+ \frac{\e}{2\Wi} \intd \tr\brk{\bpsi_\gamma-\Gd(\bpsi_\gamma)-\I}
\,\ddx
\nonumber
\\
& \qquad \qquad \ge
\frac{\Re}{2} \intd\|\bw_\gamma\|^2  \,\ddx + \frac{\e}{4\Wi}
\left[ \intd \|\bpsi_\gamma\| \,\ddx -2d|\D|\right]
\nonumber
\\
& \qquad \qquad \ge 
\frac{\Re}{2} \intd\|\bw_\gamma\|^2\,\ddx  + \frac{\e}{4\Wi\,\mu_h\gamma}
\|\bpsi_\gamma\|_{L^\infty(\D)}
\intd \|\bpsi_\gamma\| \,\ddx
-\frac{\e d |\D|}{2\Wi}
\nonumber
\\
& \qquad \qquad \ge
\min\brk{\frac{\Re}{2},\frac{\e}{4\Wi\,\mu_h\gamma}}
\brk{ \intd\left[\,\|\bw_\gamma\|^2+ \|\bpsi_\gamma\|^2\,\right] \,\ddx}
-\frac{\e d |\D|}{2\Wi}
\nonumber
\\
& \qquad \qquad = 
\min\brk{\frac{\Re}{2},\frac{\e}{4\Wi\,\mu_h\gamma}}
\gamma^2 
-\frac{\e d |\D|}{2\Wi}.
\nonumber
\end{align}
Hence for all $\gamma$ sufficiently large,
it follows from (\ref{eq:inequality1}), (\ref{bb1}) and (\ref{eq:positive-term}) that
\beq \label{eq:one-hand}
\Scal{\mathcal{F}(\bw_\gamma,\bpsi_\gamma)}{\brk{\bw_\gamma,\frac{\e}{2\Wi}
\Ad(\bpsi_\gamma,\tr(\bpsi_\gamma))
}}_\D
\ge 0.
\eeq

On the other hand as $(\bw_\gamma,\bpsi_\gamma)$ is a fixed point 
of ${\mathcal G}_\gamma$, we have that
\begin{align}
\label{eq:whereas}
&
\Scal{\mathcal{F}(\bw_\gamma,\bpsi_\gamma)}{\brk{\bw_\gamma,\frac{\e}{2\Wi}
\Ad(\bpsi_\gamma,\tr(\bpsi_\gamma))
}}_D
\\ 
\nonumber
& \hspace{1.4in} =
-\frac{\Norm{\mathcal{F}(\bw_\gamma,\bpsi_\gamma)}_\D}{\gamma} 
\intd \left[ \|\bw_\gamma\|^2+ \frac{\e}{2\Wi} 
\bpsi_\gamma :
\Ad(\bpsi_\gamma,\tr(\bpsi_\gamma))
\right] \,\ddx.
\end{align}
It follows from (\ref{Addef}), (\ref{eq:FENEPstrs-above-OBstrs}), (\ref{Entropy2}) 
and similarly to (\ref{bb1}),
on noting
(\ref{eq:norm_equivalence}) and (\ref{eq:fixed-point}), that
\begin{align}
\label{bb2}
&\intd \Big[
\|\bw_\gamma\|^2 
+
\frac{\e}{2\Wi} \bpsi_\gamma :
\Ad(\bpsi_\gamma,\tr(\bpsi_\gamma))
\Big] \ddx
\\
\nonumber
&\qquad
=
\intd \Big[
\|\bw_\gamma\|^2 
+ \frac{\e}{2\Wi} \brk{\Gd'\brk{1-\frac{\tr(\bpsi_\gamma)}{b}}-1}\tr(\bpsi_\gamma)
+ \frac{\e}{2\Wi} \bpsi_\gamma:\brk{\I-\Gd'(\bpsi_\gamma)}
\Big] \ddx
\\
\nonumber
&\qquad
\ge
\intd 
\|\bw_\gamma\|^2 \,\ddx
+ \frac{\e}{4\Wi} \left[ \intd \|\bpsi_\gamma\| \,\ddx- 2d|\D| \right]
\\
\nonumber
&\qquad
\ge \min\brk{1,\frac{\e}{4\Wi\,\mu_h \gamma}}\gamma^2 -
\frac{\e d|\D|}{2 \Wi}.
\end{align}
Therefore on combining (\ref{eq:whereas}) and (\ref{bb2}), we have for all 
$\gamma$ sufficiently large that 
\begin{align} \label{eq:other-hand}
\Scal{\mathcal{F}(\bw_\gamma,\bpsi_\gamma)}{\brk{\bw_\gamma,\frac{\e}{2\Wi}
\Ad(\bpsi_\gamma,\tr(\bpsi_\gamma))}}_D < 0,
\end{align}
which obviously contradicts~\eqref{eq:one-hand}.
Hence the mapping $\mathcal{F}$ has a zero in $\mathcal{B}_\gamma$
for $\gamma$ sufficiently large,
and so there exists a solution 
$(\buhd^{n},\strhd^{n})$ to (\ref{eq:Pdh},b).
\end{proof}

\begin{theorem}
\label{dstabthm}
For any $\delta \in (0, \frac{1}{2}]$,
$N_T \geq 1$ and 
any partitioning of $[0,T]$ into $N_T$ time steps,
there exists a solution
$\{(\buhd^{n},\strhd^{n})\}_{n=1}^{N_T}
\in [\Vhzero \times \Sh^0]^{N_T}
$ to {\bf (P$^{\dt}_{\delta,h}$)}, (\ref{eq:Pdh},b).
In addition, it follows for $n=1,\ldots,N_T$ that
\begin{align}
\label{Fstab1}
&
\Fd(\buhd^n,\strhd^n,\tr(\strhd^n)) + \frac{1}{2} \sum_{m=1}^n
\int_\D \left[ {\rm Re} \|\buhd^m-\buhd^{m-1}\|^2
+ (1-\e)\dt_{m}\|\gbuhd^m\|^2
\right] \ddx
\\
&\hspace{0.3in} \nonumber
+ \frac{\e}{2{\rm Wi}^2} \sum_{m=1}^n \dt_m
  \intd \tr\brk{
  \brk{
  \Ad(\strhd^m,\tr(\strhd^m))
  }^2
  \Bd(\strhd^m)
  }\ddx
\\
&\hspace{1in}\nonumber
\leq
\Fd(\buh^0,\strh^0,\tr(\strh^0))
+ \frac{1+C_P}{2(1-\e)} \sum_{m=1}^n \dt_m \|\f^m\|_{H^{-1}(\D)}^2
\\
&\hspace{1in}\nonumber
\leq
F(\buh^0,\strh^0)
+ \frac{1+C_P}{2(1-\e)} \|\f\|_{L^2(0,t^n;H^{-1}(\D))}^2
\leq C,
\end{align}
which yields that
\begin{align}
\label{Fstab2}
\max_{n=0, \ldots, N_T} \int_\D \left[\, \|\buhd^n\|^2 + \|\strhd^n\| + 
\delta^{-1}\,\|[\strhd^n]_{-}\| + \delta^{-1}\,\left|[b-\tr(\strhd^n)]_{-}\right| 
\,\right] \ddx
&\leq C.
\end{align}

Moreover, for some $C(h,\Delta t) \in {\mathbb R}_{>0}$, but independent of $\delta$,
it follows that 
for $k=1,\ldots,N_K$ and $n=1,\ldots,N_T$
\begin{subequations}
\begin{alignat}{2} 
\label{Fstab3}
\Gd'\brk{1-\frac{\tr(\strhd^n)}{b}}\|\Bd(\strhd^n)\| 
&\le C(h,\Delta t) 
&&\quad\mbox{on }K_k,
\\
\label{Fstab4}
\|[\beta_\delta(\strhd^n)]^{-1}\| 
&\leq C(h,\Delta t)\left[1 +   
\Gd'\left(1-\frac{\tr(\strhd^n)}{b}
\right)\right] &&\quad\mbox{on }K_k.
\end{alignat} 
\end{subequations}
 \end{theorem}
\begin{proof}
Existence of a solution to {\bf (P$^{\dt}_{\delta,h}$)}
and the first inequality in (\ref{Fstab1}) follow
immediately from Propositions~\ref{prop:existence-Pdh}
and~\ref{prop:free-energy-Pdh}, respectively.
Similarly to (\ref{Fdin}), 
the second inequality in (\ref{Fstab1}) 
is a direct consequence of (\ref{eq:free-energy-Pd}), 
(\ref{eq:Gd}), (\ref{strh0pd},c) and (\ref{fncont}).
Finally, the third inequality  in (\ref{Fstab1}) follows from  
(\ref{strh0def}--c) and (\ref{freg}).

It follows from (\ref{Fstab1}) and (\ref{eq:comparison-oldroyd-d}) that
\begin{align}
\frac{\Re}{2}\intd\|\buhd^n\|^2 \,\ddx
+ \frac{\e}{2\Wi}\intd\tr\brk{\strhd^n-\Gd(\strhd^n)-\I}\,\ddx
\label{Fstab2a} \leq C, \qquad n=1,\ldots,N_T.
\end{align}
The first three bounds in (\ref{Fstab2}) then follow immediately from
(\ref{Fstab2a}) and (\ref{Entropy2}).
Next we note that (\ref{Entropy1}), (\ref{eq:free-energy-Pd}) and 
(\ref{Fstab1}) yield that
\begin{align}
\label{Fstab2b}
&b \intd
\Brk{ \left(1-\frac{\tr(\strhd^n)}{b}\right)-
 \Gd\brk{1-\frac{\tr(\strhd^n)}{b}}}
\ddx
\\
&\hspace{1in}\leq 
- \intd
\Brk{ 
b\: \Gd\brk{1-\frac{\tr(\strhd^n)}{b}}
+
\tr\brk{
\Gd(\strhd^n)+\I} - b} \ddx
\leq C.
\nonumber
\end{align}
The last bound in (\ref{Fstab2})
is then simply obtained by using a scalar version of~(\ref{Entropy2}).

Next, we deduce from (\ref{Fstab1}), (\ref{Addef}) and (\ref{eq:positive-term}) that 
for $n=1,\ldots,N_T$
\begin{align}
\label{Fstab3z}
0 \le 
\tr\brk{
\brk{
\Ad(\strhd^n,\tr(\strhd^n))
}^2
\Bd(\strhd^n)
}  
\le C(h,\Delta t) \quad\mbox{on }K_k, 
\qquad k=1,\ldots,N_K.
\end{align}
For any $\delta >0$, $\Bd(\strhd^n) \in \RSPD$ and so it follows from 
(\ref{ipmat}), (\ref{normphi}),  
(\ref{Addef}),
(\ref{eq:inverse-Gd}),  (\ref{normprod}),
(\ref{Fstab3z}), (\ref{modphisq}), (\ref{eq:Bd}) and (\ref{Fstab2}) 
that
\begin{align}
&\left\|\Gd'\brk{1-\frac{\tr(\strhd^n)}{b}}
\Bd(\strhd^n)- \I \right\|^2 
\label{Fstab3y}
\\
& \hspace{0.2in}
=\left\|  
\Ad(\strhd^n,\tr(\strhd^n))\,
\Bd(\strhd^n) \right\|^2
\leq \left\|  
\Ad(\strhd^n,\tr(\strhd^n))\,
[\Bd(\strhd^n)]^{\frac{1}{2}} \right\|^2
\left\|[\Bd(\strhd^n)]^{\frac{1}{2}}\right\|^2
\nonumber
\\
& \hspace{0.2in} = \tr\brk{
\brk{
\Ad(\strhd^n,\tr(\strhd^n))
}^2
\Bd(\strhd^n))
} 
\tr(\Bd(\strhd^n))
\leq C(h,\Delta t)\,\|\Bd(\strhd^n)\|
\nonumber \\
& \hspace{0.2in} \leq C(h,\Delta t)\,\left(\|\strhd^n\| + \delta\right)
\leq C(h,\Delta t) \quad\mbox{on }K_k, 
\qquad k=1,\ldots,N_K,
\quad n=1,\ldots,N_T.
\nonumber
\end{align}
The desired result (\ref{Fstab3}) follows immediately from (\ref{Fstab3y}).
Similarly to (\ref{Fstab3y}), we have from (\ref{Fstab3z}), (\ref{modphisq}), 
(\ref{Addef}) 
and (\ref{eq:inverse-Gd}) that
\begin{align}
C(h,\Delta t) &\geq \left\|
\Ad(\strhd^n,\tr(\strhd^n))\,
\Bd(\strhd^n) \,
\Ad(\strhd^n,\tr(\strhd^n))
\right\| 
\label{Fstab4z}\\
&\geq \left\|
\brk{\Gd'\brk{1-\frac{\tr(\strhd^n)}{b}}}^2
\Bd(\strhd^n) - 2
\Gd'\brk{1-\frac{\tr(\strhd^n)}{b}}\I
+ [\Bd(\strhd^n)]^{-1} \right\|
\nonumber \\ 
& \hspace{2in}\quad\mbox{on }K_k, 
\qquad k=1,\dots,N_K,
\quad n=1,\ldots,N_T.
\nonumber
\end{align}
The desired result (\ref{Fstab4}) follows immediately from (\ref{Fstab4z}) and 
(\ref{Fstab3}). 
\end{proof}

\subsection{Convergence of (P$^{\Delta t}_{\delta,h}$) to (P$^{\Delta t}_{h}$)}
\label{secPdhtoPd}
We now consider the
corresponding direct finite element approximation of {\bf (P)},
i.e. {\bf (P$^{\Delta t}_{h}$)} without the regularization $\delta$: \\
{\bf (P$^{\Delta t}_h$)}
Given initial conditions 
$(\buh^0,\strh^0)\in\Vhzero\times(\Sh^0\cap\SPDb)$ as defined in (\ref{strh0def}),
then for $n = 1, \ldots, N_T$ 
find $(\buh^{n},\strh^{n})\in\Vhzero\times\Sh^0$ such that for any test functions 
$(\bv,\bphi)\in\Vhzero\times\Sh^0$
\begin{subequations}
\begin{align} 
\label{eq:Pdh-limit}
&\int_\D \biggl[ \Re\left(\frac{\buh^{n}-\buh^{n-1}}{\dt_{n}}\right)\cdot \bv 
 + \frac{\Re}{2}\Brk{ \left( (\buh^{n-1}\cdot\nabla)\buh^{n}\right) \cdot \bv - 
 \buh^{n} \cdot \left( (\buh^{n-1}\cdot\nabla)\bv \right)}
\\
& \hspace{1in}
 + (1-\e) \gbuh^{n}:\grad\bv + \frac{\e}{\Wi} 
 A(\strh^n)\,\strh^{n}:\grad\bv
\biggr] \,\ddx = 
\langle \f^n,\bv \rangle_{H^1_0(\D)},
\nonumber 
\\ 
\label{eq:Pdhb-limit}
&\int_\D \left[ \left(\frac{\strh^{n}-\strh^{n-1}}{\dt_{n}}\right) : \bphi 
 - 2\left((\gbuh^{n})\,\strh^{n}\right) :\bphi 
 + \frac{A(\strh^n)\,\strh^{n}
 : \bphi} 
 {\Wi}
\right]
\,\ddx
\\
& \hspace{1in}
 + \sum_{j=1}^{N_E} \int_{E_j}  \Abs{\buh^{n-1}\cdot\bn} 
 \jump{\strh^{n}}_{\to\buh^{n-1}}:\bphi^{+\buh^{n-1}}
 \,{\rm d}\bs
 = 0.
\nonumber
\end{align}
\end{subequations}

We note that (\ref{eq:Pdh-limit},b) 
and $F(\buh^n,\strh^n)$
are only well-defined if $\strh^n \in \Sh^0 \cap \SPDb$.  
We also note that on replacing $A(\strh^{n})$ with 
$\I-(\strh^{n})^{-1}$ then {\bf (P$_{h}^{\dt}$)}, (\ref{eq:Pdh-limit},b), collapses
to the corresponding finite element approximation of Oldroyd-B studied in 
\cite{barrett-boyaval-09}, see (3.35a,b) there.

\begin{theorem} 
\label{dconthm}
For all regular partitionings $\mathcal{T}_h$ of $\D$
into simplices $\{K_k\}_{k=1}^{N_K}$
and all partitionings $\{\Delta t_n\}_{n=1}^{N_T}$ of $[0,T]$,
there exists a subsequence
$\{\{(\buhd^{n},\strhd^{n})\}_{n=1}^{N_T}\}_{\delta>0}$, where  
$\{(\buhd^{n},$ $\strhd^{n})\}_{n=1}^{N_T} \in [\Vhzero \times \Sh^0]^{N_T}$ 
solves {\bf (P$^{\dt}_{\delta,h}$)}, (\ref{eq:Pdh},b), and
$\{(\buh^{n},\strh^{n})\}_{n=1}^{N_T} \in [\Vhzero \times \Sh^0]^{N_T}$ 
such that for the subsequence
\begin{align}
\label{conv}
\buhd^{n} \rightarrow \buh^{n}, \qquad 
\strhd^{n} \rightarrow \strh^{n} \qquad \mbox{as } \delta \rightarrow 0_+\,,
\qquad\mbox{for} \quad n=1, \ldots, N_T.
\end{align}
In addition, for all $t \in [0,T]$
$n=1,\ldots, N_T$, $\strh^{n}\mid_{K_k} \in \RSPDb$, $k=1,\ldots, N_K,$. 
Moreover, $\{(\buh^{n},\strh^{n})\}_{n=1}^{N_T}\in [\Vhzero \times \Sh^0]^{N_T}$
solves {\bf (P$_{h}^{\dt}$)}, (\ref{eq:Pdh-limit},b), and for $n=1,\ldots,N_T$
\begin{align}
\label{eq:estimate-Ph}
&\frac{F(\buh^{n},\strh^{n})-F(\buh^{n-1},\strh^{n-1})}{\dt_{n}} 
+ \frac{ {\rm Re}}{2\dt_{n}}\intd\|\buh^{n}-\buh^{n-1}\|^2 \,\ddx
+(1-\e)\intd\|\gbuh^{n}\|^2\,\ddx
\\
\nonumber 
& \hspace{.3in} 
+ \frac{\e}{2 {\rm Wi}^2} \intd \tr\brk{
  \brk{A(\strh^n)
  }^2
  \strh^n
  }\ddx
\le
\frac{1}{2}(1-\e)\intd\|\gbuh^{n}\|^2 \,\ddx
+ \frac{1+C_P}{2(1-\e)} \|\f^{n}\|_{H^{-1}(\D)}^2.
\end{align}  
\end{theorem}
\begin{proof}
For any integer $n \in[1,N_T]$,
the desired subsequence convergence results (\ref{conv}) follow immediately 
from (\ref{Fstab2}),
as $(\buhd^{n},$ $\strhd^{n})$ are finite dimensional for fixed $\Vhzero \times S_h^0$. 
It also follows from (\ref{Fstab2}),
(\ref{conv}) and (\ref{Lip}) that $[\strh^{n}]_{-}$
and $[b-\tr(\strh^{n})]_{-}$ 
vanish on $\D$, 
so that
$\strh^{n}$ must be non-negative definite 
and $\tr(\strh^{n})\le b$ a.e.\ on $\D$. 
Moreover, on noting this, (\ref{conv}), (\ref{eq:Bd}) and (\ref{Lip}), 
we have the following subsequence convergence results
\begin{align}
\beta_{\delta}(\strh^{n}) \rightarrow \strh^{n}, \qquad   
\beta_{\delta}(\strhd^{n}) \rightarrow \strh^{n} \qquad \mbox{as} \quad
\delta \rightarrow 0_+.
\label{betacon}
\end{align}

If $\tr(\strh^{n})|_{K_k}=b$ on some simplex $K_k$, 
then for the subsequence of (\ref{conv}) we have that 
\begin{align}
\tr(\strhd^{n})|_{K_k} \to b \qquad\mbox{as}\quad \delta\to0_+.
\label{trtodb}
\end{align}
In addition, it follows from (\ref{modphisq}), (\ref{eq:Bd}) and (\ref{trtodb})
for $\delta$ sufficiently small that
\begin{align}
\|\Bd(\strhd^n)\| \geq \frac{1}{\sqrt{d}} \tr\brk{\Bd(\strhd^n)}
\geq \frac{1}{\sqrt{d}} \tr\brk{\strhd^n}
\geq \frac{b}{2\sqrt{d}} \qquad \mbox{on } K_k.
\label{Bdstbel}
\end{align}
Hence, (\ref{Fstab3}) and (\ref{Bdstbel}) yield for the subsequence of (\ref{conv})
for all $\delta$ sufficiently small that
\begin{align}
\Gd'\brk{1-\frac{\tr(\strhd^{n})}{b}} \leq C(h,\Delta t) \quad \mbox{on }K_k,
\label{Gddab}
\end{align}
but this contradicts (\ref{trtodb}) on recalling (\ref{eq:BGd}) and (\ref{eq:Bd}). 
Therefore, $\tr(\strh^{n})|_{K_k}< b$ on all simplices $K_k$,
and so it follows from (\ref{conv}), (\ref{eq:BGd}) and (\ref{eq:Bd}) that
\begin{align}
\Gd\brk{1-\frac{\tr(\strhd^{n})}{b}} \to \ln\brk{1-\frac{\tr(\strh^{n})}{b}}, \qquad
\Gd'\brk{1-\frac{\tr(\strhd^{n})}{b}} \to \brk{1-\frac{\tr(\strh^{n})}{b}}^{-1}
\label{Gdconv}
\end{align}
as $\delta \to 0_+$.

The results (\ref{Fstab4}) and (\ref{Gdconv}) yield for $\delta$ sufficiently small that
\begin{align}
\|[\beta_\delta(\strhd^n)]^{-1}\| 
&\leq C(h,\Delta t) \quad \mbox{on }K_k,
\quad k=1,\ldots,N_K.
\label{Fstab4x}
\end{align} 
Furthermore, it follows from (\ref{Fstab4x}), (\ref{betacon}), (\ref{eq:BGd}), (\ref{eq:Bd}) and as 
$[\Bd(\strhd^{n})]^{-1}\Bd(\strhd^{n})=\I$ that the following subsequence result 
\begin{align}
[\Bd(\strhd^{n})]^{-1} = \Gd'(\strhd^{n})\rightarrow [\strh^{n}]^{-1}\qquad \mbox{as} \quad
\delta \rightarrow 0_+
\label{posdef}
\end{align}
holds, and so $\strh^{n}\mid_{K_k} \in \RSPDb$, $k=1,\ldots, N_K$.   
Therefore, we have from (\ref{conv}) and (\ref{eq:Gd}) that
\begin{align}
\Gd(\strhd^{n}) \rightarrow \ln(\strh^{n}) \qquad \mbox{as} \quad
\delta \rightarrow 0_+.
\label{Gcon}
\end{align}

Since $\buhd^{n-1},\,\buh^{n-1} \in C(\overline{\D})$, it follows from (\ref{conv}),
(\ref{eq:streams}) and (\ref{eq:difsum}) that
for $j=1,\ldots, N_E$ and all $\bphi \in \Sh^0$
\begin{align}
\int_{E_j}  \Abs{\buhd^{n-1}\cdot\bn} 
 \jump{\strhd^{n}}_{\to\buhd^{n-1}}:\bphi^{+\buhd^{n-1}}\,{\rm d}\bs
\rightarrow 
\int_{E_j}  \Abs{\buh^{n-1}\cdot\bn} 
 \jump{\strh^{n}}_{\to\buh^{n-1}}:\bphi^{+\buh^{n-1}}\,{\rm d}\bs 
\quad \mbox{as} \quad
\delta \rightarrow 0_+.
\label{Edgcon}
\end{align} 
Hence using (\ref{conv}), (\ref{betacon}), (\ref{Gdconv}) and (\ref{Edgcon}),
we can pass to the limit $\delta \rightarrow 0_+$ for the subsequence
in {\bf (P$_{\delta,h}^{\dt}$)},
(\ref{eq:Pdh},b),
to show that $\{(\buh^{n},\strh^{n})\}_{n=1}^{N_T}
\in [\Vhzero \times \Sh^0]^{N_T}$
solves {\bf (P$_{h}^{\dt}$)}, (\ref{eq:Pdh-limit},b). 
Similarly, using (\ref{conv}), (\ref{betacon}), (\ref{Gdconv}), (\ref{posdef}) and (\ref{Gcon}), 
and noting (\ref{eq:free-energy-Pd}) and (\ref{eq:free-energy-P}),
we can pass to the limit $\delta \rightarrow 0_+$ in (\ref{eq:estimate-Pdh})
to obtain 
(\ref{eq:estimate-Ph}).
\end{proof}

\section{Fene-p model with stress diffusion}
\label{sec:Palpha}
\setcounter{equation}{0}
\subsection{Model (P$_\alpha$), (P) with stress diffusion}

In this section we consider  
the following
modified version of {\bf (P)}, (\ref{eq:fene-p-sigma}--f), 
with a stress diffusion term
for a given constant 
$\alpha \in \R_{>0}$:

\noindent
{\bf (P$_\alpha$)} Find 
$\bua : (t,\xx)\in[0,T)\times\D \mapsto \bua(t,\xx)\in\R^d$,  
$\pa : (t,\xx) \in \D_T \mapsto \pa(t,\xx)\in\R$ 
and $\stra : (t,\xx)\in[0,T)\times\D \mapsto\stra(t,\xx)
\in \RSPDb$ 
such that
\begin{subequations}
\begin{alignat}{2}
\Re \brk{\pd{\bua}{t} + (\bua\cdot\grad)\bua} & =  -\grad \pa 
+ (1-\e)\Delta\bua +  \frac{\e}{\Wi}\div\brk{ 
A(\stra)\,\stra} 
+ \f 
\qquad 
&&\mbox{on } \D_T, 
\label{eq:aoldroyd-b-sigma}
\\
\div\bua & = 0 
\qquad \qquad \qquad &&\mbox{on } \D_T , 
\label{eq:aoldroyd-b-sigma1}
\\
\pd{\stra}{t}+(\bua\cdot\grad)\stra & = 
(\gbua)\stra+\stra(\gbua)^T
\label{eq:aoldroyd-b-sigma2} 
-\frac 
{A(\stra)\,\stra}{\Wi}
+ \alpha \Delta \stra
\qquad 
&&\mbox{on } \D_T,
\\
\bua(0,\xx) &= \bu^0(\xx) 
\label{eq:ainitial}\qquad \qquad &&\hspace{-0.7in}\forall \xx \in \D,\\
\stra(0,\xx) &= \strs^0(\xx) 
\qquad \qquad &&\hspace{-0.7in}\forall \xx \in \D,
\label{eq:ainitial1}
\\
\label{eq:adirichlet}
\bua&=\bzero \qquad \qquad \qquad &&\hspace{-0.7in}\text{on $ (0,T) \times \partial\D$},
\\
\label{eq:aneumann}
(\bn_{\partial \D} \cdot
\grad) \stra &=\bzero \qquad 
\qquad \qquad 
&&\hspace{-0.7in}\text{on $ (0,T) \times \partial\D$}.
\end{alignat}
\end{subequations}

Hence problem {\bf (P$_\alpha$)} is the same as {\bf (P)}, but with 
the added diffusion term $\alpha \Delta \stra$
for the stress equation (\ref{eq:aoldroyd-b-sigma2}),
and the associated Neumann boundary condition (\ref{eq:aneumann}).
Similarly to {\bf (P$_\delta$)}, (\ref{eq:fene-p-sigmadL}--f),
we introduce a regularization of {\bf (P$_{\alpha,\delta}$)} of {\bf (P$_\alpha$)}
mimicking the free energy structure of {\bf (P$_\alpha$)}.
Moreover, we need to be able to construct a  
finite element approximation of {\bf (P$_{\alpha,\delta}$)} that 
satisfies a discrete analogue of this free energy structure.  
Apart from the obvious addition of the stress diffusion term, there
are three other distinct differences. 
First, one has to deal with the advective
term in the stress equation differently, see (\ref{PadLeber1a}) below,
to the approach used in (\ref{Pdener1}) 
for the stability bound,
as the approach there cannot be mimicked at a discrete level using continuous piecewise linear
functions to approximate $\strad$, the regularization of $\stra$.  
Note that one cannot use    
$\S_h^0$ with the desirable property (\ref{eq:inclusion1}) to approximate $\strad$
due to the additional stress diffusion term.
Second, as a consequence of
this stress advective term, 
one has to introduce another regularization of $\tr(\stra)$,  
$\varrhoad$, as well as the obvious candidate $\tr(\Bd(\strad))$,
and solve for this directly, see Remark \ref{remrho} below.
Third, it is desirable for the convergence analysis, as $\delta \rightarrow 0$,
to have a uniform $L^2(\D_T)$ bound on the extra stress term 
$\Ad(\strad,\varrhoad)\,\Bd(\strad)$ in the Navier--Stokes equation 
(\ref{eq:aoldroyd-b-sigmad}) below. 
To achieve this we replace $\Ad(\strad,\varrhoad)\,\Bd(\strad)$ there by
$\kd(\strad,\varrhoad)\,\Ad(\strad,\varrhoad)\,\Bd(\strad)$, where we
define, for $\delta \in (0,\min\{\frac{1}{2},b\}]$, 
\begin{align}
\kd(\bphi,\eta):= \left[ \frac{\Bdb(\eta)}{\tr(\Bd(\bphi))}\right]^{\frac{1}{2}}
\qquad \forall (\bphi,\eta) \in \RS \times \R
\label{kddef}
\end{align} 
with
\begin{align}
\Bdb: s \in \R \mapsto \min\{\Bd(s),b\}.
\label{BdLdef}
\end{align}  
It follows from (\ref{kddef}), (\ref{BdLdef}), (\ref{ipmat}) and (\ref{normphi}) that
\begin{align}
\|\kd(\bphi,\eta)\,\Ad(\bphi,\eta)\,\Bd(\bphi)\|^2 
\label{kddefbd}
&\leq  
\|\kd(\bphi,\eta)\,[\Bd(\bphi)]^{\frac{1}{2}}\|^2  
\|\Ad(\bphi,\eta)\,[\Bd(\bphi)]^{\frac{1}{2}}\|^2 
\\
&\leq  b 
\tr\brk{\brk{\Ad(\bphi,\eta)}^2 \Bd(\bphi)}
\qquad \forall (\bphi,\eta) \in \RS \times \R.
\nonumber
\end{align}
Hence a uniform $L^2(\D_T)$ bound on 
$\kd(\strad,\varrhoad)\,\Ad(\strad,\varrhoad)\,\Bd(\strad)$
follows from a uniform $L^1(\D_T)$ on 
$\tr\brk{\brk{\Ad(\strad,\varrhoad)}^2 \Bd(\strad)}$,
which will follow from the free energy bound. Although 
$\varrhoad \neq \tr(\Bd(\strad))$, we will show, in the limit $\delta \rightarrow 0$,
that  $\Bd(\strad) \rightarrow \stra$ and
$\varrhoad \rightarrow \tr(\stra)$ 
with $\stra(\cdot,\cdot) \in \RSPDb$,
and hence 
implying that $\kd(\strad,\varrhoad)
\rightarrow 1$. In order to maintain the free energy bound
we need to include $\kd(\strad,\varrhoad)$ on the right-hand sides of 
(\ref{eq:aoldroyd-b-sigma2d},d) below.

Therefore, we consider the following regularization of {\bf (P$_\alpha$)}
for a given $\delta \in (0,\min\{\frac{1}{2},b\}]$:  

\noindent
{\bf (P$_{\alpha,\delta}$)} Find 
$\buad : (t,\xx)\in[0,T)\times\D \mapsto \buad(t,\xx)\in\R^d$,  
$\pad : (t,\xx) \in \D_T \mapsto \pad(t,\xx)\in\R$, 
$\strad : (t,\xx)\in[0,T)\times\D \mapsto\strad(t,\xx)
\in \RS$ and
$\varrhoad : (t,\xx)\in[0,T)\times\D \mapsto\varrhoad (t,\xx)
\in \R$  
such that
\begin{subequations}
\begin{alignat}{2}
\Re \brk{\pd{\buad}{t} + (\buad\cdot\grad)\buad}
&=  -\grad \pad 
+ (1-\e)\Delta\buad 
\label{eq:aoldroyd-b-sigmad}\\
& \qquad
+  \frac{\e}{\Wi}\div\brk{ 
\kd(\strad,\varrhoad)
\Ad(\strad,\varrhoad)\,\Bd(\strad) }  
\nonumber \\
&&&\hspace{-1.04in}+ \f\qquad \mbox{on } \D_T, 
\nonumber\\
\div\buad &= 0 
&&\hspace{-0.5in}\mbox{on } \D_T, 
\label{eq:aoldroyd-b-sigma1d}
\\
\pd{\strad}{t}+(\buad\cdot\grad)\Bd(\strad) 
\label{eq:aoldroyd-b-sigma2d}
&= \kd(\strad,\varrhoad)\left[(\gbuad)\Bd(\strad)+\Bd(\strad)(\gbuad)^T \right] 
\\
& \qquad -\frac{
\Ad(\strad,\varrhoad)\,\Bd(\strad)}{\Wi} 
+ \alpha \Delta \strad
&&\hspace{-0.5in}\mbox{on } \D_T,
\nonumber
\\
\pd{\varrhoad}{t}-b\,(\buad\cdot\grad)\Bd(1-\frac{\varrhoad}{b}) 
\label{eq:aoldroyd-b-sigma2dtr}
&= 
2 \kd(\strad,\varrhoad)\,\gbuad : \Bd(\strad)
\\
& \qquad-\frac{
\tr(\Ad(\strad,\varrhoad)\,\Bd(\strad))}{\Wi}  
+ \alpha \Delta \varrhoad
&&\hspace{-0.5in}\mbox{on } \D_T,
\nonumber
\\
\buad(0,\xx) &= \bu^0(\xx) 
\label{eq:ainitiald}
&&\hspace{-1in}
\forall \xx \in \D,\\
\strad(0,\xx) &= \strs^0(\xx), \quad \varrhoad(0,\xx) = \tr(\strs^0(\xx)) 
&&\hspace{-1in} \forall \xx \in \D,
\label{eq:ainitial1d}
\\
\label{eq:adirichletd}
\buad&=\bzero 
&&\hspace{-1in}
\text{on $ (0,T) \times \partial\D$},
\\
\label{eq:aneumannd}
(\bn_{\partial \D} \cdot
\grad) \strad &=\bzero,  
\quad
(\bn_{\partial \D} \cdot
\grad) \varrhoad =0
&&\hspace{-1in}
\text{on $ (0,T) \times \partial\D$}.
\end{alignat}
\end{subequations}
We note from (\ref{eq:Bd}) that $-b\,\grad \Bd\brk{1-\frac{\varrhoad}{b}} = \grad \varrhoad$
if $\brk{1-\frac{\varrhoad}{b}} \geq \delta$, i.e.\ $\varrhoad \le b(1-\delta)$.
We remark again, due the required regularization of the advective terms in 
(\ref{eq:aoldroyd-b-sigma2d},d), that 
$\varrhoad  \neq \tr(\strad)$. However, we note that on taking the trace of 
(\ref{eq:aoldroyd-b-sigma2d}), subtracting (\ref{eq:aoldroyd-b-sigma2dtr})
and integrating over $\D$, yields, on noting (\ref{eq:aoldroyd-b-sigma1d},f--h) and (\ref{ipmat}), 
that
for all $t \in [0,T]$ 
\begin{align}
\intd [ \tr(\strad(t,\cdot)) - \varrhoad(t,\cdot)]\,\ddx = 
\intd [ \tr(\strad(0,\cdot)) - \varrhoad(0,\cdot)]\,\ddx
=0.
\label{trint} 
\end{align}   

\subsection{Formal free energy bound for (P$_{\alpha,\delta}$)}

First, similarly to (\ref{eq:Gd}), we introduce, for $\delta \in (0,1)$,
the concave $C^{1,1}(\R_{>0})$ function 
\begin{align}\label{eq:Hd}
&\Hd : s \in \R_{>0} \mapsto 
\begin{cases}
G(s) & \forall s \in (0, \delta^{-1}], 
\\
\delta\,s+G(\delta^{-1})-1  & \forall s \geq \delta^{-1}
\end{cases}
\quad \Rightarrow \quad \Hd'(\Gd'(s))=\Bd(s) \quad \forall s \in \R.
\end{align}
We have the following analogue of Proposition  \ref{prop:free-energy-Pd}.
\begin{proposition} \label{prop:free-energy-PadL}
Let $\alpha \in {\mathbb R}_{>0}$, $\delta \in (0,\min\{\frac{1}{2},b\}]$ 
and $(\buad,\pad,\strad,\varrhoad)$ 
be a sufficiently smooth 
solution to problem $(\bf P_{\alpha,\delta})$,
(\ref{eq:aoldroyd-b-sigmad}--h).
Then the free energy $\Fd(\buad,\strad,\varrhoad)$  
satisfies for a.a.\ $t \in (0,T)$
\begin{align} 
\label{eq:estimate-PadL}
&\deriv{}{t}\Fd(\buad,\strad,\varrhoad)
+(1-\e)\intd\|\gbuad\|^2\,\ddx 
\\
&\hspace{0.3in} + \frac{\alpha\e\delta^2}{2{\rm Wi}}
\int_\D \Brk{\|\grad \Gd'(\strad)\|^2 +b\, \left\|\grad \Gd'\brk{1-\frac{\varrhoad}{b}}\right\|^2}
\ddx
\nonumber
\\
& \hspace{0.3in}+ \displaystyle \frac{\e}{2 {\rm Wi}^2}\intd
\tr
\brk{
  \brk{
  \Ad(\strad,\varrhoad)
  }^2
  \Bd(\strad)}
\ddx
\leq
\langle \f,\buad\rangle_{H^1_0(\D)}.
\nonumber
\end{align}
\end{proposition}
\begin{proof} 
Similarly to the proof of Proposition \ref{eq:estimate-Pd}, 
on noting Remark \ref{remsplitd},
we
multiply the Navier-Stokes equation (\ref{eq:aoldroyd-b-sigmad}) with $\buad$
and the stress equation 
(\ref{eq:aoldroyd-b-sigma2d})
with $-\frac{\e}{2\Wi}\Gd'(\strad)$,
the trace equation 
(\ref{eq:aoldroyd-b-sigma2dtr})
with $\frac{\e}{2\Wi}\Gd'\brk{1-\frac{\varrhoad}{b}}$,
sum and integrate over $\D$.  
This yields, after 
performing integration by parts and noting 
(\ref{eq:aoldroyd-b-sigma1d},g,h), (\ref{eq:symmetric-tr}), (\ref{eq:inverse-Gd})
and (\ref{Addef}), that
\begin{align}
&\intd  \Brk{ \frac{\Re}{2} 
\pd{\|\buad\|^2}{t} + (1-\e)\|\gbuad\|^2 }\ddx 
\label{PdaLener} \\
&\hspace{0.5in}
-\frac{\e}{2\Wi}
\intd 
\brk{\pd{\strad}{t}
+(\buad\cdot\grad) 
\Bd(\strad)} : \Gd'(\strad)\,\ddx
\nonumber \\
& \hspace{0.5in}
+\frac{\e}{2\Wi} \intd 
\brk{\pd{\varrhoad}{t}
-b\,(\buad\cdot\grad) 
\Bd\brk{1-\frac{\varrhoad}{b}}} \Gd'\brk{1-\frac{\varrhoad}{b}}
\,\ddx
\nonumber \\
&\hspace{0.5in}
-\frac{\alpha\e}{2\Wi} \intd \Brk{
\grad\strad :: \grad \Gd'(\strad)-
\grad\varrhoad \cdot \grad \Gd'\brk{1-\frac{\varrhoad}{b}}}
\ddx
\nonumber \\
& \hspace{0.5in}+ \displaystyle \frac{\e}{2 {\rm Wi}^2}\intd
\tr
\brk{
  \brk{
  \Ad(\strad,\varrhoad) 
  }^2
  \Bd(\strad)}
\ddx
= 
\langle \f,\buad\rangle_{H^1_0(\D)}.
\nonumber 
\end{align}

Using (\ref{eq:deriv-tensor-g}),  
we have that
\begin{align}
&\pd{\strad}{t} : \Gd'(\strad)
= \pd{}{t} \tr\brk{\Gd(\strad)}
\quad\mbox{and} \quad
\pd{\varrhoad}{t}\, \Gd'\brk{1-\frac{\varrhoad}{b}}
= -b\,\pd{}{t}\Gd\brk{1-\frac{\varrhoad}{b}}.
\label{PadLener1}
\end{align}
As $\Bd(\strad)\equiv \Hd'
(\Gd'(\strad))$, on recalling (\ref{eq:Hd}),
we have, on noting
$\buad \in \Uz$ and the spatial version of (\ref{eq:deriv-tensor-g}), that 
\begin{align}
\label{PadLeber1a} 
-\int_\D (\buad\cdot\grad)\Bd(\strad):\Gd'(\strad)\,\ddx 
&= \int_\D \Bd(\strad): (\buad\cdot\grad)\Gd'(\strad) \,\ddx
 \\
& =
\int_\D
(\buad\cdot\grad)\tr\brk{\Hd(\Gd'(\strad))}\,\ddx = 0.
\nonumber
\end{align}
Similarly to (\ref{PadLeber1a}), we have that
\begin{align}
\label{PadLeber1atr} 
-\int_\D \Brk{(\buad\cdot\grad)\Bd\brk{1-\frac{\varrhoad}{b}}}
\Gd'\brk{1-\frac{\varrhoad}{b}} \,\ddx
& =
\int_\D
(\buad\cdot\grad)\Hd\brk{\Gd'\brk{1-\frac{\varrhoad}{b}}}\,\ddx = 0.
\end{align}
Similarly to (\ref{matrix:Lipschitz}), we have that 
\begin{subequations}
\begin{alignat}{2}
-\grad \strad :: \grad \Gd'(\strad)
&\geq \delta^2 \| \grad \Gd'(\strad)\|^2 
\qquad \qquad &&\mbox{a.e.\ in } \D_T, 
\label{PdaLener0}
\\
\grad\varrhoad \cdot \grad \Gd'\brk{1-\frac{\varrhoad}{b}}
&= -b \grad \brk{1-\frac{\varrhoad}{b}} \cdot \grad \Gd'\brk{1-\frac{\varrhoad}{b}}
\qquad &&
\label{Pdalener0b} \\
&\geq b\,\delta^2 \left\| \grad \Gd'\brk{1-\frac{\varrhoad}{b}}\right\|^2 
\qquad \qquad &&\mbox{a.e.\ in } \D_T.
\nonumber
\end{alignat}
\end{subequations}
Combining (\ref{PdaLener})--(\ref{PdaLener0},b) yields the desired result (\ref{eq:estimate-PadL}).
\end{proof}

\begin{remark} \label{remrho}
We note if we multiply the advection term in 
(\ref{eq:aoldroyd-b-sigma2d}) by $-\Gd'(1-\frac{\tr(\strad)}{b})\,\I$ 
and integrate over $\D$, then we obtain, on noting $\buad \in \Uz$ and
(\ref{PadLeber1atr}), that
\begin{align*}
&-\int_\D (\buad \cdot \grad) \Bd(\strad) 
: \Gd'\brk{1-\frac{\tr(\strad)}{b}}\I \ddx\\  
&\hspace{1in} 
=\int_\D
\tr(\Bd(\strad)) \, (\buad \cdot \grad) \Gd'\brk{1-\frac{\tr(\strad)}{b}}\, \ddx   \\
& \hspace{1in}=-b \int_\D
\left(1 - \frac{\tr(\Bd(\strad))}{b}\right) 
\, (\buad \cdot \grad) \Gd'\brk{1-\frac{\tr(\strad)}{b}}\, \ddx 
\\
& \hspace{1in}\neq -b \int_\D
\Bd\left(1 - \frac{\tr(\strad)}{b}\right) 
\, (\buad \cdot \grad) \Gd'\brk{1-\frac{\tr(\strad)}{b}}\, \ddx =0.
\end{align*}
Hence, the need for the new variable $\varrhoad$ in order to mimic the 
free energy structure of {\bf (P$_\alpha$)}.
\end{remark}

The following Corollary follows from   
(\ref{eq:estimate-PadL}) on noting the proof of 
Corollary \ref{cor:free-energy-Pd}.

\begin{corollary} 
\label{cor:free-energy-PadL}
Under the assumptions of Proposition \ref{prop:free-energy-PadL}
it follows that
\begin{align}
\label{eq:free-energy-boundaL}
&\sup_{t \in (0,T)}\Fd(\buad(t,\cdot),\strad(t,\cdot),\varrhoad(t,\cdot))
+
\frac{1-\e}{2}
\int_{\D_T} 
\|\gbuad\|^2\,\ddx\,\ddt 
\\
& \qquad +
\frac{\alpha\e\delta^2}{{2\rm Wi}} 
\,\int_{\D_T}\Brk{
\|\grad \Gd'(\strad)\|^2
+b \,\left\|\grad \Gd'\brk{1-\frac{\varrhoad}{b}}\right\|^2}
\ddx\,\ddt
\nonumber
\\
& \qquad
+\frac{\e}{2{\rm Wi}^2}
\int_{\D_T}
\tr\brk{
  \brk{
  \Ad(\strad,\varrhoad)
  }^2
  \Bd(\strad)}
\ddx\,\ddt
\nonumber
\\
& \hspace{1.3in}
\le 2\brk{ F(\bu^0,\strs^0) + \frac{1+C_P}{2(1-\e)} 
\Norm{\f}_{L^2(0,T;H^{-1}(\D))}^2 }.
\nonumber
\end{align}
\end{corollary}

\section{Finite element approximation of {\bf(P$_{\alpha,\delta}$)}}
\label{sec:Palphah}
\setcounter{equation}{0}
\subsection{Finite element discretization}

\label{5.1}
We now introduce a conforming finite element discretization of
{\bf (P$_{\alpha,\delta}$)}, (\ref{eq:aoldroyd-b-sigmad}--h),
which satisfies a discrete analogue of 
(\ref{eq:estimate-PadL}).  
As noted in  
Section \ref{sec:Palpha}, 
we cannot use $\S_h^0$ with the desirable property
(\ref{eq:inclusion1}) to approximate $\strad$, 
as we now have the added diffusion term. 
In the following, we choose 
 \begin{subequations}
 \begin{align}
 \Uh^1 &:= \Uh^2  \subset \U \quad \mbox{or} \quad \Uh^{1,+}
 \subset \U,
 \label{Vhmini} \\
 {\rm Q}_h^1 &=\{q \in C(\overline{\D}) \,:\, q \mid_{K_k} \in \PP_1 \quad
 k=1,\ldots, N_K\} \subset {\rm Q},  
 \label{Qh1}\\
 \Shone &=\{\bphi \in [C(\overline{\D})]^{d \times d}_{\rm S} 
 \,:\, \bphi \mid_{K_k} \in [\PP_1]^{d \times d}_{\rm S} \quad
 k=1,\ldots, N_K\} \subset \S\, 
 \label{Sh1}\\
\mbox{and} \qquad
\Vhone &= \BRK{\bv \in\Uh^1 \,:\, \int_{\D} q \, {\div\bv} \,\ddx= 0 \quad  
\forall q\in\mathrm{Q}_h^1} \,;
\label{Vh1}
\end{align}
\end{subequations}
where $\Uh^2$ is defined as in (\ref{Vh2}) and, on recalling the barycentric
coordinate notation used in (\ref{bvarsig}),
\begin{align}
  \Uh^{1,+}&:=\left\{\bv \in [C(\overline{\D})]^d\cap \U \,:\, \bv \mid_{K_k} \in \left[\PP_1
  \oplus \mbox{span} \prod_{i=0}^d \eta^k_i \right]^d \quad
 k=1,\ldots, N_K\right\}\label{Vh1p}\,.
\end{align}

The velocity-pressure choice, $\Uh^2 \times {\rm Q}_h^1$, is the lowest order
Taylor-Hood element. It    
satisfies (\ref{eq:LBB}) with $\Uh^0$ and ${\rm Q}_h^0$
replaced by $\Uh^2$ and
${\rm Q}_h^1$, respectively,
provided, in addition to $\{{\mathcal T}_h\}_{h>0}$ being a regular family of meshes,
that each simplex has at least one vertex in $\D$, see p177 in Girault and 
Raviart \cite{GiraultRaviart} in the case $d=2$ and Boffi \cite{Boffi} in the 
case $d=3$.      
Of course, this is a very mild restriction on $\{{\mathcal T}_h\}_{h>0}$. 
The velocity-pressure choice, $\Uh^{1,+} \times {\rm Q}_h^1$, is called the
mini-element. It    
satisfies (\ref{eq:LBB}) with $\Uh^0$ and ${\rm Q}_h^0$
replaced by $\Uh^{1,+}$ and
${\rm Q}_h^1$, respectively;
see Chapter II, Section 4.1 in Girault and 
Raviart \cite{GiraultRaviart} in the case $d=2$ 
and Section 4.2.4 in Ern and Guermond \cite{ern-guermond-04} in the 
case $d=3$.  
Hence for both choices of $\Uh^1$, it follows that
for all $\bv \in \Uz$ there exists a sequence $\{\bv_h\}_{h>0}$, with
$\bv_h \in \Vhone$, such that
\begin{align}
\lim_{h \rightarrow 0_+} \|\bv-\bv_h\|_{H^1(\D)} =0\,.
\label{Vhconv} 
\end{align}  
 
We recall the well-known local inverse inequality
for ${\rm Q}^1_h$
\begin{alignat}{3}
\|q\|_{L^{\infty}(K_k)} &\leq C\,|K_k|^{-1}\,\int_{K_k} |q| \,\ddx
\qquad &&\forall q \in {\rm Q}^1_h, \qquad
&&k=1,\ldots,N_K  
\label{inverse}
\\
\Rightarrow \qquad
\|\bchi\|_{L^{\infty}(K_k)} &\leq C\,|K_k|^{-1}\,\int_{K_k} \|\bchi\| \,\ddx\qquad &&\forall 
\bchi \in {\rm S}^1_h, \qquad
&&k=1,\ldots,N_K. 
\nonumber
\end{alignat}
We recall a similar well-known local inverse inequality for $\Shone$
and $\Vhone$
\begin{subequations}
\begin{align}
\|\grad \bphi\|_{L^2(K_k)} &\leq C\, h_k^{-1} \|\bphi\|_{L^2(K_k)}
\qquad \forall 
\bphi \in \Shone, \qquad
k=1,\ldots,N_K,
\label{Sinverse}\\
\|\grad \bv\|_{L^2(K_k)} &\leq C\, h_k^{-1} \|\bv\|_{L^2(K_k)}
\qquad \forall 
\bv \in \Vhone, 
\qquad
k=1,\ldots,N_K.  
\label{Vinverse}
\end{align} 
\end{subequations}

We introduce the interpolation operator $\pi_h : C(\overline{\D})
\rightarrow {\rm Q}^1_h$, and extended naturally to 
$\pi_h : [C(\overline{\D})]^{d \times d}_{\rm S}
\rightarrow {\rm S}^1_h$,
such that for all $\eta \in C(\overline{\D})$
and $\bphi \in [C(\overline{\D})]^{d \times d}_{\rm S}$ 
\begin{align}
\pi_h \eta( P_p) =\eta(P_p) 
\qquad
\mbox{and}
\qquad  
\pi_h \bphi( P_p) =\bphi(P_p) 
\qquad p=1,\ldots, N_P,  
\label{eq:pi1h}
\end{align}  
where $\{P_p\}_{p=1}^{N_P}$ are the vertices of $\mathcal{T}_h$.
As $\bphi \in \Shone$ and $q \in \Qhone$ do not imply that  
$\Gd'(\bphi) \in \Shone$ and 
$\Gd'\brk{1-\frac{q}{b}} \in \Qhone$, we have to test the finite element 
approximation 
of (\ref{eq:aoldroyd-b-sigma2d},d) with
$-\frac{\e}{2\Wi}\pi_h[\Gd'(\stradh^{n})] \in \Shone$ and 
$\frac{\e}{2\Wi}\pi_h\Brk{\Gd'\brk{1-\frac{\varrhoadh^n}{b}}} \in \Qhone$, respectively, where
$\stradh^{n}\in \Shone$ and $\varrhoadh^n\in \Qhone$ are our finite element approximations
to $\strad$ 
and $\varrhoad$ at time level $t_{n}$.
Therefore the finite element approximation 
of (\ref{eq:aoldroyd-b-sigma2d},d)
have to be constructed to mimic the results 
(\ref{PdaLener})--(\ref{PdaLener0},b),
when tested with $-\frac{\e}{2\Wi}\pi_h[\Gd'(\stradh^{n})] \in \Shone$
and $\frac{\e}{2\Wi}\pi_h\Brk{\Gd'\brk{1-\frac{\varrhoadh^n}{b}}} \in \Qhone$, respectively.

In order to mimic (\ref{PadLener1}) and the {\bf (P$_{\alpha,\delta}$)} analogue of 
(\ref{Pdener2}), we need to use 
numerical integration (vertex sampling).
We note the following results.
As the basis functions associated with ${\rm Q}^1_h$ and $\Shone$
are nonnegative and sum to unity everywhere, we have, on noting (\ref{normprod}), 
for $k=1,\ldots,N_K$ that  
\begin{subequations}
\begin{alignat}{2}
\|\pi_h [\bphi \,\bpsi]\| & \leq 
\pi_h [\,\|\bphi\| \,\|\bpsi\|\,] && 
\label{interprod}\\
&\leq \left[\pi_h[\,\|\bphi\|^{r_1}\,]\,\right]^{\frac{1}{r_1}}\,
\left[\pi_h[\,\|\bpsi\|^{r_2}\,]\,\right]^{\frac{1}{r_2}}
\quad &&\mbox{on }K_k, \quad \forall 
\bphi,\,\bpsi \in [C(\overline{K_k})]^{d \times d}_{\rm S},
\nonumber \\
\|\pi_h \bphi\|^2 &\leq \pi_h[\,\|\bphi\|^2\,]
\quad &&\mbox{on }K_k, \quad \forall 
\bphi \in [C(\overline{K_k})]^{d \times d}_{\rm S},
\label{interpinf} 
\end{alignat}
\end{subequations}
where $r_1,\,r_2 \in (1,\infty)$ satisfy $r_1^{-1}+r_2^{-1}=1$.
In addition, we have for $k=1,\ldots,N_K$ that
\begin{align}
\int_{K_k} \|\bchi\|^2 \,\ddx\leq \int_{K_k} \pi_h[\,\|\bchi\|^2]\,\ddx 
\leq C \int_{K_k} \|\bchi\|^2 \,\ddx \qquad  \forall \bchi \in \Shone.
\label{eqnorm}
\end{align}
The first inequality in (\ref{eqnorm}) follows immediately from
(\ref{interpinf}), and the second from applying (\ref{inverse})
and a Cauchy--Schwarz inequality.  
Of course, scalar versions of (\ref{interprod},b) and (\ref{eqnorm})
hold with $[C(\overline{K_k})]^{d \times d}_{\rm S}$ and $\Shone$
replaced by 
$C(\overline{K_k})$ and $\Qhone$, respectively.
 
Furthermore, for later use,
we recall the following well-known results concerning the interpolant $\pi_h$
for $k=1,\ldots, N_K$:
\begin{subequations}
\begin{align}
\|(\I-\pi_h)q\|_{W^{1,\infty}(K_k)} &\leq C\,h\,|q|_{W^{2,\infty}(K_k)} 
\qquad \forall q \in W^{2,\infty}(K_k), 
\label{interp1}\\
\|(\I-\pi_h)[q_1 \,q_2]\|_{L^1(K_k)}
& \leq C\,h_k^2\,\|\nabla q_1\|_{L^2(K_k)}\,\|\nabla q_2\|_{L^2(K_k)}
\label{interp2}
\\
&\leq C\,h_k\,\|q_1\|_{L^2(K_k)}\,\|\nabla q_2\|_{L^2(K_k)}
\qquad \forall q_1,\,q_2 \in \Qhone. 
\nonumber
\end{align}
\end{subequations}

In order to mimic (\ref{PadLeber1a}) and (\ref{PadLeber1atr}), 
we have to carefully construct our finite
element approximation of the advective terms  
in (\ref{eq:aoldroyd-b-sigma2d},d). 
Our construction is a non-trivial extension of an approach that has been used 
in the finite element approximation of fourth-order degenerate nonlinear parabolic
equations, such as the thin film equation; see e.g.\ 
Gr\"{u}n and Rumpf \cite{GrunRumpf}
and Barrett and N\"{u}rnberg \cite{surf2d}. 
Let $\{\be_i\}_{i=1}^d$ be the orthonormal vectors 
in $\R^d$, 
such that the $j^{th}$ component of $\be_i$ is $\delta_{ij}$, $i,j=1,\ldots,d$.
Let $\widehat{K}$   
be the standard open reference simplex in $\R^d$
with vertices $\{\widehat{P}_i\}_{i=0}^d$, 
where $\widehat{P}_0$ is the origin and 
$\widehat{P}_i = \be_i$, $i=1,\ldots,d$.
Given a simplex $K_k\in\mathcal{T}_h$ with vertices $\{P^k_i\}_{i=0}^d$,
then there exists 
a non-singular matrix $B_k$ 
such that the linear mapping
\beq \label{eq:mapping-RK}
\mathcal{B}_k : \widehat{\xx}\in\R^d \mapsto P_0^k + B_k\widehat{\xx} \in \R^d 
\eeq
maps vertex $\widehat{P}_i$  
to vertex $P^k_i$, $i=0,\ldots,d$. 
Hence $\mathcal{B}_k$ maps $\widehat{K}$ 
to $K_k$.
For all $\eta\in  
{\rm Q}^1_h$
and $K_k \in \mathcal{T}_h$, we define
\begin{align}
\label{eq:mapping-derivative}
\widehat{\eta}(\widehat{\xx}) := \eta(\mathcal{B}_k(\widehat{\xx})) \quad 
\forall \widehat{\xx}\in\widehat{K}
\quad \Rightarrow
\quad
\grad \eta(
\mathcal{B}_k(\widehat \xx)
) = 
(B_k^T)^{-1}  
\widehat{\grad} \widehat{\eta}
(\widehat \xx)
\quad \forall \widehat \xx\in \widehat{K}\,,
\end{align}
where for all  $\widehat \xx\in \widehat{K}$ 
\begin{align}
\label{diff}
[\widehat{\grad} \widehat{\eta}
(\widehat \xx)]_j = \frac{\partial}{\partial \widehat x_j}\widehat{\eta}
(\widehat \xx) = \widehat \eta(\widehat{P}_j)-\widehat \eta (\widehat{P}_0)
= \eta({P}_j^k)-\eta ({P}_0^k)\qquad j=1,\ldots,d.
\end{align}
Such notation is easily extended to $\bphi \in \S^1_h$.

Given $q \in \Qhone$ and $K_k \in \mathcal{T}_h$, then first, for  
$j=1, \ldots, d$, we find the unique 
$\widehat \Lambda_{\delta,j}^k(\widehat q) \in \R$,
which is continuous on $q$,
such that
\begin{align}
\widehat \Lambda_{\delta,j}^k(\widehat q) \,  
\frac{\partial}{\partial \widehat x_j}
\widehat{\pi}_h[\Gd'(\widehat q)] =
\frac{\partial}{\partial \widehat x_j}
\widehat{\pi}_h[\Hd(\Gd'(\widehat q))]
\qquad \mbox{on } \widehat K,
\label{qhLambdaj}
\end{align}
where $(\widehat \pi_h \widehat \eta)(\widehat \xx) \equiv
(\pi_h \eta)(\mathcal{B}_k \widehat \xx)$ for all
$\widehat \xx \in \widehat K$ and $\eta \in C(\overline{K_k})$.
We set 
\begin{align}
\widehat \Lambda_{\delta,j}^k(\widehat q)
:= \left\{ \begin{array}{ll}
\displaystyle \frac{\Hd(\Gd'(q(P_j^k))) -
\Hd(\Gd'(q(P_0^k)))}
{\Gd'(q(P_j^k)) -
\Gd'(q(P_0^k))} \qquad &\mbox{if }\Bd(q(P_j^k)) \neq
\Bd(q(P_0^k)),\\[2mm]
\Bd(q(P_j^k))=\Bd(q(P_0^k))
\qquad &\mbox{if }\Bd(q(P_j^k)) =
\Bd(q(P_0^k)),
\end{array}
\right.
\label{qhLambdajdef}
\end{align}
where we have noted that $\Gd'(\Bd(s))=\Gd'(s)$ for all $s \in \R$
and $\Gd'(\cdot)$ is strictly decreasing on $[\delta,\infty)$.
Clearly, $\widehat \Lambda_{\delta,j}^k(\widehat q)
\in \R$, $j=1,\ldots,d$, 
satisfies (\ref{qhLambdaj}) and
depends continuously on $q \mid_{K_k}$.

Next, we extend the construction (\ref{qhLambdajdef}) for a  
given $\bphi \in \Shone$ and $K_k \in \mathcal{T}_h$, to find for  
$j=1, \ldots, d$ the unique 
$\widehat \Lambda_{\delta,j}^k(\widehat \bphi) \in \RS$,
which is continuous on $\bphi$,
such that
\begin{align}
\widehat \Lambda_{\delta,j}^k(\widehat \bphi) :  
\frac{\partial}{\partial \widehat x_j}
\widehat{\pi}_h[\Gd'(\widehat \bphi)] =
\frac{\partial}{\partial \widehat x_j}
\widehat{\pi}_h[\tr(\Hd(\Gd'(\widehat \bphi)))]
\qquad \mbox{on } \widehat K.
\label{hLambdaj}
\end{align} 
To construct $\widehat \Lambda_{\delta,j}^k(\widehat \bphi)$
satisfying (\ref{hLambdaj}),
we note the following. 
We have from (\ref{eq:Hd})  
and (\ref{gconcave2}) that
\begin{align}
\label{hLambdajineq}
\Bd(\bphi(P_j^k)) :
(\Gd'(\bphi(P_j^k)) -
\Gd'(\bphi(P_0^k)))
&\leq 
\tr(\Hd(\Gd'(\bphi(P_j^k)) -
\Hd(\Gd'(\bphi(P_0^k)))
\\
& 
\leq
\Bd(\bphi(P_0^k)) :
(\Gd'(\bphi(P_j^k)) -
\Gd'(\bphi(P_0^k))).
\nonumber
\end{align}
Since $\Gd'(\Bd(s)) = \Gd'(s)$ for all $s \in \R$,
it follows from (\ref{matrix:Lipschitz}) that
\begin{align}
\label{betajtrnoL} 
&-(\Bd(\bphi(P_j^k))-
\Bd(\bphi(P_0^k))):
(\Gd'(\bphi(P_j^k)) -
\Gd'(\bphi(P_0^k))) 
\\
&\hspace{1in} = -(\Bd(\bphi(P_j^k))-
\Bd(\bphi(P_0^k))):
(\Gd'(\Bd(\bphi(P_j^k))) -
\Gd'(\Bd(\bphi(P_0^k))))
\nonumber \\
& \hspace{1in}
\geq \delta^2 \|\Gd'(\Bd(\bphi(P_j^k)))-
\Gd'(\Bd(\bphi(P_0^k)))\|^2.
\nonumber 
\end{align}
Therefore the left-hand side of 
(\ref{betajtrnoL}) is zero   
if and only if $\Gd'(\Bd(\bphi(P_j^k)))=\Gd'(\Bd(\bphi(P_0^k)))$,
which is equivalent to $\Bd(\bphi(P_j^k))=\Bd(\bphi(P_0^k))$
as $\Gd'(\cdot)$ is invertible on $[\delta,\infty)$, the range of $\Bd(\cdot)$.
Hence, on noting (\ref{diff}), 
(\ref{hLambdajineq}),  
(\ref{betajtrnoL})
and (\ref{ipmat}), we have that
\begin{subequations}
\begin{align}
\widehat \Lambda_{\delta,j}^k(\widehat \bphi)
&:= (1-\lambda_{\delta,j}^k)\Bd(\bphi(P_j^k))
+ \lambda_{\delta,j}^k\Bd(\bphi(P_0^k))
\label{hLambdajdef}  \\
&\hspace{0.9in} \mbox{if} \quad  (\Bd(\bphi(P_j^k))-
\Bd(\bphi(P_0^k))):
(\Gd'(\bphi(P_j^k)) -
\Gd'(\bphi(P_0^k))) \neq 0\,,
\nonumber 
\\
\widehat \Lambda_{\delta,j}^k(\widehat \bphi)
&:= \Bd(\bphi(P_j^k))
= \Bd(\bphi(P_0^k))
 \label{hLambdajdefb} \\
& \hspace{0.9in} \mbox{if} \quad  (\Bd(\bphi(P_j^k))-
\Bd(\bphi(P_0^k))):
(\Gd'(\bphi(P_j^k)) -
\Gd'(\bphi(P_0^k))) = 0
\nonumber  
\end{align}
\end{subequations}
satisfies (\ref{hLambdaj}) for $j=1,\ldots,d$;
where $\lambda_{\delta,j}^k \in [0,1]$ is defined as  
\begin{align*}
\lambda_{\delta,j}^k &:= \frac{\Bigl[\tr(\Hd(\Gd'(\bphi(P_j^k)) -
\Hd(\Gd'(\bphi(P_0^k)))
-\Bd(\bphi(P_j^k)):
(\Gd'(\bphi(P_j^k)) -\Gd'(\bphi(P_0^k)))\Bigr]}
{(\Bd(\bphi(P_0^k))-
\Bd(\bphi(P_j^k))):
(\Gd'(\bphi(P_j^k)) -
\Gd'(\bphi(P_0^k)))}.
\end{align*}
Furthermore, $\widehat \Lambda_{\delta,j}^k(\widehat \bphi)
\in \RS$, $j=1,\ldots,d$, depends continuously on $\bphi \mid_{K_k}$.

Therefore given $q\in \Qhone$ and $\bphi \in \Shone$, we introduce, 
for $m,p =1, \ldots, d$,
\begin{subequations}
\begin{alignat}{2}
\Lambda_{\delta,m,p}(q) 
&= \sum_{j=1}^d
[(B_k^T)^{-1}]_{mj}\, \widehat \Lambda_{\delta,j}^k(\widehat q)
\,[B_k^T]_{jp} 
\in \R,
\qquad &&\mbox{on } K_k, 
\qquad k=1,\ldots,N_K,  
\label{Lambdampdefq}\\
\Lambda_{\delta,m,p}(\bphi) 
&= \sum_{j=1}^d
[(B_k^T)^{-1}]_{mj}\, \widehat \Lambda_{\delta,j}^k(\widehat \bphi)
\,[B_k^T]_{jp} \in \RS
\qquad &&\mbox{on } K_k, 
\qquad k=1,\ldots,N_K.  
\label{Lambdampdef}
\end{alignat}
\end{subequations}
It follows from 
(\ref{Lambdampdefq},b), (\ref{qhLambdaj}),
(\ref{hLambdaj}) and (\ref{eq:mapping-derivative}) that
\begin{align}
\Lambda_{\delta,m,p}(q) \approx \Bd(q) \,\delta_{mp},\qquad 
\Lambda_{\delta,m,p}(\bphi) \approx \Bd(\bphi) \,\delta_{mp}
\qquad  m,p=1,\ldots,d\,;
\label{Lambdampap}
\end{align}
and for $m=1,\ldots,d$ 
\begin{subequations}
\begin{alignat}{2}
\sum_{p=1}^d 
\Lambda_{\delta,m,p}(q) \, 
\frac{\partial}{\partial x_p}\pi_h[\Gd'(q)]
&=\frac{\partial}{\partial x_m}\pi_h [\Hd(\Gd'(q))]
\qquad &&\mbox{on } K_k, 
\label{Lambdajq}
\qquad k=1,\ldots,N_K,\\  
\sum_{p=1}^d 
\Lambda_{\delta,m,p}(\bphi) : 
\frac{\partial}{\partial x_p}\pi_h[\Gd'(\bphi)]
&=\frac{\partial}{\partial x_m}\pi_h [\tr(\Hd(\Gd'(\bphi)))]
\qquad &&\mbox{on } K_k,
\label{Lambdaj}
\qquad k=1,\ldots,N_K.  
\end{alignat}
\end{subequations}
For a more precise version of (\ref{Lambdampap}), see Lemma \ref{lemMXitt}
below. 
Of course, for (\ref{Lambdampdefq}) and (\ref{Lambdajq})
we can adopt the more compact notation
on $K_k$, $k=1,\ldots,N_K$, 
\begin{align}
\Lambda_{\delta}(q) 
= (B_k^T)^{-1}\, \widehat \Lambda_{\delta}^k(\widehat q)\,B_k^T 
\approx \Bd(q)\, \I \qquad \Rightarrow \qquad 
\Lambda_{\delta}(q) \grad \pi_h[\Gd'(q)] = \grad \pi_h [\Hd(\Gd'(q))],  
\label{Lambdajqcom}
\end{align} 
where $\Lambda_{\delta}^k(\widehat q) \in \RS$ is diagonal with 
$[\Lambda_{\delta}^k(\widehat q)]_{jj}=
\Lambda_{\delta,j}^k(\widehat q)$, $j=1,\ldots,d$, so that 
$\Lambda_{\delta}(q) \in \RS$ with $[\Lambda_{\delta}(q)]_{mp}
=\Lambda_{\delta,m,p}(q)$, $m,p=1,\ldots,d$. 

Finally, as the partitioning ${\mathcal T}_h$ consists of regular simplices, we have that
\begin{align}
\|(B_k^T)^{-1}\|\,\|B_k^T\| \leq C, \qquad k=1,\ldots,N_K\,.
\label{Breg}
\end{align}
Hence, it follows  from (\ref{Lambdampdefq},b), (\ref{Breg}) and (\ref{hLambdajdef},b) 
that for $k=1,\ldots,N_K$
\begin{subequations}
\begin{align}
\|\Lambda_{\delta,m,p}(q)\|_{L^{\infty}(K_k)} &\leq C\,
\|\pi_h[\Bd(q)]\|_{L^{\infty}(K_k)} \qquad \forall q \in \Qhone, 
\label{LdmpLinfq}
\\
\|\Lambda_{\delta,m,p}(\bphi)\|_{L^{\infty}(K_k)} &\leq C\,
\|\pi_h[\Bd(\bphi)]\|_{L^{\infty}(K_k)}
\qquad
\forall \bphi \in \Shone.
\label{LdmpLinf}
\end{align}
\end{subequations}

In order to mimic (\ref{PdaLener0},b),
we shall assume from now on that the family of meshes, $\{\mathcal{T}_h\}_{h>0}$, 
for the polytope
$\D$ 
consists of non-obtuse simplices only,
i.e. all dihedral angles of any simplex in
$\mathcal{T}_h$
are less than or equal to
$\frac{\pi}{2}$. 
Let $K_k$ have vertices $\{P^k_j\}_{j=0}^d$, and let $\eta^k_j(\xx)$ be the basis 
functions on $K_k$ associated with $\rm{Q}^1_h$ and $\Shone$, i.e.
$\eta^k_j \mid_{K_k} \in \mathbb{P}_1$ and $\eta^k_j(P^k_i)=\delta_{ij}$,
$i,j = 0,\ldots,d$. As $K_k$ is non-obtuse it follows that
\begin{align}
\grad \eta^k_i \cdot \grad \eta^k_j \leq 0 \qquad \mbox{on }K_k, \qquad
i,j = 0, \ldots , d, \ \mbox{ with }\ i \neq j.
\label{obtuse}
\end{align}
We then have the following result.

\begin{lemma} \label{lem:inf-bound}
Let $g \in C^{0,1}(\mathbb R)$ be monotonically increasing  
with Lipschitz constant $g_{\rm Lip}$.
As $\mathcal{T}_h$ consists of only non-obtuse simplices,
then we have for all $q \in {\rm Q}_h^1 $, $\bphi\in\Shone$ that
\begin{align} 
g_{\rm Lip} \,\grad\pi_h[g(q)] \cdot \grad q
&\ge 
\|\grad\pi_h[g(q)]\|^2 \quad 
\label{eq:inf-bound}
\mbox{and} \quad
g_{\rm Lip} \,\grad\pi_h[g(\bphi)] :: \grad \bphi
\ge 
\|\grad\pi_h[g(\bphi)]\|^2 
\\
&
\quad 
\hspace{1.5in}
\mbox{on } K_k, \quad
k=1, \ldots, N_K. 
\nonumber
\end{align}
\end{lemma}
\begin{proof}
See the proof of Lemma 5.1 in \cite{barrett-boyaval-09}. 
\end{proof}

Of course, the construction of a non-obtuse mesh in the case 
$d=3$ is not straightforward for a general polytope ${\mathcal D}$.
However, we stress that our numerical 
method {\bf (P$_{\alpha,\delta,h}^{\Delta t}$)}, see (\ref{eq:PdLaha}--c) below,
does not require this constraint. It is only required to show that 
{\bf (P$_{\alpha,\delta,h}^{\Delta t}$)} 
mimics the free energy structure of {\bf (P$_{\alpha,\delta}$)}.

\subsection{A free energy preserving approximation, (P$^{\Delta t}_{\alpha,\delta,h}$),
of (P$_{\alpha,\delta}$)}

In addition to the assumptions on the finite element discretization stated
in subsection \ref{5.1}, and our definition of $\Delta t$ in subsection \ref{FEd},    
we shall assume for the convergence analysis, see Section \ref{sec:convergence},  that
there exists a $C \in {\mathbb R}_{>0}$ such that
\begin{align}
\Delta t_n \leq C\, \Delta t_{n-1}, \qquad n=2, \ldots, N,\qquad
\mbox{as} \quad \Delta t \rightarrow 0_+.
\label{Deltatqu}
\end{align}
We note that this constraint is not required for the results in this section,
in particular Theorem \ref{dstabthmaorL}.

With $\Delta t_1$ and $C$ as above, let $\Delta t_0 \in {\mathbb R}_{>0}$ be such that
$\Delta t_1 \leq C \Delta t_0$.
Given initial data satisfying (\ref{freg}), we choose
$\buh^{0} \in \Vhone$ and $\strh^{0} \in \Shone$
throughout the rest of this paper  
such that
\begin{subequations}
\begin{alignat}{2}
\intd \left[ \buh^{0}  \cdot \bv  + \Delta t_0
\grad  \buh^{0} : \grad \bv \right]\,\ddx
&=   \intd \bu^0 \cdot \bv \,\ddx\qquad
&&\forall \bv \in \Vhone,
\label{proju0} \\
\intd \left[ \pi_h[\strh^{0} : \bchi]  + \Delta t_0
\grad  \strh^{0} :: \grad \bchi \right] \,\ddx
&=
\intd \strs^{0} : \bchi \,\ddx 
&& \forall \bchi \in \Shone.
\label{projp0}
\end{alignat}
\end{subequations}
It follows from (\ref{proju0},b), (\ref{eqnorm}) and (\ref{freg}) that
\begin{align}
&\int_{\D} \left[\, \|\buh^{0}\|^2 + \|\strh^{0}\|^2 +
\Delta t_0 \, \left[\|\grad \buh^{0}\|^2 + \|\grad \strh^0\|^2
\right]\,\right]\,\ddx 
\leq C.
 \label{idatah}
\end{align}
In addition, we note the following result.

\begin{lemma}
\label{lem:idatahspd}
For $p=1,\ldots, N_P$ we have that
\begin{align}
&\sigma_{\rm min}^0\, \|\bxi\|^2 \leq {\bxi}^T \strh^0(P_p) \,{\bxi} 
\leq \sigma_{\rm max}^0\, \|\bxi\|^2
\quad \forall \bxi \in {\mathbb R}^d
\qquad \mbox{and}
\qquad \tr(\strh^0(P_p)) \leq b^\star.
\label{idatahspd}
\end{align}
\end{lemma}
\begin{proof}
For the proof of the first result in (\ref{idatahspd}), see the proof of Lemma
5.2 in \cite{barrett-boyaval-09}. 

We now prove the second result in (\ref{idatahspd}).
On choosing $\bchi = \I \eta$, with $\eta \in {\rm Q}^1_h$,
in (\ref{projp0}) yields that $z_h := 
\tr(\strh^{0}) - b^\star \in  {\rm Q}^1_h$ satisfies
\begin{align}
\intd \left[ \pi_h[z_h \,\eta]  + \Delta t_0
\grad z_h \cdot \grad \eta \right]\,\ddx
=
\intd z \, \eta  \,\ddx
\qquad \forall \eta \in {\rm Q}^h_1,
\label{projp0z}
\end{align}  
where $z := 
\tr(\strs^{0})- b^\star \in L^\infty(\D)$ 
and is non-positive on recalling (\ref{freg}).
Choosing $\eta=\pi_h[z_h]_{+}\in {\rm Q}^1_h$, it follows, on noting
the  ${\rm Q}^1_h$ version of (\ref{eqnorm}) and
(\ref{eq:inf-bound}) with $g(\cdot) = [\,\cdot\,]_{+}$, that
\begin{align}
\label{trbd0h}
\intd \left [\pi_h [z_h]_{+}]^2 + \Delta t_0 \,\|\grad \pi_h[z_h]_{+} \|^2
\right] \,\ddx& \leq \int_d \left[ \pi_h\left[ [z_h]_{+}^2 \right]  + \Delta t_0
\grad z_h \cdot \grad \pi_h [z_h]_{+} \right]\,\ddx
 \\
&
=\intd z \,  \pi_h [z_h]_{+} \,\ddx \leq 0.
\nonumber
\end{align} 
Hence $\pi_h[z_h]_{+} \equiv 0$ and so the second result in (\ref{idatahspd}) holds.   
\end{proof}

Furthermore, it follows from (\ref{proju0},b), (\ref{idatah}), (\ref{freg}),
(\ref{Vhconv}) and (\ref{interp1},b) that, as $h,\,\Delta t_0 \rightarrow 0_+$,
\begin{align}
\buh^{0} \rightarrow \bu^0 \quad \mbox{weakly in } [L^2(\D)]^d
\qquad \mbox{and} \qquad \strh^0
\rightarrow  \strs^0 \quad \mbox{weakly in } [L^2(\D)]^{d\times d}.  
\label{bu0hconv} 
\end{align}

Our approximation {\bf (P$_{\alpha,\delta,h}^{\dt}$)} 
of {\bf (P$_{\alpha,\delta}$)}
is then:

({\bf P}$_{\alpha,\delta,h}^{\Delta t}$)
Setting 
$(\buhdLa^{0},\strhdLa^{0},\varrhohda^0)
= (\buh^{0},\strh^{0},\tr(\strh^0))
\in\Vhone\times(\Shone\cap\SPDb) \times \Qhone$, 
with $\buh^0$ and $\strh^0$ as defined in (\ref{proju0},b),
then for $n = 1, \ldots, N_T$ 
find $(\buhdLa^{n},$ $\strhdLa^{n},\varrhohda^n)\in\Vhone\times\Shone\times\Qhone$ 
such that for any test functions 
$(\bv,\bphi,\eta)\in\Vhone\times\Shone \times \Qhone$
\begin{subequations}
\begin{align} 
\label{eq:PdLaha}
 &\int_\D \Biggl[ \Re\left(\frac{\buhdLa^{n}-\buhdLa^{n-1}}{\dt_{n}}\right)
\cdot \bv 
+ \frac{\Re}{2}\Brk{ \left( (\buhdLa^{n-1}\cdot\nabla)\buhdLa^{n}\right) 
 \cdot \bv - 
 \buhdLa^{n} \cdot \left( (\buhdLa^{n-1}\cdot\nabla)\bv \right)}
\\
& \quad
 + (1-\e) \gbuhdLa^{n}:\grad\bv + \frac{\e}{\Wi}\,\pi_h\! 
 \left[ 
 \kd(\strhdLa^n,\varrhohda^n)\, \Ad(\strhdLa^n,\varrhohda^n)\,
 \Bd(\strhdLa^{n})\right] : 
 \grad\bv 
\Biggr] \ddx
\nonumber \\
& \hspace{3in}= 
\langle \f^n,\bv\rangle_{H^1_0(\D)},
\nonumber
\\ 
\label{eq:PdLahb}
&\int_\D \pi_h \left[\left(\frac{\strhdLa^{n}-\strhdLa^{n-1}}{\dt_{n}}\right) 
: \bphi + \frac{\Ad(\strhdLa^n,\varrhohda^n)\,
 \Bd(\strhdLa^{n}): \bphi}{\Wi}
\right]\ddx
\\ \nonumber
&\quad +  \int_D \left[ \alpha \grad\strhdLa^{n}::\grad\bphi
- 2
\gbuhdLa^{n}: \pi_h[\kd(\strhdLa^n,\varrhohda^n)\,\bphi\,\Bd(\strhdLa^{n})\right]
\,\ddx\\
& \quad 
- \int_\D \sum_{m=1}^d \sum_{p=1}^d 
[\buhdLa^{n-1}]_m \,\Lambda_{\delta,m,p}(\strhdLa^{n})
: \frac{\partial \bphi}{\partial \xx_p}\,\ddx
= 0,
\nonumber
\\
\label{eq:PdLahc}
&\int_\D \pi_h \left[\left(\frac{\varrhohda^{n}-\varrhohda^{n-1}}{\dt_{n}}\right) 
\eta + 
\frac{\tr\brk{\Ad(\strhdLa^n,\varrhohda^n)\,\Bd(\strhdLa^n)}}{\Wi} 
\,\eta 
\right]\ddx
\\ \nonumber
&\quad +  \int_D 
\left[ \alpha \grad\varrhohda^{n}\cdot\grad\eta  - 
2\gbuhdLa^{n}: \pi_h[\kd(\strhdLa^n,\varrhohda^n)\,\eta\,\Bd(\strhdLa^{n})]\,\right]
\,\ddx
\\
& \quad 
+b \int_\D \sum_{m=1}^d \sum_{p=1}^d 
[\buhdLa^{n-1}]_m \,\Lambda_{\delta,m,p}\left(1-\frac{\varrhohda^{n}}{b}\right)
\, \frac{\partial \eta}{\partial \xx_p}
\,\ddx
= 0 .
\nonumber
\end{align}
\end{subequations}

In deriving {\bf (P$^{\Delta t}_{\alpha,\delta,h}$)}, we have noted
 (\ref{conv0c}),
 (\ref{eq:symmetric-tr}), (\ref{Lambdampdefq},b) and (\ref{Lambdampap}). 
We note that on replacing $\Ad(\strhdLa^{n}, \tr(\varrhohda^{n}))$ with 
$\I-\Gd'(\strhdLa^{n})$ and $\kd(\strhdLa^{n},\varrhohda^{n})$ by $1$ 
then {\bf (P$_{\alpha,\delta,h}^{\dt}$)}, (\ref{eq:PdLaha},b), collapses
to the corresponding finite element approximation of Oldroyd-B with stress diffusion 
studied in \cite{barrett-boyaval-09}, see (5.34a,b) with no $L$ cut-off there.

Before proving existence of a solution to {\bf (P$^{\Delta t}_{\alpha,\delta,h}$)},
we first derive a discrete analogue of the energy bound (\ref{eq:estimate-PadL})
for {\bf (P$_{\alpha,\delta}$)}.

\subsection{Energy bound}

On setting
\begin{align}
 \label{eq:free-energy-PdLah}
 \Fdh
 (\bv,\bphi,\eta) &:= \frac{\Re}{2}\intd\|\bv\|^2\,\ddx 
 - \frac{\e}{2\Wi}\intd
 \pi_h\left[b\,\Gd\left(1-\frac{\eta}{b}\right)+
 \tr\brk{\Gd(\bphi)+\I}\right] \ddx
\\ & \hspace{2.5in}
\qquad \forall (\bv,\bphi,\eta) \in\Vhone \times \Shone \times \Qhone,
\nonumber
\end{align}
we have the following discrete analogue of Proposition \ref{prop:free-energy-PadL}.
\begin{proposition} \label{prop:free-energy-PdLah}
For $n= 1, \ldots, N_T$, a solution $\brk{\buhdLa^{n},\strhdLa^{n},\varrhohda^n}
\in \Vhone\times\Shone\times \Qhone$ to {\bf (P$_{\alpha,\delta,h}^{\dt}$)},
(\ref{eq:PdLaha}--c), if it exists, satisfies 
\begin{align}
&\frac{\Fdh(\buhdLa^{n},\strhdLa^{n},\varrhohda^n)-\Fdh(\buhdLa^{n-1},
\strhdLa^{n-1},\varrhohda^{n-1})}{\dt_{n}} 
\label{eq:estimate-PdLah}
\\
& \hspace{0.3in}
+ \frac{ {\rm Re}}{2\dt_{n}}\intd\|\buhdLa^{n}-\buhdLa^{n-1}\|^2 \,\ddx 
+(1-\e)\intd\|\gbuhdLa^{n}\|^2 \,\ddx
\nonumber \\
& \hspace{0.3in}
+\frac{\e}{2{\rm Wi}^2}\intd
\pi_h\left[\tr\left(
\left( 
\Ad(\strhdLa^n,\varrhohda^n)
\right)^2
\Bd(\strhdLa^{n})\right)
\right]\ddx
\nonumber \\
& \hspace{0.3in} 
+ \frac{\alpha\e\delta^2}{2{\rm Wi}}
\intd \left[ \|\grad\pi_h[\Gd'(\strhdLa^{n})]\|^2 + b\left\| 
\grad\pi_h\left[\Gd'\left(1-\frac{\varrhohda^{n}}{b}\right)\right]\right\|^2 \right]
\ddx
\nonumber \\
& \hspace{0.7in}
\le
\langle \f^n,\buhdLa^{n}\rangle_{H^1_0(\D)}
\le 
\frac{1}{2}(1-\e)\intd\|\gbuhdLa^{n}\|^2\,\ddx
+ \frac{1+C_P}{2(1-\e)} \|\f^{n}\|_{H^{-1}(\D)}^2.
\nonumber 
\end{align}
\end{proposition}
\begin{proof}
The proof is similar to that of  Proposition~\ref{prop:free-energy-Pdh},
we choose as test functions $\bv=\buhdLa^{n}\in\Vhone$,
$\bphi=-\frac{\e}{2\Wi}\pi_h[\Gd'(\strhdLa^{n})]\in\Shone$
and $\eta= \frac{\e}{2\Wi}\pi_h\left[\Gd'\left(1-\frac{\varrhohda^{n}}{b}\right)\right] \in \Qhone$
in (\ref{eq:PdLaha}--c), 
and obtain, on noting (\ref{elemident}), (\ref{eq:inverse-Gd},d), (\ref{Addef}),
(\ref{eq:inf-bound}) with $g=-\Gd'$ having Lipschitz
constant $\delta^{-2}$, (\ref{Lambdajq},b) and (\ref{eq:free-energy-PdLah}) 
that
\begin{align}
 \label{eq:free-energy-PdhLa-demo1}
\langle \f^n,\buhdLa^{n}\rangle_{H^1_0(\D)}
&\geq
\frac{\Fdh(\buhdLa^{n},\strhdLa^{n},\varrhohda^n)-\Fdh(\buhdLa^{n-1},
\strhdLa^{n-1},\varrhohda^{n-1})}
{\dt_{n}} 
\\
& \qquad  
+ \frac{ {\rm Re}}{2\dt_{n}}\intd\|\buhdLa^{n}-\buhdLa^{n-1}\|^2\,\ddx
+(1-\e)\intd\|\gbuhdLa^{n}\|^2\,\ddx
\nonumber \\
& \qquad 
+\frac{\e}{2{\rm Wi}^2}\intd
\pi_h\left[\tr\left(\left( 
\Ad(\strhdLa^n,\varrhohda^n)
\right)^2
\Bd(\strhdLa^{n})\right)\right]\ddx
\nonumber
\\
& \qquad 
+ \frac{\alpha\e\delta^2}{2{\rm Wi}}
\intd \left[ \|\grad\pi_h[\Gd'(\strhdLa^{n})]\|^2 + b\left\| 
\grad\pi_h\left[\Gd'\left(1-\frac{\varrhohda^{n}}{b}\right)\right]\right\|^2 \right]
\ddx
\nonumber \\
& \qquad 
+ \frac{\e}{2{\rm Wi}}\intd \buhdLa^{n-1} \cdot \grad \pi_h\left[\tr(\Hd(\Gd'(\strhdLa^{n})))
+b \Hd\left(\Gd'\left(1-\frac{\varrhohda^n}{b}\right)\right)
\right]
\ddx.
\nonumber
\end{align}
The first desired inequality in (\ref{eq:estimate-PdLah})
follows immediately from (\ref{eq:free-energy-PdhLa-demo1}) 
on noting (\ref{Vhmini},d), (\ref{spaces}) and that $\pi_h : C(\overline{\D}) 
\rightarrow {\rm Q}^1_h$.
The second inequality in (\ref{eq:estimate-PdLah}) follows immediately from (\ref{fbound})
with $\nu^2 = (1-\e)/(1+C_P)$. 
\end{proof}

\subsection{Existence of discrete solutions}
\begin{proposition} 
\label{prop:existence-PdLah}
Given $(\buhdLa^{n-1},\strhdLa^{n-1},\varrhohda^{n-1}) \in \Vhone \times \Shone \times \Qhone$
such that 
$\intd [ \tr(\strhdLa^{n-1})-\varrhohda^{n-1}]\,\ddx$\linebreak$=0$
and for any time step $\dt_{n} > 0$,
then there exists at least 
one solution $\brk{\buhdLa^{n},\strhdLa^{n},\varrhohda^{n}} \in \Vhone\times\Sh^1 \times 
\Qhone$ 
to {\bf (P$_{\alpha,\delta,h}^{\dt}$)},
(\ref{eq:PdLaha}--c), such that 
$\intd [ \tr(\strhdLa^{n})-\varrhohda^{n}]\,\ddx=0$.
\end{proposition}
\begin{proof}
The proof is similar to that of Proposition \ref{prop:existence-Pdh}.
We introduce the following inner product on 
the Hilbert space $\Vhone\times\Shone \times \Qhone$
$$ \Scal{ (\bw,\bpsi,\xi) }{ (\bv,\bphi,\eta) }_\D^h =\! 
\intd \Brk{ \bw\cdot\bv + \pi_h[\bpsi:\bphi + \xi\,\eta]\, }\,\ddx 
\qquad \forall (\bw,\bpsi,\xi),(\bv,\bphi,\eta) \in \Vhone\times\Shone \times \Qhone.
$$
Given $(\buhdLa^{n-1},\strhdLa^{n-1},\varrhohda^{n-1})\in
\Vhone\times\Shone \times \Qhone$, let
$\mathcal{F}^h : \Vhone\times\Shone \times \Qhone \rightarrow \Vhone\times\Shone \times 
\Qhone$ be such that for any
$(\bw,\bpsi,\xi) \in \Vhone\times\Shone \times \Qhone$
\begin{align} 
\label{eq:mapping1}
&\Scal{\mathcal{F}^h(\bw,\bpsi,\xi)}{(\bv,\bphi,\eta)}_\D^h 
\\
&\; := \int_\D \Biggl[\Re \left(\frac{\bw-\buhdLa^{n-1}}
{\dt_{n}}\right) \cdot \bv
+ (1-\e) \grad\bw :\grad\bv + \frac{\e}{\Wi} 
\pi_h\left[ 
\kd(\bpsi,\xi)\,\Ad(\bpsi,\xi)\,
\Bd(\bpsi)\right]
:\grad\bv   
\nonumber \\   
& \qquad \qquad      
  + \frac{\Re}{2} \left[\left((\buhdLa^{n-1}\cdot\nabla)\bw \right) \cdot \bv 
  - \bw \cdot \left((\buhdLa^{n-1}\cdot\nabla)\bv \right) \right]  
\nonumber \\[2mm]
& \qquad \qquad + \alpha \left[ \grad \bpsi :: \grad \bphi
+ \grad \xi \cdot \grad \eta \right] 
- 2  \grad\bw : \pi_h\left[\,\kd(\bpsi,\xi) \left[\bphi+ \eta \I\right] \Bd(\bpsi)\right]
\Biggr] \ddx
\nonumber
\\[2mm]
& \quad 
+ \int_\D \pi_h\biggl[\left(\frac{\bpsi-\strhdLa^{n-1}}{\dt_{n}}\right):\bphi 
 + \frac{\Ad(\bpsi,\xi)\,\Bd(\bpsi)}{\Wi}
 : \bphi
+ 
\left(\frac{\xi-\varrhohda^{n-1}}{\dt_{n}}\right)\eta 
 + \frac{\tr\brk{\Ad(\bpsi,\xi)\,\Bd(\bpsi) }}{\Wi}
 \eta 
 \Biggr]\ddx 
\nonumber
\\
& \quad 
- \intd \sum_{m=1}^d \sum_{p=1}^d 
[\buhdLa^{n-1}]_m \left[\Lambda_{\delta,m,p}(\bpsi) :
\frac{\partial \bphi}{\partial \xx_p} 
-b\,\Lambda_{\delta,m,p}\left(1-\frac{\xi}{b}\right) \,
\frac{\partial \eta}{\partial \xx_p} 
\right] \ddx
- \langle \f^n, \bv \rangle_{H^1_0(\D)}
\nonumber \\
& \hspace{4in}
\quad \forall (\bv,\bphi,\eta) \in \Vhone \times \Shone \times \Qhone.
\nonumber
\end{align}
A solution 
$(\buhdLa^{n},\strhdLa^{n},\varrhohda^n)$
to (\ref{eq:Pdh}--c), if it exists,
corresponds to a zero of $\mathcal{F}^h$. 
On recalling 
(\ref{Lambdampdefq},b), (\ref{hLambdajdef},b) and (\ref{qhLambdajdef}), 
it is easily deduced that the mapping $\mathcal{F}^h$ is continuous.
For any $(\bw,\bpsi,\xi) \in \Vhone\times\Shone \times \Qhone$, on choosing
$(\bv,\bphi,\eta) = \brk{\bw,-\frac{\e}{2\Wi}\pi_h[\Gd'(\bpsi)],
\frac{\e}{2\Wi}\pi_h\left[\Gd'\left(1-\frac{\xi}{b}\right)\right]}$,
we obtain analogously to (\ref{eq:estimate-PdLah}) that
\begin{align} 
 \label{eq:inequality1aL}
&\Scal{\mathcal{F}^h(\bw,\bpsi,\xi)}
{\brk{\bw,-\frac{\e}{2\Wi}\pi_h[\Gd'(\bpsi)],
\frac{\e}{2\Wi}\pi_h\left[\Gd'\left(1-\frac{\xi}{b}\right)\right]
}}_\D^h
\\
& \qquad \ge \frac{\Fdh(\bw,\bpsi,\xi)-\Fdh(\buhdLa^{n-1},
\strhdLa^{n-1},\varrhohda^{n-1})}
{\dt_{n}} 
+ \frac{\Re}{2\dt_{n}}\intd\|\bw-\buhdLa^{n-1}\|^2\,\ddx
\nonumber
\\
& \quad\qquad 
+ \frac{1-\e}{2} 
\intd\|\grad\bw\|^2\,\ddx
-\frac{1+C_P}{2(1-\e)}\Norm{\f^{n}}_{H^{-1}(\D)}^2
+\frac{\e}{2{\rm Wi}^2}\intd
\pi_h\left[\tr\left(\left( 
\Ad(\bpsi,\xi)
\right)^2
\Bd(\bpsi)\right)
\right]\ddx
\nonumber \\
& \quad \qquad  
+ \frac{\alpha\e\delta^2}{2{\rm Wi}}
\intd \left[ \|\grad\pi_h[\Gd'(\bpsi)]\|^2 + b\left\| 
\grad\pi_h\left[\Gd'\left(1-\frac{\xi}{b}\right)\right]\right\|^2 \right]
\ddx.
\nonumber
\end{align}

Let
$$ \Norm{(\bv,\bphi,\eta)}_D^h := 
\left[((\bv,\bphi,\eta),(\bv,\bphi,\eta))_D^h\right]^{\frac{1}{2}} =
\brk{\intd\Brk{\,\|\bv\|^2+\pi_h[\,\|\bphi\|^2 + |\eta|^2 \,]\,}\,\ddx }^\frac12. $$
If for any $\gamma \in \R_{>0}$, 
the continuous mapping $\mathcal{F}^h$ 
has no zero $(\buhdLa^{n},\strhdLa^{n},\varrhohda^n)$,  
which lies in the ball
$$ \mathcal{B}_\gamma^h := \BRK{ (\bv,\bphi,\eta) \in \Vhone\times\Shone \times \Qhone 
\,: \, 
\Norm{(\bv,\bphi,\eta)}_\D^h\le\gamma };$$
then for such $\gamma$, we can define the continuous mapping $\mathcal{G}_\gamma^h
: \mathcal{B}_\gamma^h \rightarrow  \mathcal{B}_\gamma^h$
such that for all $(\bv,\bphi,\eta) \in \mathcal{B}_\gamma^h$
$$ \mathcal{G}_\gamma^h(\bv,\bphi,\eta) := -\gamma 
\frac{\mathcal{F}^h(\bv,\bphi,\eta)}{\Norm{\mathcal{F}^h(\bv,\bphi,\eta)}_\D^h}. $$
By the Brouwer fixed point theorem, $\mathcal{G}_\gamma^h$ has at least 
one fixed point $(\bw_\gamma,\bpsi_\gamma,\xi_\gamma)$ in $\mathcal{B}_\gamma^h$. 
Hence it satisfies
\begin{equation}\label{eq:fixed-pointaL}
\Norm{(\bw_\gamma,\bpsi_\gamma,\xi_\gamma)}_\D^h=
 \Norm{\mathcal{G}_\gamma^h(\bw_\gamma,\bpsi_\gamma,\xi_\gamma)}_\D^h=\gamma.
\end{equation}
In addition,
$\Scal{(\bw_\gamma,\bpsi_\gamma,\xi_\gamma)}{(\bv,\bphi,\eta)}_\D^h= 
\Scal{\mathcal{G}^h(\bw_\gamma,\bpsi_\gamma,\xi_\gamma)}{(\bv,\bphi,\eta)}_\D^h$
with $(\bv,\bphi,\eta)=(\bzero,\I,-1)$ yields that 
\begin{align}
\intd [ \tr(\bpsi_\gamma)-\xi_\gamma]\,\ddx
=\intd [ \tr(\strhdLa^{n-1})-\varrhohda^{n-1}]\,\ddx=0.
\label{intgam}
\end{align}

On noting  
(\ref{inverse}), we have that  
there exists a $\mu_h \in \R_{>0}$
such that for all $\bphi \in \Shone$,
\begin{equation}\label{eq:norm_equivalenceaL} 
\|\pi_h[\,\|\bphi\|\,]\|_{L^\infty(\D)}^2
= \|\pi_h[\,\|\bphi\|^2]\|_{L^\infty(\D)}
\leq  
\mu_h^2 \int_\D \pi_h[\,\|\bphi\|^2\,]\,\ddx,
\end{equation}
and an equivalent result holding for all $\eta \in \Qhone$. 
It follows from (\ref{eq:free-energy-PdLah}), (\ref{intgam}),  (\ref{Entropy2}), 
(\ref{Gdbbelow}), 
(\ref{eq:norm_equivalenceaL})  
and (\ref{eq:fixed-pointaL}) 
that 
\begin{align} 
\label{bb1aL} 
&\Fdh(\bw_\gamma,\bpsi_\gamma,\xi_\gamma)  
\\
& \quad=
\frac{\Re}{2} \intd \|\bw_\gamma\|^2 \,\ddx + \frac{\e}{2\Wi}
\intd
\pi_h\left[ \tr(\bpsi_\gamma-\Gd(\bpsi_\gamma)-\I) - b \,\Gd\left( 1-\frac{\xi_\gamma}{b} 
\right) - \xi_\gamma \right]\ddx
\nonumber
\\
& \quad \ge
\frac{\Re}{2} \intd\|\bw_\gamma\|^2 \,\ddx+ \frac{\e}{4\Wi}
\left[\intd \pi_h[\,\|\bpsi_\gamma\|+ |\xi_\gamma|\,]\,\ddx-(2d+3b)|\D|\right]
\nonumber
\\
& \quad \ge 
\frac{\Re}{2} \intd\|\bw_\gamma\|^2 \,\ddx
-\frac{\e (2d+3b) |\D|}{4\Wi}
\nonumber \\
& \quad \quad  
+ \frac{\e}{4\Wi\,\mu_h\gamma}
\left[\|\pi_h[\,\|\bpsi_\gamma\|\,]\|_{L^\infty(\D)}
\intd \pi_h[\,\|\bpsi_\gamma\|]\,\ddx+\|\pi_h[\,|\xi_\gamma|\,]\|_{L^\infty(\D)}
\intd \pi_h[\,|\xi_\gamma|\,]\,\ddx\right]
\nonumber
\\
& \ge
\min\brk{\frac{\Re}{2},\frac{\e}{4\Wi\,\mu_h\gamma}}
\brk{ \intd\left[\,\|\bw_\gamma\|^2+ \pi_h[\,\|\bpsi_\gamma\|^2
+|\xi_\gamma|^2\,]\,\right] \,\ddx}
-\frac{\e (2d+3b) |\D|}{4\Wi}
\nonumber
\\
& = 
\min\brk{\frac{\Re}{2},\frac{\e}{4\Wi\,\mu_h\gamma}}
\gamma^2 
-\frac{\e (2d+3b) |\D|}{4\Wi}.
\nonumber
\end{align}
Hence 
for all $\gamma$ sufficiently large,
it follows from (\ref{eq:inequality1aL}) and (\ref{bb1aL})   
that
\beq \label{eq:one-handaL}
\Scal{\mathcal{F}^h(\bw_\gamma,\bpsi_\gamma,\xi_\gamma)}{\brk{\bw_\gamma,-\frac{\e}{2\Wi}
\pi_h[\Gd'(\bpsi_\gamma)],\frac{\e}{2\Wi}
\pi_h\!\left[\Gd'\left(1-\frac{\xi_\gamma}{b}\right)\right]}}_D^h
\ge 0.
\eeq

On the other hand as $(\bw_\gamma,\bpsi_\gamma,\xi_\gamma)$ is a fixed point 
of ${\mathcal G}_\gamma^h$, we have that
\begin{align}
\label{eq:whereasaL}
&\Scal{\mathcal{F}^h(\bw_\gamma,\bpsi_\gamma,\xi_\gamma)}{\brk{\bw_\gamma,-\frac{\e}{2\Wi}
\pi_h[\Gd'(\bpsi_\gamma)],\frac{\e}{2\Wi}
\pi_h\!\left[\Gd'\left(1-\frac{\xi_\gamma}{b}\right)\right]}}_D^h
\\ 
& \hspace{0.2in} = 
-\frac{\Norm{\mathcal{F}^h(\bw_\gamma,\bpsi_\gamma,\xi_\gamma)}_\D^h}{\gamma} 
\intd \left[ \|\bw_\gamma\|^2 - \frac{\e}{2\Wi} 
\pi_h\!\left[\bpsi_\gamma :\Gd'(\bpsi_\gamma) - \xi_\gamma\,\Gd'\left(1-\frac{\xi_\gamma}{b}\right)
\right]\right]
\ddx.
\nonumber
\end{align}
It follows from (\ref{intgam}), (\ref{Entropy2}), 
(\ref{eq:FENEPstrs-above-OBstrs}), and similarly to (\ref{bb1aL}),
on noting (\ref{eq:norm_equivalenceaL}) and (\ref{eq:fixed-pointaL}), that
\begin{align}
&\intd \left[
\|\bw_\gamma\|^2 - \frac{\e}{2\Wi} \pi_h\!\left[\bpsi_\gamma :\Gd'(\bpsi_\gamma)
- \xi_\gamma\,\Gd'\left(1-\frac{\xi_\gamma}{b}\right)
\right] 
\right] \ddx 
 \label{bb2aL}
\\
&\hspace{0.7in}=
\intd \left[
\|\bw_\gamma\|^2 + \frac{\e}{2\Wi} \pi_h\!\left[\bpsi_\gamma :(\I-\Gd'(\bpsi_\gamma))
+ \xi_\gamma\,\left(\Gd'\left(1-\frac{\xi_\gamma}{b}\right)-1\right)
\right] 
\right]\ddx
\nonumber \\
&\hspace{0.7in}\ge  
\intd \left[
\|\bw_\gamma\|^2 + \frac{\e}{4\Wi} \left[\pi_h[\,\|\bpsi_\gamma\|
+|\xi_\gamma|
\,] - 2(d+b) \right]\right]
\ddx
\nonumber
\\
& \hspace{0.7in}
\ge \min\brk{1,\frac{\e}{4\Wi\,\mu_h \gamma}}\gamma^2 - \frac{\e (d+b)|\D|}{2\Wi}. 
\nonumber
\end{align}
Therefore on combining (\ref{eq:whereasaL}) and (\ref{bb2aL}), we have for all 
$\gamma$ sufficiently large that 
\begin{align} \label{eq:other-handaL}
\Scal{\mathcal{F}^h(\bw_\gamma,\bpsi_\gamma,\xi_\gamma)}{\brk{\bw_\gamma,-\frac{\e}{2\Wi}
\pi_h[\Gd'(\bpsi_\gamma)],
\frac{\e}{2\Wi}
\pi_h\!\left[\Gd'\left(1-\frac{\xi_\gamma}{b}\right)\right]
}}^h_D < 0 \,,
\end{align}
which obviously contradicts~\eqref{eq:one-handaL}.
Hence the mapping $\mathcal{F}^h$ has a zero, $(\buhdLa^{n},\strhdLa^{n},\varrhohda^{n}) 
\in \mathcal{B}_\gamma^h$
for $\gamma$ sufficiently large. 
Finally, similarly to (\ref{intgam}), it follows,
on choosing $(\bv,\bphi,\eta)=(\bzero,\I,-1)$
in
$\Scal{\mathcal{F}^h(\buhdLa^{n},\strhdLa^{n},\varrhohda^{n})}{(\bv,\bphi,\eta)}_\D^h=0$,
that $\intd [ \tr(\strhdLa^n)-\varrhohda^n]\,\ddx
=\intd [ \tr(\strhdLa^{n-1})-\varrhohda^{n-1}]\,\ddx=0$.
\end{proof}

We now have the analogue of Theorem \ref{dstabthm}.

\begin{theorem} 
\label{dstabthmaorL}
For any $\delta \in (0,\min\{\frac{1}{2},b\}]$, 
$N_T \geq 1$ and any
partitioning of $[0,T]$ into $N_T$
time steps, 
there exists a solution  
$\{(\buhdLa^{n},\strhdLa^{n},\varrhoadh^n)\}_{n=1}^{N_T}
\in [\Vhone \times \Sh^1\times \Qhone]^{N_T}
$  
to {\bf (P$^{\dt}_{\alpha,\delta,h}$)}, (\ref{eq:PdLaha}--c).

In addition, it follows for $n=1,\ldots,N_T$ that
$\intd [ \tr(\strhdLa^n)-\varrhohda^n]\,\ddx=0$ and
\begin{align}
&\Fdh(\buhdLa^{n},\strhdLa^{n},\varrhohda^n) 
+ \frac{1}{2} \sum_{m=1}^n
\int_\D \left[ {\rm Re} \|\buhdLa^{m}-\buhdLa^{m-1}\|^2
+ (1-\e)\dt_{m}\|\gbuhdLa^{m}\|^2
\right]\ddx
\label{Fstab1aL}\\
& \qquad
+\frac{\e}{2{\rm Wi}^2}\sum_{m=1}^n \Delta t_m \intd
\pi_h\left[\tr\left(\left( 
\Ad(\strhdLa^m, \varrhohda^m)
\right)^2
\Bd(\strhdLa^{m})\right)
\right]\ddx
\nonumber \\
& \qquad
+ \frac{\alpha\e\delta^2}{2{\rm Wi}} \sum_{m=1}^n \Delta t_m
\intd \left[ \|\grad\pi_h[\Gd'(\strhdLa^{m})]\|^2 + b\left\| 
\grad\pi_h\left[\Gd'\left(1-\frac{\varrhohda^{m}}{b}\right)\right]\right\|^2 \right]
\ddx
\nonumber \\
& \hspace{1in}
\leq
\Fdh(\buh^0,\strh^0,\tr(\strh^0))
+ \frac{1+C_P}{2(1-\e)} \sum_{m=1}^n \dt_m \|\f^m\|_{H^{-1}(\D)}^2
\leq C,
\nonumber 
\end{align}
which yields that
\begin{align}
\label{Fstab2aL}
\max_{n=0, \ldots, N_T} \int_\D \left[\, \|\buhdLa^n\|^2 + \|\strhdLa^n\| +
|\varrhohda^n|+ 
\delta^{-1}\,\pi_h\left[\,\|[\strhdLa^n]_{-}\|
+ \left|[b-\varrhohda^n]_{-}\right|\,\right] 
\,\right]\,\ddx
&\leq C;
\end{align}
where $C$ is independent of $\alpha$, as well as $\delta$, $h$ and $\Delta t$.
 \end{theorem}
\begin{proof}
Existence and the stability result (\ref{Fstab1aL}) follow 
immediately from Propositions \ref{prop:existence-PdLah}
and \ref{prop:free-energy-PdLah}, respectively,
on noting 
(\ref{eq:free-energy-PdLah}), (\ref{eq:Gd}), 
(\ref{idatah}), (\ref{idatahspd}), (\ref{fncont}) and (\ref{freg}).
The bounds (\ref{Fstab2aL}) follow immediately from 
(\ref{Fstab1aL}), on noting (\ref{Addef}), 
(\ref{eq:positive-term}), (\ref{eq:free-energy-PdLah}), 
(\ref{Entropy2}) and that
\begin{align} 
\label{Fstab2aLpr}
& \intd \pi_h\left[
b \Brk{ \left(1-\frac{\varrhohda^n}{b}\right)-
 \Gd\brk{1-\frac{\varrhohda^n}{b}}}
+ \tr(\strhdLa^n-\Gd(\strhdLa^n))
\right]\ddx
\\
&\qquad= 
\intd \pi_h\left[
b \Brk{1-
 \Gd\brk{1-\frac{\varrhohda^n}{b}}}
- \tr(\Gd(\strhdLa^n))
\right]\ddx
\leq C.
\nonumber 
\end{align}
\end{proof}

\begin{remark}\label{remPdah}
We recall that we have used $\Shone$ for the approximation of $\strad$ in 
{\bf (P$_{\alpha,\delta,h}^{\Delta t}$)}, (\ref{eq:PdLaha}--c), 
due to the presence of the diffusion term in (\ref{eq:aoldroyd-b-sigma2d}).    
Secondly, due to the advective term in (\ref{eq:aoldroyd-b-sigma2d}), one has to  
introduce the variable $\varrhoadh^n$ and its equation (\ref{eq:PdLahc}) 
in {\bf (P$_{\alpha,\delta,h}^{\Delta t}$)}
in order to 
obtain the entropy bound (\ref{eq:estimate-PdLah}).
However, we now have a bound on $\pi_h[[b-\varrhoadh^n]_-]$ in
(\ref{Fstab2aL}), as opposed to $[b-\tr(\strhd^n)]_-$ in
(\ref{Fstab2}). Now, it does not seem possible to pass to the limit $\delta \rightarrow 0$
in {\bf (P$_{\alpha,\delta,h}^{\Delta t}$)} to prove well-posedness of the corresponding
direct approximation of {\bf (P$_\alpha$)}, i.e.\ {\bf (P$_{\alpha,h}^{\Delta t}$)}
without the regularization $\delta$,
as we did for {\bf (P$_{\delta,h}^{\Delta t}$)} in subsection 
\ref{secPdhtoPd}.
\end{remark}

Finally, we note the following Lemmas for later purposes.
\begin{lemma}\label{lem:usefulinv}
For all $K_k \in {\mathcal T}_h$
and $\bphi \in [C(\overline{K_k})]^{d \times d}_S$, we have, for $r\in [1,\infty)$, that 
\begin{align}
\int_{K_k} \left[ \pi_h\left[ \|\bphi\|^r \right] +\left|\pi_h\left[ \|\bphi\| \right]\right|^r \right]\ddx
&\leq C \int_{K_k} \|\pi_h \bphi\|^r \,\ddx.
\label{use1} 
\end{align}
\end{lemma}
\begin{proof}
It follows immediately from (\ref{inverse}) that
\begin{align*}
\int_{K_k}  \left[\pi_h\left[ \|\bphi\|^r \right] 
+\left|\pi_h\left[ \|\bphi\| \right]\right|^r \right]
\ddx \leq 2\,|K_k|\, \| \pi_h \bphi\|_{L^\infty(K_k)}^r
\leq C \int_{K_k} \|\pi_h \bphi\|^r \,\ddx,
\end{align*}
and hence the desired result(\ref{use1}).
\end{proof}

\begin{lemma} 
\label{lemMXitt}
Let $g \in C^{0,1}({\mathbb R})$ with Lipschitz constant $g_{\rm Lip}$.
For all $K_k \in {\mathcal T}_h$,
and for all $q \in \Qhone$, $\bphi \in \Shone$ we have that
\begin{subequations}
\begin{align}
&\int_{K_k}\!
\|\pi_h[\Bd(\bphi)]-\Bd(\bphi)\|^2 \,\ddx
+ \max_{m,p=1,\ldots,d}
\int_{K_k}\!
\|\Lambda_{\delta,m,p}(\bphi)- \Bd(\bphi)\,\delta_{mp}\|^2
\,\ddx 
\label{MXittxbd}
\\
& \hspace{3.5in}
\leq C  
\,h^2\int_{K_k} \!\|\grad \bphi\|^2 \,\ddx,
\nonumber
\\
&\int_{K_k}
\|\pi_h[g(q)]-g(q)\|^2\,\ddx 
\leq C \,g_{\rm Lip}^2  
\,h^2\int_{K_k} \|\grad q\|^2
\,\ddx \label{MXittxbdq}\\
&\hspace{1.5in} \mbox{and} \quad
\int_{K_k}
\|\pi_h[g(\bphi)]-g(\bphi)\|^2 \,\ddx
\leq C\,g_{\rm Lip}^2  
\,h^2\int_{K_k} \|\grad \bphi\|^2\,\ddx.
\nonumber
\end{align}
\end{subequations}
In addition, if $g$ is monotonic then, for
all $K_k \in {\mathcal T}_h$
and for all $q \in \Qhone$, we have that
\begin{align}
&\int_{K_k}
\|\pi_h[g(q)]-g(q)\|^2\,\ddx 
\leq C  
\,h^2\int_{K_k} \|\grad \pi_h[g(q)]\|^2
\,\ddx. \label{CME}
\end{align}
\end{lemma}
\begin{proof}
The results (\ref{MXittxbd}) are proved in Lemma 5.3 of \cite{barrett-boyaval-09} for the case
when $\Bd$, $\Lambda_{\delta,m,p}$ and $\Shone$ are replaced by 
$\beta$, $\Lambda_{m,p}$ and $\Shone\cap\SPD$. The proofs given there are trivially adapted 
to the present case. In fact, the proof of the first result in (\ref{MXittxbd}) in 
\cite{barrett-boyaval-09} is easily adapted to any function $g \in C^{0,1}({\mathbb R})$.
Hence, we have the results (\ref{MXittxbdq}).

The result (\ref{CME}) is a simple variation of (\ref{MXittxbdq}) and follows on noting that
\begin{align}
\|\pi_h[g(q)]-g(q)\|_{L^\infty(K_k)} \leq \max_{i,\,j =0,\ldots,P_d^k} 
|g(q(P_i^k)) -g(q(P_j^k))|,
\label{CME1}
\end{align}
where $\{P^k_j\}_{j=0}^d$  are the vertices of $K_k$.
\end{proof}

\section{Convergence of {\bf (P$^{\Delta t}_{\alpha,\delta,h}$)} to
{\bf(P$_{\alpha}$)} in the case $d=2$} 
\label{sec:convergence}
\setcounter{equation}{0}
\subsection{Stability}
Before proving our stability results,
we introduce some further notation. 
We require the $L^2$ projector ${\mathcal R}_h : \Uz \rightarrow \Vhone$
defined by
\begin{align}
\intd (\bv -{\mathcal R}_h\bv)\cdot\bw \,\ddx=0 \qquad \forall \bw \in \Vhone. 
\label{Rh}
\end{align}
In addition, we require  
${\mathcal P}_h : [L^2(\D)]^{d \times d}_{\rm S} \rightarrow \Shone$
defined by
\begin{align}
\intd \pi_h[{\mathcal P}_h\bchi :\bphi]\,\ddx =
\intd \bchi : \bphi \,\ddx \qquad \forall \bphi \in \Shone. 
\label{Ph}
\end{align}
It is easily deduced for $p =1,\ldots,N_P$ and $i,\,j
= 1, \ldots, d$ that
\begin{align}
[{\mathcal P}_h \bchi]_{ij}(P_p) = \frac{1}{\intd \eta_p} 
\intd [{\mathcal P}_h \bchi]_{ij}\,\eta_p\,\ddx,
\label{strh01def}
\end{align}
where $\eta_p \in {\rm Q}_h^1$ is such that $\eta_p(P_r)=\delta_{pr}$
for $p,\,r = 1,\ldots,N_P$.
It follows from
(\ref{Ph}) and
(\ref{interpinf}) with $\bphi= {\mathcal P}_h \bchi$, in both cases, that 
\begin{align}
\intd \|{\mathcal P}_h \bchi\|^2 \,\ddx\leq \intd \pi_h[\,\|{\mathcal P}_h \bchi\|^2]\,\ddx
\leq \intd \|\bchi\|^2 \,\ddx\qquad \forall \bchi \in [L^2(\D)]^{d \times d}_{\rm S}.
\label{PhstabL2}
\end{align}

We shall assume from now on that $\D$ is convex and that 
the family $\{{\mathcal T}_h\}_{h>0}$ is quasi-uniform,
i.e.\ $h_k \geq C\,h$, $k=1,\ldots, N_K$.  
It then follows that
\begin{align}
\|{\mathcal R}_h \bv\|_{H^1(\D)} \leq C \|\bv\|_{H^1(\D)}
\qquad \forall \bv \in \Uz,  
\label{Rhstab}
\end{align}
see Lemma 4.3 in Heywood and Rannacher \cite{HR}. Similarly, it is easily established that
\begin{align}
\|{\mathcal P}_h \bchi\|_{H^1(\D)} \leq C \|\bchi\|_{H^1(\D)}
\qquad \forall \bchi \in [H^1(\D)]^{d \times d}_{\rm S}.  
\label{Phstab}
\end{align}
We also require the scalar analogue of ${\mathcal P}_h$,
where $d=1$ and $\Shone$ is replaced by $\Qhone$,
satisfying the analogues of (\ref{Ph})--(\ref{PhstabL2})
and (\ref{Phstab}).

Let $([H^1(\D)]^{d \times d}_{\rm S})'$ be the topological dual of
  $[H^1(\D)]^{d \times d}_{\rm S}$ with $[L^2(\D)]^{d \times d}_{\rm S}$
being the pivot space.
Let ${\mathcal E} : ([H^1(\D)]^{d \times d}_{\rm S})' \rightarrow [H^1(\D)]^{d \times d}_{\rm S}$
be such that ${\mathcal E} \bchi$ is the unique solution of the Helmholtz
problem
\begin{align}
\intd \left [ \grad ({\mathcal E} \bchi) ::
\grad \bphi + ({\mathcal E} \bchi) : \bphi \right] \,\ddx=
\langle \bchi,\bphi\rangle_{H^1(\D)} \qquad \forall \bphi \in [H^1(\D)]^{d \times d}_{\rm S},
\label{Echi}
\end{align}  
where $\langle \cdot,\cdot \rangle_{H^1(\D)}$ denotes the duality pairing 
between $([H^1(\D)]^{d \times d}_{\rm S})'$ and $[H^1(\D)]^{d \times d}_{\rm S}$. We note that
\begin{align}
\langle \bchi,{\mathcal E}\bchi\rangle_{H^1(\D)} =\|{\mathcal E} \bchi\|_{H^1(\D)}^2
\qquad \forall \bchi \in ([H^1(\D)]^{d \times d}_{\rm S})',
\label{Echinorm}
\end{align} 
and $\|{\mathcal E}\cdot\|_{H^1(\D)}$ is a norm on $([H^1(\D)]^{d \times d}_{\rm S})'$.
We also employ this operator in the scalar case, $d=1$, i.e.\ ${\mathcal E}
: H^1(\D)' \rightarrow H^1(\D)$.

Let $\Uz'$ be the topological dual of $\Uz$
with the space of weakly divergent free functions in $[L^2(\D)]^d$
being the pivot space.
Let ${\mathcal S} : \Uz' \rightarrow
\Uz$ be such that ${\mathcal S} \bw$ is the unique solution to the Helmholtz-Stokes
problem
\begin{align}
\intd \left [ \grad ({\mathcal S} \bw) :
\grad \bv + ({\mathcal S} \bw) \cdot \bv \right]\,\ddx=
\langle \bw,\bv\rangle_{\Uz} \qquad \forall \bv \in \Uz,
\label{Sw}
\end{align}  
where $\langle \cdot,\cdot \rangle_{\Uz}$ denotes the duality pairing 
between $\Uz'$ and $\Uz$. We note that
\begin{align}
\langle \bw,{\mathcal S}\bw\rangle_{\Uz} =\|{\mathcal S} \bw\|_{H^1(\D)}^2
\qquad \forall \bw \in \Uz',
\label{Swnorm}
\end{align} 
and $\|{\mathcal S}\cdot\|_{H^1(\D)}$ is a norm on the reflexive space $\Uz'$.

We recall the following well-known Gagliardo-Nirenberg inequality. 
Let $r \in [2,\infty)$ if $d=2$, and $r \in [2,6]$ if $d=3$ and $\theta =
d(\frac{1}{2}-\frac{1}{r})$. Then, there exists a positive constant $C(\D,r,d)$
such that
\begin{align}
\|\eta\|_{L^r(\D)} \leq C(\D,r,d) \|\eta\|_{L^2(\D)}^{1-\theta} 
\|\eta\|_{H^1(\D)}^\theta \qquad \forall \eta \in H^1(\D).
\label{GN}
\end{align}

We recall also the following compactness results.
Let ${\mathcal Y}_0$, ${\mathcal Y}$ and
${\mathcal Y}_1$ be real Banach
spaces, ${\mathcal Y}_i$, $i=0,1$, reflexive, with a compact embedding 
${\mathcal Y}_0
\hookrightarrow {\mathcal Y}$ and a continuous embedding ${\mathcal Y} \hookrightarrow
{\mathcal Y}_1$. 
Then, for $\mu_i>1$, $i=0,1$, the following embedding is compact:
\begin{align}
&\{\,\eta \in L^{\mu_0}(0,T;{\mathcal Y}_0): \frac{\partial \eta}{\partial t}
\in L^{\mu_1}(0,T;{\mathcal Y}_1)\,\} \hookrightarrow L^{\mu_0}(0,T;{\mathcal Y});
\label{compact1}
\end{align}
see Theorem 2.1 on p184 in Temam \cite{Temam}. 
Let ${\mathcal X}_0$, ${\mathcal X}$ and
${\mathcal X}_1$ be real Hilbert
spaces with a compact embedding 
${\mathcal X}_0
\hookrightarrow {\mathcal X}$ and a continuous embedding ${\mathcal X} \hookrightarrow
{\mathcal X}_1$. 
Then, for $\gamma>0$, the following embedding is compact:
\begin{align}
&\{\,\eta \in L^{2}(0,T;{\mathcal X}_0): D^\gamma_t \eta
\in L^{2}(0,T;{\mathcal X}_1)\,\} \hookrightarrow L^{2}(0,T;{\mathcal X}),
\label{compact2}
\end{align}
where $D^\gamma_t \eta$ is the time derivative of order $\gamma$ of $\eta$,
which can be defined in terms of the Fourier transform of $\eta$;
see Theorem 2.2 on p186 in Temam \cite{Temam}.  

Finally, we recall the discrete Gronwall inequality:
\begin{alignat}{2}
\label{DG}
(r^0)^2 +(s^0)^2 &\leq (q^0)^2, 
\\
(r^m)^2 +(s^m)^2 &\leq \sum_{n=0}^{m-1} (\eta^n)^2 (r^n)^2 + \sum_{n=0}^m
(q^n)^2 \qquad &&m\geq 1
\nonumber
\\
\Rightarrow \qquad (r^m)^2 +(s^m)^2 &\leq \exp( \sum_{n=0}^{m-1} (\eta^n)^2)
\sum_{n=0}^m (q^n)^2 \qquad &&m\geq1.
\nonumber 
\end{alignat}

\begin{theorem}
\label{dstabthmaL}
Under the assumptions of Theorem \ref{dstabthmaorL},
there exists a solution 
$\{(\buadh^{n},\stradh^{n},$\linebreak$\varrhoadh^n)\}_{n=1}^{N_T} 
\in [\Vhone \times \Shone \times \Qhone]^{N_T}$ 
of {\bf (P$^{\dt}_{\alpha,\delta,h}$)}, (\ref{eq:PdLaha}--c),
such that the bounds (\ref{Fstab1aL}) and (\ref{Fstab2aL}) hold.

Moreover, if $d=2$, (\ref{Deltatqu}) holds 
and $\Delta t \leq C_\star(\zeta^{-1})\,\alpha^{1+\zeta}\,h^2$,
for a $\zeta >0$ and a $C_\star(\zeta^{-1}) \in \mathbb{R}_{>0}$ sufficiently small,
then the following bounds hold:
\begin{subequations}
\begin{align}
\label{Fstab3aL}
&\max_{n=0, \ldots, N_T} \intd
\pi_h[\,\|\stradh^{n}\|^2\,]\,\ddx
+  \sum_{n=1}^{N_T} \intd \left[\dt_n\, \alpha \|\grad \stradh^{n}\|^2 
+ \pi_h[\,\|\stradh^{n}-\stradh^{n-1}\|^2\,]
\right] \ddx\leq C,\\
\label{Fstab3aLrho}
&\max_{n=0, \ldots, N_T} \intd
\pi_h[\,|\varrhoadh^{n}|^2\,]\,\ddx
+  \sum_{n=1}^{N_T} \intd \left[\dt_n\, \alpha \|\grad \varrhoadh^{n}\|^2 
+ \pi_h[\,|\varrhoadh^{n}-\varrhoadh^{n-1}|^2\,]
\right] \ddx\leq C,\\
\label{FstabLam}
&\sum_{n=1}^{N_T} \dt_n
\sum_{m=1}^2 \sum_{p=1}^2
\left[
\left \|\Lambda_{\delta,m,p}(\stradh^{n}) \right\|_{L^4(\D)}^4
+ \left \|\Lambda_{\delta,m,p}\left(1-\frac{\varrhoadh^{n}}{b}\right) \right\|_{L^4(\D)}^4
\right] \ddx\leq C.
\end{align}
\end{subequations}
\end{theorem}
\begin{proof}
Existence  
and the bounds
(\ref{Fstab1aL}) and (\ref{Fstab2aL})
were proved in Theorem
\ref{dstabthmaorL}.

On choosing $\bphi = \stradh^{n}$ 
in  
(\ref{eq:PdLahb}), 
it follows
from 
(\ref{eq:BGd}),
and (\ref{elemident}),   
on applying a Young's inequality, for any $\zeta >0$ that 
\begin{align}
\label{sigstaba}
&\frac{1}{2}\intd \pi_h[ \,\|\stradh^{n}\|^2+\|\stradh^{n}-
\stradh^{n-1}\|^2\,]\,\ddx + \dt_n \,\alpha \intd \|\grad \stradh^{n}\|^2
\,\ddx\\
& \quad + \frac{\dt_n}{\Wi}
\intd \pi_h\left[
\Gd'\left(1-\frac{\varrhoadh^n}{b}\right)\,\tr\left((\Bd(\stradh^n))^2\right) 
\right]\ddx
\nonumber\\
&\qquad \leq \frac{1}{2}\intd \pi_h[ \,\|\stradh^{n-1}\|^2\,]\,\ddx 
+ 
2\dt_n
\intd \grad \buadh^{n} : 
\pi_h[\kd(\stradh^{n}, \varrhoadh^n)\,\stradh^{n}\,\Bd(\stradh^{n})]\,\ddx
\nonumber\\
& \qquad \quad
+ \frac{\dt_n}{\Wi}
\intd \pi_h\left[\Gd'\left(1-\frac{\varrhoadh^n}{b}\right)
\tr\left(\Bd(\stradh^n) \,(\Bd(\stradh^n)-\stradh^n)\right) \right]\ddx
\nonumber
\\
&\qquad \quad + 
\frac{\dt_n}{\Wi} 
\,\|\stradh^n\|_{L^1(\D)}
+ \dt_n
\int_\D \sum_{m=1}^d \sum_{p=1}^d 
[\buadh^{n-1}]_m \,\Lambda_{\delta,m,p}(\stradh^{n})
: \frac{\partial \stradh^n}{\partial \xx_p}\,\ddx
\nonumber \\
&\qquad \leq \frac{1}{2} \intd \pi_h[\, \|\stradh^{n-1}\|^2\,]\,\ddx 
+\frac{\dt_n}{\Wi} 
\,\|\stradh^n\|_{L^1(\D)}
\nonumber \\
& \qquad \quad
+ C\,\delta^{-1}\,\dt_n
\intd \pi_h\left[
\tr\left(\Bd(\stradh^n)\left|\Bd(\stradh^n)-\stradh^n\right|\right) 
\right]\,\ddx
\nonumber \\
& \qquad \quad
+2\,\dt_n \,\|\grad \buadh^{n}\|_{L^2(\D)} 
\,
\|\pi_h[\kd(\stradh^{n}, \varrhoadh^n)\,\stradh^{n}\,\Bd(\stradh^n)]\|_{L^2(\D)}
\nonumber \\
& \qquad \quad 
+ C\,\dt_n\,\|\buadh^{n-1}\|_{L^{\frac{2(2+\zeta)}{\zeta}}(\D)}\,
\|\Lambda_{\delta,m,p}(\stradh^{n})\|_{L^{2+\zeta}(\D)}\,
\|\grad \stradh^{n}\|_{L^2(\D)}.
\nonumber 
\end{align}
We deduce from   
(\ref{eq:Bd}), (\ref{modphisq}), (\ref{interpinf}), 
(\ref{normprod}), (\ref{kddef}), (\ref{BdLdef}), (\ref{use1}) and (\ref{GN}), as $d=2$,
that
\begin{subequations}
\begin{align}
&\tr\left(\Bd(\stradh^n)\left|\Bd(\stradh^n)-\stradh^n\right|\right)
\leq \delta\, \tr\left(\left|\Bd(\stradh^n)-\stradh^n\right|\right)
\leq C\,\delta\,(\delta + \|\stradh^n\|),
\label{deltr}\\
&\|\pi_h[\kd(\stradh^{n}, \varrhoadh^n)\,\stradh^{n}\,\Bd(\stradh^n)]\|_{L^2(\D)}^2 
\label{deltnorm}
\\
& \qquad \leq 
\intd  \pi_h\left[\,\|\kd(\stradh^{n}, \varrhoadh^n)\,\stradh^{n}\,\Bd(\stradh^n)\|^2\right]
\,\ddx \nonumber \\
& \qquad \leq C\,b
\intd  \pi_h\left[\,\|\stradh^{n}\|^2\,\|\Bd(\stradh^n)\|\right]\,\ddx
\nonumber
\leq C\left( \delta^3 + \intd \pi_h[\,\|\stradh^{n}\|^3\,] \,\ddx\right)
\\
& \qquad \leq C\left( \delta^3 + \|\stradh^{n}\|_{L^3(\D)}^3 \right)
\leq C\,\left( \delta^3 + \|\stradh^{n}\|_{L^2(\D)}^2\,\|\stradh^{n}\|_{H^1(\D)}
\right).
\nonumber 
\end{align} 
\end{subequations}
Similarly, as $d=2$, it follows from (\ref{LdmpLinf}),
(\ref{eq:Bd}), (\ref{inverse}) 
and (\ref{GN}) that
for all $\zeta >0$ 
\begin{align}
\label{LambdaL2}
 \|\Lambda_{\delta,m,p}(\stradh^{n})\|_{L^{2+\zeta}(\D)}^{2+\zeta}
&\leq \sum_{k=1}^{N_K} |K_k|\, \|\Lambda_{\delta,m,p}(\stradh^{n})\|_{L^\infty(K_k)}^{2+\zeta}
\leq C\,\sum_{k=1}^{N_K} |K_k|\, \|\pi_h[\Bd(\stradh^{n})]\|_{L^\infty(K_k)}^{2+\zeta}
\\
&\leq C\left[\delta^{2+\zeta} + \|\stradh^{n}\|_{L^{2+\zeta}(\D)}^{2+\zeta} \right]
\leq C + C(\zeta)\,\|\stradh^{n}\|_{L^2(\D)}^2\,\|\stradh^{n}\|_{H^1(\D)}^\zeta.
\nonumber
\end{align}
In addition, as $d=2$, we note from (\ref{GN}), (\ref{eq:poincare}) and (\ref{Fstab2aL}) that
for all $\zeta>0$
\begin{align}
\|\buadh^{n-1}\|_{L^{\frac{2(2+\zeta)}{\zeta}}(\D)}
\leq C(\zeta^{-1})\, \|\buadh^{n-1}\|_{L^{2}(\D)}^{\frac{\zeta}{2+\zeta}}
\,\|\buadh^{n-1}\|_{H^{1}(\D)}^{\frac{2}{2+\zeta}}
\leq C(\zeta^{-1})\,
\|\grad \buadh^{n-1}\|_{L^{2}(\D)}^{\frac{2}{2+\zeta}}\,.
\label{uzeta}
\end{align}

Combining (\ref{sigstaba})--(\ref{uzeta}), yields, 
on 
applying a Young's inequality, that
for all $\zeta >0$
\begin{align}
&\intd \pi_h[ \,\|\stradh^{n}\|^2+ \|\stradh^{n}-
\stradh^{n-1}\|^2\,] \,\ddx+ \dt_n\, \alpha \intd \|\grad \stradh^{n}\|^2\,\ddx
\label{sigstabb}
\\
& \quad
\leq 
\intd \pi_h[ \,\|\stradh^{n-1}\|^2\,]\,\ddx 
\nonumber \\
& \qquad 
+ C(\zeta^{-1})\,\dt_n \,\alpha^{-(1+\zeta)}
\left[ 1 
+ \|\grad \buadh^{n}\|_{L^2(\D)}^2\,
+ \|\grad \buadh^{n-1}\|_{L^2(\D)}^2 \right]
\left(1+
\|\stradh^{n}\|^2_{L^2(\D)}\right).
\nonumber
\end{align}
Hence, summing (\ref{sigstabb}) from $n=1,\ldots,m$ for $m=1, \ldots,N_T$
yields, on noting 
(\ref{interpinf}), that
for any $\zeta >0$
\begin{align}
\label{sigstabc}  
&\intd \pi_h[ \,\|\stradh^{m}\|^2\,]\,\ddx
+ \alpha \,\sum_{n=1}^m 
\dt_n 
\intd 
 \|\grad \stradh^{n}\|^2\,\ddx
+ \sum_{n=1}^m 
\intd
\pi_h[\,\|\stradh^{n}-\stradh^{n-1}\|^2\,]\,\ddx 
\\
& \quad \quad
\leq 
\intd \pi_h[\, \|\strh^{0}\|^2\,] \,\ddx+ C(\zeta^{-1})\,\alpha^{-(1+\zeta)}
\nonumber \\
& \quad \quad \quad  + C(\zeta^{-1})\,\alpha^{-(1+\zeta)}
\,\sum_{n=1}^m
\dt_n
\left[1+ \sum_{k=n-1}^n \|\grad \buadh^{k}\|_{L^2(\D)}^2 
\right]
\,
\intd \pi_h[\,\|\stradh^{n}\|^2\,]\,\ddx.
\nonumber
\end{align}
Applying the discrete Gronwall inequality (\ref{DG}) to (\ref{sigstabc}), and noting
(\ref{Deltatqu}), (\ref{eqnorm}), (\ref{idatah}), 
(\ref{Vinverse}),
(\ref{Fstab1aL}),
(\ref{Fstab2aL}), 
 and that $\dt \leq C_\star(\zeta^{-1})
\,\alpha^{1+\zeta}\,h^2$, 
for a $\zeta >0$ where $C_\star(\zeta^{-1})$ is sufficiently small, 
yields the bounds (\ref{Fstab3aL}).

Similarly to (\ref{sigstaba})
on choosing $\eta = \varrhoadh^{n}$ 
in  
(\ref{eq:PdLahc}), 
it follows
from 
(\ref{elemident}), (\ref{ipmat}), (\ref{normphi}) and (\ref{modphisq}),  
on applying a Young's inequality, for any $\zeta >0$ that  
\begin{align}
\label{sigstabarho}
&\frac{1}{2}\intd \pi_h[ \,|\varrhoadh^{n}|^2+|\varrhoadh^{n}-
\varrhoadh^{n-1}|^2\,]\,\ddx + \dt_n \,\alpha \intd \|\grad \varrhoadh^{n}\|^2
\,\ddx
\\
&\quad \leq \frac{1}{2}\intd \pi_h[ \,|\varrhoadh^{n-1}|^2\,]\,\ddx 
+ 2\dt_n
\intd \grad \buadh^{n} : 
\pi_h[\kd(\stradh^n,\varrhoadh^n)\,\varrhoadh^n\,\Bd(\stradh^{n})]\,\ddx
\nonumber
\\
&\qquad \qquad 
- \frac{\dt_n}{\Wi}
\intd \pi_h\left[\tr\left(
\Ad(\stradh^n,\varrhoadh^n)\,
\Bd(\stradh^n)\right) \varrhoadh^n \right]\,\ddx
\nonumber \\
& \qquad \qquad - \dt_n\,b
\int_\D \sum_{m=1}^d \sum_{p=1}^d 
[\buadh^{n-1}]_m \,\Lambda_{\delta,m,p}\left(1-\frac{\varrhoadh^{n}}{b}\right)
\frac{\partial \varrhoadh^n}{\partial \xx_p}\,\ddx
\nonumber \\
&\quad \leq \frac{1}{2} \intd \pi_h[\, |\varrhoadh^{n-1}|^2\,]\,\ddx
+ 2\,\dt_n \,\|\grad \buadh^{n}\|_{L^2(\D)} 
\,
\|\pi_h[\kd(\stradh^n,\varrhoadh^n)\,\varrhoadh^n\,\Bd(\stradh^n)]\|_{L^2(\D)}
\nonumber \\
& \qquad \qquad 
+ C\,\dt_n\,\|\buadh^{n-1}\|_{L^{\frac{2(2+\zeta)}{\zeta}}(\D)}\,
\left\|\Lambda_{\delta,m,p}\left(1-\frac{\varrhoadh^{n}}{b}\right)\right\|_{L^{2+\zeta}(\D)}
\|\grad \varrhoadh^{n}\|_{L^2(\D)}
\nonumber \\
& \qquad \qquad 
+ \frac{\dt_n}{2\Wi}
\intd \pi_h\left[\tr\left(\left(
\Ad(\stradh^n,\varrhoadh^n)
\right)^2\Bd(\stradh^n)\right)+ d^{\frac{1}{2}}(\varrhoadh^n)^2 \|\Bd(\stradh^n)\|
\right]\ddx.
\nonumber 
\end{align}
Similarly to (\ref{deltnorm}), as $d=2$, we deduce from 
(\ref{interpinf}), (\ref{kddef}), (\ref{BdLdef}),
(\ref{eq:Bd}), (\ref{inverse}), (\ref{GN}) and (\ref{Fstab3aL})  
that
\begin{align}
\| \pi_h[\kd(\stradh^n,\varrhoadh^n)\,\varrhoadh^n \,\Bd(\stradh^n)] \|_{L^2(\D)}^2
\!\!
&\leq C\,b\,\intd \pi_h\left[ (\varrhoadh^n)^2\,\|\Bd(\stradh^n)\|\right]
\,\ddx
\label{rhoL4a}\\
&\leq C\left( 1+ \| \stradh^n \|_{L^3(\D)}\right)\| \varrhoadh^n\|_{L^3(\D)}^2 
\nonumber \\
&\leq C\left( 1+ \| \stradh^n \|_{H^1(\D)}^{\frac{1}{3}}\right)
\| \varrhoadh^n\|_{L^2(\D)}^{\frac{4}{3}}
\| \varrhoadh^n\|_{H^1(\D)}^{\frac{2}{3}}.
\nonumber 
\end{align}
Similarly to (\ref{LambdaL2}), as $d=2$, it follows from (\ref{LdmpLinfq}), 
(\ref{eq:Bd}), (\ref{inverse}) 
and (\ref{GN}) that
for all $\zeta >0$ 
\begin{align}
\label{LambdaL2rho}
 \left \|\Lambda_{\delta,m,p}\left(1-\frac{\varrhoadh^{n}}{b}
 \right)\right \|_{L^{2+\zeta}(\D)}^{2+\zeta}
&
\leq C + C(\zeta)\,\|\varrhoadh^{n}\|_{L^2(\D)}^2\,\|\varrhoadh^{n}\|_{H^1(\D)}^\zeta.
\end{align}

Combining (\ref{sigstabarho})--(\ref{LambdaL2rho}) and (\ref{uzeta}), yields, 
on    
applying a Young's inequality, that
for all $\zeta >0$
\begin{align}
&\intd \pi_h[ \,|\varrhoadh^{n}\|^2+ |\varrhoadh^{n}-
\varrhoadh^{n-1}|^2\,]\,\ddx + \dt_n\, \alpha \intd \|\grad \varrhoadh^{n}\|^2
\,\ddx
\label{sigstabbrho}
\\
& \quad
\leq 
\intd \pi_h[ \,|\varrhoadh^{n-1}|^2\,]\,\ddx + 2\, \dt_n\,
\|\grad \buadh^{n}\|_{L^2(\D)}^2
\nonumber \\
&\qquad 
+ \frac{\dt_n}{\Wi}
\intd \pi_h\left[
\tr\left(\left(
\Ad(\stradh^n,\varrhoadh^n)\right)^2
\Bd(\stradh^n)\right) 
\right]\ddx
\nonumber \\
& \qquad + C(\zeta^{-1})\,\dt_n \,\alpha^{-(1+\zeta)}
\left[ 1+ \|\stradh^{n}\|_{H^1(\D)}^2\,
+ \|\grad \buadh^{n-1}\|_{L^2(\D)}^2 \right]
\left(1+
\|\varrhoadh^{n}\|^2_{L^2(\D)}\right).
\nonumber
\end{align}
Hence, summing (\ref{sigstabbrho}) from $n=1,\ldots,m$ for $m=1, \ldots,N_T$
yields, on noting 
(\ref{interpinf}) and (\ref{Fstab1aL}), 
that for any $\zeta >0$
\begin{align}
\label{sigstabcrho}  
&\intd \pi_h[ \,|\varrhoadh^{m}|^2\,]\,\ddx
+ \alpha \,\sum_{n=1}^m 
\dt_n 
\intd 
 \|\grad \varrhoadh^{n}\|^2\,\ddx
+ \sum_{n=1}^m 
\intd
\pi_h[\,|\varrhoadh^{n}-\varrhoadh^{n-1}\|^2\,]\,\ddx 
\\
& \;
\leq 
\intd \pi_h[\, 
(\tr(\strh^0))^2\,]\,\ddx + C(\zeta^{-1})\,\alpha^{-(1+\zeta)}
\nonumber \\
& \quad   + C(\zeta^{-1})\,\alpha^{-(1+\zeta)}
\,\sum_{n=1}^m
\dt_n
\left[1+ \|\stradh^n\|_{H^1(\D)}^2 + \|\grad \buadh^{n-1}\|_{L^2(\D)}^2
\right]
\,
\intd \pi_h[\,|\varrhoadh^{n}|^2\,]\,\ddx.
\nonumber
\end{align}
Applying the discrete Gronwall inequality (\ref{DG}) to (\ref{sigstabcrho}), and noting
(\ref{Deltatqu}), (\ref{idatahspd}), 
(\ref{Sinverse},b),
(\ref{Fstab1aL}),
(\ref{Fstab2aL}), (\ref{eqnorm}), (\ref{Fstab3aL}),
 and that $\dt \leq C_\star(\zeta^{-1})
\,\alpha^{1+\zeta}\,h^2$, 
for a $\zeta >0$ where $C_\star(\zeta^{-1})$ is sufficiently small, 
yields the bounds (\ref{Fstab3aLrho}).

Finally, the desired result (\ref{FstabLam}) follows immediately from 
(\ref{LambdaL2}) and (\ref{LambdaL2rho}) with $\zeta=2$,
on noting (\ref{Fstab3aL},b) and (\ref{eqnorm}). 
\end{proof}

\begin{remark} \label{remd2}
Our final convergence result will be   restricted to $d=2$ for the same reason as
why our result for Oldroyd-B in \cite{barrett-boyaval-09} was
restricted to $d=2$. For example, the control of the term
(\ref{deltnorm}) necessitates the restriction to $d=2$.
This could be overcome for FENE-P
by replacing the regularization $\Bd$ in {\bf (P$_{\alpha,\delta}$)},
(\ref{eq:aoldroyd-b-sigmad}--h),
by $\Bdb$ and using test functions based on $\Gd^b$ in place of $\Gd$, where 
$(\Gd^b)'(s)=\Bdb(s)$ for all $s \in \R$.
On making a similar change to our numerical approximation {\bf (P$_{\alpha,\delta}^{h,\Delta t}$)},
(\ref{eq:PdLaha}--c), 
it is then possible to prove the analogues of Theorem \ref{dstabthmaorL}
and Theorem \ref{dstabthmaL} for $d=2$ and $3$ with no restriction on $\Delta t$.       
One can then establish analogues of Lemmas \ref{dstabthmaLex}, \ref{dstabthmaLexF}
and Theorem \ref{convaL}, below, but now the limits involve the cut-off $b$, i.e. 
$\bu_\alpha^b$, $\stra^b$ and $\varrho_\alpha^b$.
Moreover, it does not seem possible
to establish that $\varrho_\alpha^b=\tr(\stra^b)$, as we now have
$\Lambda_{\delta,m,p}^b$ in place of $\Lambda_{\delta,m,p}$ in the 
analogue of (\ref{eqvarrhocondif}). Hence, in taking the limit $\delta,\,h,\,\Delta t
\rightarrow 0_+$,
the last term in the analogue of (\ref{eqdifcon}) would have
$b\, \beta^b\left(1-\frac{\varrho_\alpha^b}{b}\right) + \tr(\beta^b(\stra^b))$     
instead of $-(\varrho_\alpha^b-\tr(\stra^b))$, where $\beta^b(s)=\min\{s,b\}$. 
Therefore, the  extension of the convergence analysis in this paper to $d=3$ and with
a weaker restriction on $\Delta t$ will be a topic of further research. 
\end{remark}

\begin{lemma}
\label{dstabthmaLex}
Under all of the assumptions of Theorem \ref{dstabthmaL}, 
the solution 
$\{(\buadh^{n},\stradh^{n},$\linebreak$\varrhoadh^n)\}_{n=1}^{N_T}$ 
of {\bf (P$^{\dt}_{\alpha,\delta,h}$)}, (\ref{eq:PdLaha}--c), satisfies 
the following bounds:
\begin{subequations}
\begin{align}
&\sum_{n=1}^{N_T} \dt_n \left\|{\mathcal S}
\left( \frac{\buadh^{n}-\buadh^{n-1}}{\dt_n}\right)
\right\|^{\frac{4}{\vartheta}}_{H^1(\D)} \leq C,
\label{Fstab4aL} \\
&\sum_{n=1}^{N_T} \dt_n \left\|{\mathcal E}
\left( \frac{\stradh^{n}-\stradh^{n-1}}{\dt_n}\right)
\right\|^{2}_{H^1(\D)}
+ 
\sum_{n=1}^{N_T} \dt_n \left\|{\mathcal E}
\left( \frac{\varrhoadh^{n}-\varrhoadh^{n-1}}{\dt_n}\right)
\right\|^{2}_{H^1(\D)}
\leq C,
\label{Fstab4aLa}\\
&\sum_{n=1}^{N_T}\dt_n\left\|\pi_h\left[
\kd(\stradh^{n},\varrhoadh^n)\,\Ad(\stradh^{n},\varrhoadh^n)\,
\Bd(\stradh^{n})\right]\right\|^2_{L^2(\D)} \leq C,
\label{kdAdrhobd}\\
&\sum_{n=1}^{N_T}\dt_n\left\|\pi_h\left[
\Ad(\stradh^{n},\varrhoadh^n)\,
\Bd(\stradh^{n})\right]\right\|^{\frac{8}{5}}_{L^{\frac{8}{5}}(\D)}
\leq C,
\label{kdAdrhobda}
\end{align}
\end{subequations}
where
$\vartheta \in (2,4]$ and $C$ in (\ref{Fstab4aL},c) is independent of
$\alpha$, as well as $\delta$, $h$ and $\Delta t$.
\end{lemma}
\begin{proof}
On choosing $\bv = {\mathcal R}_h \left[ {\mathcal S}
\left(\frac{\buadh^{n}-\buadh^{n-1}}{\dt_n}\right)\right] \in \Vhone$ 
in (\ref{eq:PdLaha})
yields, on noting (\ref{Rh}), (\ref{Swnorm}), (\ref{Rhstab})
 and Sobolev embedding, that
\begin{align}
\label{equndtb}
&\Re \left\|{\mathcal S} \left(\frac{\buadh^{n}-\buadh^{n-1}}{\Delta t_n}\right)
\right\|^2_{H^1(\D)}
=\Re\int_{\D}
\frac{\buadh^{n}-\buadh^{n-1}}{\Delta t_n} \cdot
{\mathcal R}_h \left[ {\mathcal S}
\left(\frac{\buadh^{n}-\buadh^{n-1}}{\Delta t_n}\right)\right]\ddx
\\
&\; =
\frac{\e}{\Wi} \intd 
\pi_h\left[ 
\kd(\stradh^{n},\varrhoadh^n)\,\Ad(\stradh^{n},\varrhoadh^n)
\,\Bd(\stradh^{n})\right] 
: \grad \left[{\mathcal R}_h \left[{\mathcal S}
\left(\frac{\buadh^{n}-\buadh^{n-1}}{\Delta t_n}\right)\right] \right]\ddx
\nonumber\\
& \hspace{0.8in} 
-(1-\e) \intd  \grad
\buadh^{n} 
: \grad \left[{\mathcal R}_h \left[{\mathcal S}
\left(\frac{\buadh^{n}-\buadh^{n-1}}{\Delta t_n}\right)\right] \right]\ddx
\nonumber \\ 
& \hspace{0.8in} 
- \frac{\Re}{2} \intd 
\left((\buadh^{n-1} \cdot \grad) \buadh^{n}\right)
\cdot \displaystyle
{\mathcal R}_h \left[{\mathcal S}
\left(\frac{\buadh^{n}-\buadh^{n-1}}{\Delta t_n}\right)
\right]\ddx
\nonumber \\
& \hspace{0.8in} + \frac{\Re}{2}
\intd
\buadh^{n} \cdot
\left((\buadh^{n-1} \cdot \grad)
\left[ {\mathcal R}_h \left[{\mathcal S} 
\left(\frac{\buadh^{n}-\buadh^{n-1}}{\Delta t_n}\right)
\right]\right]\right)\ddx \nonumber \\
& \hspace{0.8in} + \left \langle \f^n, 
\displaystyle
{\mathcal R}_h \left[{\mathcal S}
\left(\frac{\buadh^{n}-\buadh^{n-1}}{\Delta t_n}\right)
\right] \right \rangle_{H^1_0(\D)}
\nonumber \\
& \; \leq C \biggl[\,\left\|\pi_h\left[
\kd(\stradh^{n},\varrhoadh^n)\,\Ad(\stradh^{n},\varrhoadh^n)\,
\Bd(\stradh^{n})\right]\right\|^2_{L^2(\D)}
+\|\grad \buadh^{n}\|^2_{L^2(\D)}
\nonumber\\
& \hspace{0.8in}
+\|\,\|\buadh^{n-1}\|\,\|\buadh^{n}\|\,\|^2_{L^2(\D)}
+ \|\,\|\buadh^{n-1}\| \,\|\grad \buadh^{n}\|\,\|_{L^{1+\theta}(\D)}^2
+ \|\f^n\|_{H^{-1}(\D)}^2 \biggr],
\nonumber 
\end{align}
where $\theta>0$ as $d=2$. 
It follows from (\ref{interpinf}) and (\ref{kddefbd}) that 
\begin{align}
\label{kdAdbd}
&\left\|\pi_h\left[
\kd(\stradh^{n},\varrhoadh^n)\,\Ad(\stradh^{n},\varrhoadh^n)\,
\Bd(\stradh^{n})\right]\right\|^2_{L^2(\D)}
\\
& \hspace{1in}\leq \intd \pi_h\left[\left\|\kd(\stradh^{n},\varrhoadh^n)\,\Ad(\stradh^{n},\varrhoadh^n)\,
\Bd(\stradh^{n})\right\|^2\right]\ddx
\nonumber
\\
&\hspace{1in} \leq b\,\intd \pi_h\left[\tr\left(\left(\Ad(\stradh^{n},\varrhoadh^n)\right)^2
\Bd(\stradh^{n})\right)\right]\ddx.
\nonumber
\end{align}
Applying the Cauchy--Schwarz and the algebraic-geometric mean
inequalities, in conjunction with 
(\ref{GN}), for $d=2$, and the Poincar\'{e} inequality
(\ref{eq:poincare})
yields that
\begin{align}
\label{L4sob}
\displaystyle
\|\,\|\buadh^{n-1}\| \,\|\buadh^{n}\|\,\|_{L^2(\D)}^2
&\leq
\|\buadh^{n-1}\|^2_{L^4(\D)}
\,\|\buadh^{n}\|^2_{L^4(\D)}
\leq
\textstyle \frac{1}{2}
\displaystyle\sum_{m=n-1}^n
\|\buadh^m\|^4_{L^4(\D)}
\\ &
\leq C \displaystyle\sum_{m=n-1}^n
\left[
\|\buadh^m\|^{2}_{L^2(\D)}\,
\|\grad \buadh^m\|^{2}_{L^2(\D)}\,
\right].
\nonumber
\end{align}
Similarly, we have for any $\theta \in (0,1)$, as $d=2$, but now using a Young's inequality
\begin{align}
\label{sob1}
\|\,\|\buadh^{n-1}\| \,\|\grad \buadh^{n}\|\,\|_{L^{1+\theta}(\D)}^2
& \leq
\|\buadh^{n-1}\|_{L^{\frac{2(1+\theta)}{1-\theta}}(\D)}^2
\,\|\grad \buadh^{n}\|_{L^{2}(\D)}^2
\\
& \leq C\,
\|\buadh^{n-1}\|_{L^2(\D)}^{\frac{2(1-\theta)}{1+\theta}}
\,\|\grad \buadh^{n-1}\|_{L^{2}(\D)}^{\frac{4\theta}{1+\theta}}
\,\|\grad \buadh^{n}\|_{L^{2}(\D)}^2
\nonumber \\
&
\leq C \,\|\buadh^{n-1}\|_{L^2(\D)}^{\frac{2(1-\theta)}{1+\theta}}
\displaystyle\sum_{m=n-1}^n
\|\grad \buadh^m\|^{\frac{2(1+3\theta)}{1+\theta}}_{L^2(\D)}.
\nonumber
\end{align}
On taking the $\frac{2}{\vartheta}$
power of both sides of
(\ref{equndtb}),
multiplying by $\Delta t_n$,
summing from $n=1,\ldots,N_T$
and noting (\ref{kdAdbd}), (\ref{L4sob}), (\ref{sob1})
with $\theta = \frac{\vartheta - 2}{6 - \vartheta}
\Leftrightarrow \vartheta = \frac{2(1+3\theta)}{(1+\theta)} \in (2,4)$,
(\ref{Deltatqu}), (\ref{fncont}), (\ref{Fstab1aL}),
(\ref{Fstab2aL}) and (\ref{idatah}) yields that
\begin{align}
\label{corotfor3dis}
&\sum_{n=1}^{N_T} \Delta t_n 
\left \|{\mathcal S} \left(\frac{\buadh^{n}-\buadh^{n-1}}{\Delta
t_n}\right)\right\|^{\frac{4}{\vartheta}}_{H^1(\D)}
\\
&\hspace{0.5in} \leq 
C \left[\,\sum_{n=1}^{N_T} \Delta t_n
\intd \pi_h\left[\tr\left(\left(\Ad(\stradh^{n},\varrhoadh^n)\right)^2
\Bd(\stradh^{n})\right)\right]\ddx\right]^{\frac{2}{\vartheta}}
\nonumber \\
& \hspace{1in}
+ C \left[\,\sum_{n=1}^{N_T} \Delta t_n
\,\left[\|\grad \buadh^{n}\|_{L^2(\D)}^2 +
\|\f^n\|^2_{H^{-1}(\D)}
\right]
\right]^{\frac{2}{\vartheta}}
\nonumber \\
&\hspace{1in} 
+C \left[1+ \max_{n=0,\ldots, N_T}
\left(\|\buadh^{n}\|^2_{L^2(\D)} 
\right) 
 \right]
\,\left[\sum_{n=0}^{N_T} \Delta t_n\, 
\|\grad \buadh^{n}\|^2_{L^2(\D)}
\right]\nonumber\\
& \hspace{0.5in} \leq C,
\nonumber
\end{align}
and hence the bound (\ref{Fstab4aL}) with $C$ independent of $\alpha$.

Choosing $\bphi
= {\mathcal P}_h\left[{\mathcal E} \left( \frac{\stradh^{n}-\stradh^{n-1}}
{\Delta t_n}
\right) \right] \in \Shone$ in (\ref{eq:PdLahb})
yields,
on noting (\ref{Ph}) and (\ref{Echinorm}),
that
\begin{align}
\label{psitG1}
&\left
\|{\mathcal E} \left( \frac{\stradh^{n}-\stradh^{n-1}}{\Delta t_n}
\right)
\right\|^2_{H^1(\D)}
 \\
&\qquad =\intd \pi_h\left[ \left(\frac{\stradh^{n}-\stradh^{n-1}}{\Delta t_n}
\right) :
 {\mathcal P}_h \left[{\mathcal E} \left( \frac{\stradh^{n}
 -\stradh^{n-1}}{\Delta t_n}
\right) \right] \right]\ddx
\nonumber
\\
&\qquad =- \frac{1}{\Wi}\intd \pi_h \left[
\Ad(\stradh^{n},\varrhoadh^n)\,
\Bd(\stradh^{n}) : {\mathcal P}_h 
\left[{\mathcal E} \left( \frac{\stradh^{n}-\stradh^{n-1}}{\Delta t_n}
\right) \right] \right]\ddx
\nonumber \\
& \qquad \qquad -\alpha \intd
\grad  \stradh^{n} ::
\grad \left[{\mathcal P}_h \left[{\mathcal E} 
\left( \frac{\stradh^{n}-\stradh^{n-1}}{\Delta t_n}
\right) \right] \right]\ddx 
\nonumber \\
& \qquad \qquad + 2 \intd \grad \buadh^{n} :
\pi_{h}
\left[
\kd(\stradh^{n},\varrhoadh^n)\,
{\mathcal P}_h \left[{\mathcal E} 
\left( \frac{\stradh^{n}-\stradh^{n-1}}{\Delta t_n}
\right) \right] \Bd(\stradh^{n})\right]\ddx
\nonumber \\
& \qquad \qquad +
\intd \sum_{m=1}^d \sum_{p=1}^d 
[\buadh^{n-1}]_m \,
\Lambda_{\delta,m,p}(\stradh^{n}) : \frac{\partial}{\partial \xx_p}
\left[{\mathcal P}_h 
\left[{\mathcal E} \left( \frac{\stradh^{n}-\stradh^{n-1}}{\Delta t_n}
\right) \right] \right]\ddx.
\nonumber 
\end{align}
It follows from (\ref{eq:symmetric-tr},b), 
(\ref{ipmat}), (\ref{normphi})
and (\ref{modphisq})
that for any $\zeta \in \mathbb{R}_{>0}$
\begin{align}
&\left|\intd \pi_h \left[
\Ad(\stradh^{n},\varrhoadh^n)\,
\Bd(\stradh^{n}) : {\mathcal P}_h 
\left[{\mathcal E} \left( \frac{\stradh^{n}-\stradh^{n-1}}{\Delta t_n}
\right) \right] \right]\ddx \right|
\label{psitG1a} \\
& \quad \leq
\intd \pi_h
\left[
\|\Ad(\stradh^{n},\varrhoadh^n)\,[\Bd(\stradh^n)]^{\frac{1}{2}}\|
\left\|{\mathcal P}_h 
\left[{\mathcal E} \left( \frac{\stradh^{n}-\stradh^{n-1}}{\Delta t_n}
\right) \right]\right\| 
\,\|[\Bd(\stradh^n)]^{\frac{1}{2}}\|
\right]\ddx
\nonumber
\\
& \quad \leq
\zeta^{-1}
\intd \pi_h
\left[
\tr \left((\Ad(\stradh^{n},\varrhoadh^n))^2\,\Bd(\stradh^n)\right)\right]
\,\ddx
\nonumber \\
& \hspace{1in} +
\zeta\,d^{\frac{1}{2}}\intd
\pi_h
\left[\,
\left\|{\mathcal P}_h 
\left[{\mathcal E} \left( \frac{\stradh^{n}-\stradh^{n-1}}{\Delta t_n}
\right) \right]\right\|^2 
\|\Bd(\stradh^n)\|
\right]\ddx.
\nonumber
\end{align}
Similarly to (\ref{deltnorm}),  
it follows from
(\ref{interpinf}), (\ref{normprod}), (\ref{kddef}), (\ref{BdLdef})
and (\ref{modphisq})  
that  
\begin{align}
&\left\|\pi_{h}
\left[
\kd(\stradh^{n},\varrhoadh^n)\,
{\mathcal P}_h \left[{\mathcal E} 
\left( \frac{\stradh^{n}-\stradh^{n-1}}{\Delta t_n}
\right) \right] \Bd(\stradh^{n})\right]\right\|_{L^2(\D)}^2
\label{psitG1b}
\\
& \qquad \leq C 
\int_{\D} \pi_h \left[
\left\|{\mathcal P}_h \left[{\mathcal E} 
\left( \frac{\stradh^{n}-\stradh^{n-1}}{\Delta t_n}
\right)\right]\right\|^2\|\Bd(\stradh^{n})\| \right]\ddx.
\nonumber 
\end{align} 
In addition,  
(\ref{eq:Bd}), (\ref{inverse}),  
(\ref{eqnorm}) and (\ref{Fstab3aL}) imply that 
for all $\bphi \in \Shone$
\begin{align}
\label{psitG1c}
\int_{\D} \pi_h \left[
\left\|
\bphi
\right\|^2\|\Bd(\stradh^n)\| \right]\ddx
& \leq 
\sum_{k=1}^{N_K} \left[\|\stradh^{n}\|_{L^\infty(K_k)}+\delta\right]
\int_{K_k} 
\left\|\bphi \right\|^2 \,\ddx
\\
& \leq C\,\left[\|\stradh^{n}\|_{L^2(\D)} +\delta \right]
\left\|\bphi \right\|^2_{L^4(\D)} 
\leq C\,
\left\|\bphi \right\|^2_{H^1(\D)}. 
\nonumber
\end{align}
Combining (\ref{psitG1})--(\ref{psitG1c}), 
yields, on noting (\ref{Phstab}), that
\begin{align}
\label{psitG2}
&\left
\|{\mathcal E} \left( \frac{\stradh^{n}-\stradh^{n-1}}{\Delta t_n}
\right)
\right\|^2_{H^1(\D)}\\
& \hspace{0.5in} \leq
C  
\biggl[ \intd \pi_h
\left[
\tr \left((\Ad(\stradh^{n},\varrhoadh^n))^2\,\Bd(\stradh^n)\right)\right]
\,\ddx
+\alpha \|\grad \stradh^{n}\|_{L^2(\D)}^2
\nonumber 
\\
& \hspace{1in}
+ \|\grad \buadh^{n}\|_{L^2(\D)}^2
+ \|\buadh^{n-1}\|_{L^4(\D)}^2 
\intd \sum_{m=1}^2 \sum_{p=1}^2 
\left\|\Lambda_{\delta,m,p}(\stradh^{n})\right\|_{L^4(\D)}^2
\,\ddx
\biggr].
\nonumber 
\end{align}
Similarly to (\ref{psitG1})--(\ref{psitG2}),
choosing $\eta
= {\mathcal P}_h\left[{\mathcal E} \left( \frac{\varrhoadh^{n}-\varrhoadh^{n-1}}
{\Delta t_n}
\right) \right] \in \Qhone$ in (\ref{eq:PdLahc})
yields 
that 
\begin{align}
\label{psitG1rho}
&\left
\|{\mathcal E} \left( \frac{\varrhoadh^{n}-\varrhoadh^{n-1}}{\Delta t_n}
\right)
\right\|^2_{H^1(\D)}
\\
& \hspace{0.3in} \leq
C  
\biggl[ \intd \pi_h
\left[
\tr \left((\Ad(\stradh^{n},\varrhoadh^n))^2\,\Bd(\stradh^n)\right)\right]
\,\ddx
+\alpha \|\grad \varrhoadh^{n}\|_{L^2(\D)}^2
\nonumber 
\\
& \hspace{0.6in}
+ \|\grad \buadh^{n}\|_{L^2(\D)}^2
+ \|\buadh^{n-1}\|_{L^4(\D)}^2 
\intd \sum_{m=1}^2 \sum_{p=1}^2 
\left\|\Lambda_{\delta,m,p}\left(1-\frac{\varrhoadh^{n}}{b}\right)\right\|_{L^4(\D)}^2
\ddx
\biggr].
\nonumber
\end{align}
Multiplying (\ref{psitG2}) and (\ref{psitG1rho}) by $\Delta t_n$, summing from $n=1,...,N_T$
and noting (\ref{Fstab1aL}), (\ref{Fstab2aL}), (\ref{L4sob}), (\ref{Deltatqu}),
(\ref{idatah}) and (\ref{Fstab3aL}--c)
yields the bounds (\ref{Fstab4aLa}).

Multiplying (\ref{kdAdbd}) by $\Delta t_n$, summing from $n=1,...,N_T$
and noting (\ref{Fstab1aL}) yields the result (\ref{kdAdrhobd}) with $C$
independent of $\alpha$. 
Finally, it follows from (\ref{interprod}) that
\begin{align}
&
\left\|\pi_h\left[
\Ad(\stradh^{n},\varrhoadh^n)\,
\Bd(\stradh^{n})\right]\right\|^{\frac{8}{5}}_{L^{\frac{8}{5}}(\D)}
\label{kdAdrhobd1} \\
& \quad \leq
\int_{\D} 
\left(\pi_h\left[\left\|\Ad(\stradh^{n},\varrhoadh^n)\,
[\Bd(\stradh^{n})]^{\frac{1}{2}}\right\|^2\right]\right)^{\frac{4}{5}}
\left(\pi_h\left[\left\|
[\Bd(\stradh^{n})]^{\frac{1}{2}}\right\|^2\right]\right)^{\frac{4}{5}}
\ddx
\nonumber  \\
& \quad \leq
 \left\|\pi_h \left[\left\|\Ad(\stradh^{n},\varrhoadh^n)\,
[\Bd(\stradh^{n})]^{\frac{1}{2}}\right\|^2\right]
\right\|_{L^1(\D)}^{\frac{4}{5}}
\left\|\pi_h\left[\left\|
[\Bd(\stradh^{n})]^{\frac{1}{2}}\right\|^2\right]\right\|_{L^4(\D)}^{\frac{4}{5}}.
\nonumber  
\end{align}
Multiplying (\ref{kdAdrhobd1}) by $\Delta t_n$, summing from $n=1,...,N_T$
and noting
(\ref{ipmat}), (\ref{normphi}),  
(\ref{Fstab1aL}), (\ref{modphisq}), (\ref{eq:Bd}), (\ref{use1}), (\ref{GN}) with $d=2$,
(\ref{eqnorm}) and (\ref{Fstab3aL})
yields that 
\begin{align}
&\sum_{n=1}^{N_T}\dt_n
\left\|\pi_h\left[
\Ad(\stradh^{n},\varrhoadh^n)\,
\Bd(\stradh^{n})\right]\right\|^{\frac{8}{5}}_{L^{\frac{8}{5}}(\D)}
\label{kdAdrhobd2} \\
& \quad \leq
\left(\sum_{n=1}^{N_T}\dt_n\left\| \pi_h \left[\left\|\Ad(\stradh^{n},\varrhoadh^n)\,
[\Bd(\stradh^{n})]^{\frac{1}{2}}\right\|^2\right]
\right\|_{L^1(\D)}\right)^{\frac{4}{5}}\nonumber \\
& \hspace{1in} \times 
\left(\sum_{n=1}^{N_T}\dt_n
\left\|\pi_h\left[\|
\Bd(\stradh^{n})\|\right]\right\|_{L^4(\D)}^4
\right)^{\frac{1}{5}}
\nonumber \\
& \quad \leq C 
\left(\sum_{n=1}^{N_T}\dt_n
\left\|\pi_h\left[
\Bd(\stradh^{n})\right]\right\|_{L^4(\D)}^4
\right)^{\frac{1}{5}}
\leq C 
\left(\delta^4 + \sum_{n=1}^{N_T}\dt_n
\left\|\stradh^{n}\right\|_{L^4(\D)}^4
\right)^{\frac{1}{5}}
\nonumber \\
& \quad \leq C
\left(\delta^4 + \sum_{n=1}^{N_T}\dt_n \left\|\stradh^{n}\right\|_{L^2(\D)}^2\,
\left\|\stradh^{n}\right\|_{H^1(\D)}^2
\right)^{\frac{1}{5}} \leq C.
\nonumber
\end{align}
Hence, we have the desired result (\ref{kdAdrhobda}). 
\end{proof}

Unfortunately, the bound (\ref{Fstab4aL}) is not useful 
for obtaining compactness via (\ref{compact1}), see the discussion in the proof of 
Theorem \ref{convaL} below. Instead one has to exploit the compactness result (\ref{compact2}).
This we now do, by following the proof of Lemma 5.6 on p237 in Temam \cite{Temam}. 
Here the introduction of $\kd(\stradh^{n},\varrhoadh^n)$ in the extra stress term,
as discussed above (\ref{kddef}), 
is crucial, as this yields an $L^2(\D_T)$ bound on this stress term  via the bound (\ref{kddefbd}).
For this purpose, we introduce the following notation 
in line with (\ref{eq:constant-source}). 
Let $\buadh^{\dt} \in C([0,T];\Vhone)$ and $\buadh^{\dt,\pm}  
\in L^\infty(0,T;\Vhone)$
be such that
for $n=1,\ldots,N_T$
\begin{subequations}
\begin{alignat}{2}
\buadh^{\dt}(t,\cdot) &:= \frac{t-t^{n-1}}{\dt_n} \buadh^{n}(\cdot) 
+ \frac{t^n-t}{\dt_n} \buadh^{n-1}(\cdot)
&&\quad t \in [t^{n-1},t^n], 
\label{timeconaL}\\
\buadh^{\dt,+}(t,\cdot)&:= \buadh^{n}(\cdot),
\quad \buadh^{\dt,-}(t,\cdot):= \buadh^{n-1}(\cdot)
&&\quad t \in [t^{n-1},t^n), 
\label{time+-aL}\\
\mbox{and} \qquad \qquad \Delta(t)&:= \dt_n 
&&\quad t \in [t^{n-1},t^n). 
\label{deltat}
\end{alignat}
\end{subequations}
We note that 
\begin{align}
\buadh^{\dt}-\buadh^{\dt,\pm}&=(t-t^n_{\pm})\frac{\partial \buadh^{\dt}}
{\partial t}
\quad t \in (t^{n-1},t^n), \quad n=1,\ldots,N_T\,, 
\label{eqtime+}
\end{align}
where $t^n_{+}:=t^n$ and $t^n_{-}:=t^{n-1}$.
We shall adopt $\buadh^{\dt(,\pm)}$ as a collective symbol for
$\buadh^{\dt}$,  
$\buadh^{\dt,\pm}$.
We also define 
$\stradh^{\dt(,\pm)}$ 
and $\varrhoadh^{\dt(,\pm)}$ 
similarly to (\ref{timeconaL},b).

Using the notation (\ref{timeconaL},b), {\bf (P$_{\alpha,\delta,h}^{\Delta t}$)},
i.e.\
(\ref{eq:PdLaha}--c)
multiplied by $\Delta t_n$ and summed for $n=1,\ldots, N_T$,
can be restated as:
\begin{subequations}
\begin{alignat}{2}
\label{equncon}
 & \displaystyle \int_{\D_T} 
\left[ \Re \frac{\partial \buadh^{\dt}}{\partial t}
 \cdot \bv 
+ (1-\e) 
\grad \buadh^{\dt,+} :
\grad \bv \right] \ddx\,\ddt
\\
&\hspace{0.05in}+ \frac{\Re}{2} \int_{\D_T}
\left[ \left[ (\buadh^{\dt,-} \cdot \grad) \buadh^{\dt,+} \right] \cdot \bv
- \left[(\buadh^{\dt,-} \cdot \grad) \bv \right] \cdot \buadh^{\dt,+}
\right] \ddx\,\ddt
\nonumber \\
&\hspace{0.05in}+ \frac{\e}{\Wi} \int_{\D_T}
\pi_h\left[
\kd(\stradh^{\dt,+},\varrhoadh^{\dt,+})\,
\Ad(\stradh^{\dt,+},\varrhoadh^{\dt,+})
\,\Bd(\stradh^{\dt,+})\right]: \grad
\bv \,\ddx\,\ddt
\nonumber \\
& \hspace{1in}
=
\int_0^T \langle \f^+,\bv \rangle_{H^1_0(\D)} \, \ddt
&& 
\hspace{-1.2in}
\forall \bv \in L^2(0,T;\Vhone),
\nonumber
\\
\label{eqpsincon}
&\int_{\D_T} \pi_h\left[\frac{\partial \stradh^{\dt}}{\partial t}
: \bphi + \frac{\Ad(\stradh^{\dt,+},\varrhoadh^{\dt,+})\,
\Bd(\stradh^{\dt,+})
}{\Wi} 
: \bphi
\right] \ddx\,\ddt
\\
& \hspace{0.05in} 
+ \alpha \int_{\D_T} \grad \stradh^{\dt,+} :: \grad \bphi \,\ddx\,\ddt
-2 \int_{\D_T} \grad \buadh^{\dt,+}
: \pi_h[\kd(\stradh^{\dt,+},\varrhoadh^{\dt,+})\,\bphi\,\Bd(\stradh^{\dt,+})] \,\ddx\,\ddt
\nonumber \\
& \hspace{0.05in} -\int_{\D_T}
\sum_{m=1}^d \sum_{p=1}^d [\buadh^{\dt,-}]_m \,\Lambda_{\delta,m,p}
(\stradh^{\dt,+}) : \frac{\partial \bphi}{\partial \xx_p}\,
\ddx\,\ddt
=0 
&&\hspace{-1.2in}\forall \bphi \in L^2(0,T;\Shone),
\nonumber
\\
\label{eqvarrhocon}
&\int_{\D_T} \pi_h\left[\frac{\partial \varrhoadh^{\dt}}{\partial t}
\,\eta + 
\frac{\tr\brk{\Ad(\stradh^{\dt,+},\varrhoadh^{\dt,+})\,
\Bd(\stradh^{\dt,+})}
}{\Wi}
\,\eta
\right] \ddx\,\ddt
\\
& \hspace{0.05in} 
+ \alpha \int_{\D_T} \grad \varrhoadh^{\dt,+} \cdot \grad \eta \,\ddx\,\ddt
-2 \int_{\D_T} \grad \buadh^{\dt,+}
: \pi_h[\kd(\stradh^{\dt,+},\varrhoadh^{\dt,+})\,\eta\,\Bd(\stradh^{\dt,+})] \,\ddx\,\ddt
\nonumber \\
& \hspace{0.05in} +b \int_{\D_T}
\sum_{m=1}^d \sum_{p=1}^d [\buadh^{\dt,-}]_m \,\Lambda_{\delta,m,p}
\left(1-\frac{\varrhoadh^{\dt,+}}{b}\right) \frac{\partial \eta}{\partial \xx_p}\,
\ddx\,\ddt
=0 
&&\hspace{-1.2in}\forall \eta \in L^2(0,T;\Qhone)
\nonumber
\end{alignat}
\end{subequations}
subject to the initial conditions 
$\buadh^{\dt}(0)= \buh^0$, $\stradh^{\dt}(0)=\strh^0$ 
and $\varrhoadh^{\dt}(0)=\tr(\strh^0)$. 

\begin{lemma}
\label{dstabthmaLexF}
Under all of the assumptions of Theorem \ref{dstabthmaL},
the solution 
$(\buadh^{\Delta t},\stradh^{\Delta t},\varrhoadh^{\Delta t})$ 
of \linebreak {\bf (P$^{\dt}_{\alpha,\delta,h}$)}, (\ref{equncon}--c), satisfies 
the following bound:
\begin{align}
\int_0^T \|D^\gamma_t \buadh^{\Delta t}\|^2_{L^2(\D)} \,\ddt
\leq C,
\label{Fstab4aLF}
\end{align}
where
$\gamma \in (0, \textstyle \frac{1}{4})$
and $C$ is independent of $\alpha$, as well as $\delta$, $h$ and $\Delta t$.
\end{lemma}
\begin{proof}
Equation (\ref{equncon}) can be reinterpreted as
\begin{align}
\Re \, \frac{\rm d}{\ddt} \int_\D \buadh^{\Delta t}(t)
 \cdot \bv \,\ddx
= \int_\D \grad \bg^{\Delta t,+}(t) : \grad \bv \,\ddx
\qquad \forall \bv \in \Vhone, 
\qquad t \in (0,T), 
\label{reinterp}
\end{align} 
where $\bg^{\Delta t,+}(t) \in \Vhone$ is defined by  
\begin{align}
\label{bg}
&\int_\D \grad \bg^{\Delta t,+}(t) : \grad \bv \,\ddx=
\langle \f^+(t),\bv \rangle_{H^1_0(\D)}
- (1-\e)\, \int_\D 
\grad \buadh^{\dt,+}(t) :
\grad \bv \,\ddx\\
& \quad - \frac{\Re}{2} \int_{\D}
\left[ \left[ (\buadh^{\dt,-}(t) \cdot \grad) \buadh^{\dt,+}(t) \right] \cdot \bv
- \left[(\buadh^{\dt,-}(t) \cdot \grad) \bv \right] \cdot \buadh^{\dt,+}(t)
\right]\ddx
\nonumber \\ 
& \quad -
\frac{\e}{\Wi} \int_{\D}
\pi_h\left[
\kd(\stradh^{\dt,+}(t),\varrhoadh^{\dt,+}(t))\,
\Ad(\stradh^{\dt,+}(t),\varrhoadh^{\dt,+}(t))
\,\Bd(\stradh^{\dt,+}(t))\right]: \grad
\bv\,\ddx.
\nonumber 
\end{align} 
Similarly to (\ref{equndtb}) and (\ref{kdAdbd}) with $\theta \in (0,1)$,
it follows from (\ref{bg}) that
\begin{align}
\label{bgbnd}
\|\grad \bg^{\Delta t,+}(t)\|_{L^2(\D)}   
& \leq C\,\biggl[ \|\f^+(t)\|_{H^{-1}(\D)} +
\|\grad \buadh^{\Delta t,+}(t)\|_{L^2(\D)}
\\
& \qquad
+ \| \,\|\buadh^{\Delta t,-}(t)\|   
\, \|\buadh^{\Delta t,+}(t)\| \,\|_{L^2(\D)}  
+ \| \,\|\buadh^{\Delta t,-}(t)\|   
\, \|\grad \buadh^{\Delta t,+}(t)\| \,\|_{L^{1+\theta}(\D)}
\nonumber \\
& \qquad
+ \left(\intd \tr \left(\left(\Ad(\stradh^{\dt,+}(t),\varrhoadh^{\dt,+}(t))\right)^2
\Bd(\stradh^{\dt,+}(t)) \right)\ddx\right)^{\frac{1}{2}}
\biggr].
\nonumber
\end{align}
On noting (\ref{L4sob}), (\ref{sob1}) and the bound on 
$\buadh^{n}$ in (\ref{Fstab2aL}), we deduce from (\ref{bgbnd}) that 
\begin{align}
\label{bgbnd1}
\|\grad \bg^{\Delta t,+}(t)\|_{L^2(\D)} 
&\leq C\,\biggl[1+ \|\f^+(t)\|_{H^{-1}(\D)} 
+ \|\grad \buadh^{\Delta t,-}(t)\|_{L^2(\D)}^{\frac{1+3\theta}{1+\theta}}    
+ \|\grad \buadh^{\Delta t,+}(t)\|_{L^2(\D)}^{\frac{1+3\theta}{1+\theta}} 
\\
& \hspace{0.5in} +
\left(\intd \tr\left( \left(\Ad(\stradh^{\dt,+}(t),\varrhoadh^{\dt,+}(t))\right)^2
\Bd(\stradh^{\dt,+}(t))\right) \ddx\right)^{\frac{1}{2}}
\biggr].
\nonumber
\end{align}
Similarly to (\ref{corotfor3dis}), on recalling (\ref{idatah}), (\ref{fncont}) and  
(\ref{Fstab2aL}), we deduce from (\ref{bgbnd1}) that  
\begin{align}
\int_0^T 
\|\grad \bg^{\Delta t,+}(t)\|_{L^2(\D)}^{\frac{4}{\vartheta}} \,\ddt \leq C,
\label{bgbnd2}
\end{align}
where $\vartheta \in (2,4]$ and $C$ is independent of $\alpha$, 
as well as $\delta$, $h$ and $\Delta t$.
The rest of the proof follows as on p1825--6 in 
\cite{barrett-boyaval-09}, which is based on the proof of Lemma 5.6 on p237 in \cite{Temam}.
\end{proof}

\subsection{Convergence}

It follows from (\ref{Fstab1aL}), (\ref{Fstab2aL}), (\ref{Fstab3aL}--c), (\ref{idatah}), 
(\ref{eqnorm}), 
(\ref{Fstab4aL}--d), (\ref{Fstab4aLF}) and (\ref{timeconaL}--c)
that
\begin{subequations}
\begin{align}
&\sup_{t \in (0,T)} \|\buadh^{\dt(,\pm)}\|^2_{L^2(\D)}
+ 
\int_{0}^T \left[
\|\grad \buadh^{\dt (,\pm)}\|^2_{L^2(\D)}
+
\frac{\|\buadh^{\dt,+}
-\buadh^{\dt,-}\|^2_{L^2(\D)}}{\Delta(t)}
\right]
 \ddt \leq C,
\label{stab1c}
\\
&\sup_{t \in (0,T)}  
\|\stradh^{\dt(,\pm)}\|^2_{L^2(\D)}
\label{stab2c}
+  \int_{0}^T \left[\alpha \|\grad \stradh^{\dt(, \pm)}\|_{L^2(\D)}^2
+ \frac{\|\stradh^{\dt,+}-\stradh^{\dt,-}\|^2_{L^2(\D)}}{\Delta(t)}
\right]
\ddt \leq C,\\
&\sup_{t \in (0,T)} 
\|\varrhoadh^{\dt(,\pm)}\|^2_{L^2(\D)}
\label{stab2crho}
+  \int_{0}^T \left[\alpha \|\grad \varrhoadh^{\dt(, \pm)}\|_{L^2(\D)}^2
+ \frac{\|\varrhoadh^{\dt,+}-\varrhoadh^{\dt,-}\|^2_{L^2(\D)}}{\Delta(t)}
\right]
\ddt\\ 
& \hspace{2in}
+ \delta^2 \int_0^T \left\|\grad \pi_h\left[\Gd'\left(1-\frac{\varrhoadh^{\dt,+}}{b}\right)
\right] \right\|^2_{L^2(\D)}\ddt
\leq C,\nonumber \\
\label{stab2ctr}
&\int_{\D_T} \pi_h\left[\tr\left(\Ad(\stradh^{\dt,+},\varrhoadh^{\dt,+})^2
\Bd(\stradh^{\dt,+})\right)\right]\ddx \,\ddt \leq C,\\
&\sup_{t \in (0,T)} \intd 
\pi_h\left[ \|[\stradh^{\dt(,\pm)}]_{-}\| + |[b-\varrhoadh^{\dt(,\pm)}]_{-}|
\right] \ddx  \leq C\,\delta,
\label{stab3delta}
\\
&\int_0^T \left[ \left\|{\mathcal S}\,\frac{\partial \buadh^{\dt}}{\partial t}
\right\|_{H^1(\D)}^{\frac{4}{\vartheta}}
+ \|D^{\gamma}_t \buadh^{\dt}\|_{L^2(\D)}^2
\right] \ddt \leq C, 
\label{stab3c}\\
&\int_0^T \left[
\left\|{\mathcal E}\,\frac{\partial 
\stradh^{\dt}}{\partial t} \right\|_{H^1(\D)}^{2}
+ \left\|{\mathcal E}\,\frac{\partial 
\varrhoadh^{\dt}}{\partial t} \right\|_{H^1(\D)}^{2}
\right] \ddt
\leq C,
\label{stab3crho}\\
&\left\|\pi_h\left[\kd(\stradh^{\dt,+},\varrhoadh^{\dt,+})\,
\Ad(\stradh^{\dt,+},\varrhoadh^{\dt,+})
\,\Bd(\stradh^{\dt,+})\right]\right\|_{L^{2}(\D_T)} 
\leq C,
\label{stabkdAdBd} \\
&\left\|\pi_h\left[
\Ad(\stradh^{\dt,+},\varrhoadh^{\dt,+})
\,\Bd(\stradh^{\dt,+})\right]\right\|_{L^{\frac{8}{5}}(\D_T)}
\leq C,
\label{stabkdAdBda}
\end{align}
\end{subequations}
where $\vartheta \in (2,4]$, $\gamma\in (0,\frac{1}{4})$ and
$C$ in (\ref{stab1c},d,e,f,h) is independent of $\alpha$, as well as $\delta$, $h$ and $\Delta t$.

We are now in a position to prove the following convergence result 
concerning {\bf (P$^{\Delta t}_{\alpha,\delta,h}$)}.
\begin{theorem}\label{convaL}
Under all of the assumptions of Theorem \ref{dstabthmaL},
there exists a subsequence of \linebreak 
$\{(\buadh^{\dt},\stradh^{\dt},\varrhoadh^{\dt})\}_{
\delta>0,h>0,\Delta t>0}$, 
and functions 
\begin{subequations}
\begin{align}
\buaL &\in
L^{\infty}(0,T;
{\rm H})\cap L^{2}(0,T;\Uz) \cap
W^{1,\frac{4}{\vartheta}}(0,T;\Uz')
\mbox{ with } \buaL(0)=\bu^0,
\label{ualpreg}\\ 
\straL &\in L^{\infty}(0,T;[L^{2}(\D)]^{2 \times 2}_{\rm S})
\cap L^{2}(0,T;[H^{1}(\D)]^{2 \times 2}_{\rm S})
\cap H^1(0,T;([H^{1}(\D)]^{2 \times 2}_{\rm S})')
\label{stralpreg}\\
& \hspace{1in}\mbox{ 
with $\straL$ non-negative definite a.e.\ in } \D_T
\mbox{ and } \straL(0)=\strs^0,
\nonumber \\
\varrho_\alpha &\in L^{\infty}(0,T;
L^{2}(\D))
\cap L^{2}(0,T;H^{1}(\D))
\cap H^1(0,T;(H^{1}(\D))')
\label{rhoalpreg}\\
& \hspace{1in}
\mbox{ with } \varrho_\alpha \leq b \mbox{ a.e.\ in } \D_T
\mbox{ and } \varrho_\alpha(0)=\tr(\strs^0),
\nonumber
\end{align}
\end{subequations}
such that, as $\delta,\,h,\,\Delta t \rightarrow 0_+$,
\begin{subequations}
\begin{alignat}{2}
\buadh^{\Delta t (,\pm)} &\rightarrow \buaL \qquad &&\mbox{weak* in }
L^{\infty}(0,T;[L^2(\D)]^2), \label{uwconL2}\\
\buadh^{\Delta t (,\pm)} &\rightarrow \buaL \qquad &&\mbox{weakly in }
L^{2}(0,T;[H^1(\D)]^{2}), \label{uwconH1}\\
{\mathcal S} \frac{\partial \buadh^{\dt}}{\partial t} 
&\rightarrow {\mathcal S} \frac{\partial \buaL}{\partial t}
 \qquad &&\mbox{weakly in }
L^{\frac{4}{\vartheta}}(0,T;\Uz), \label{utwconL2}\\
\buadh^{\Delta t (,\pm)} &\rightarrow \buaL
\qquad &&\mbox{strongly in }
L^{2}(0,T;[L^{r}(\D)]^2), \label{usconL2}
\end{alignat}
\end{subequations}
\begin{subequations}
\begin{alignat}{2}
\stradh^{\Delta t (,\pm)} &\rightarrow
\straL
\quad &&\mbox{weak* in }
L^{\infty}(0,T;[L^2(\D)]^{2 \times 2}), \label{psiwconL2}\\
\stradh^{\Delta t(,\pm)}
&\rightarrow  \straL
\quad &&\mbox{weakly in }
L^{2}(0,T;[H^1(\D)]^{2 \times 2}), \label{psiwconH1x}\\
{\mathcal E} \frac{\partial \stradh^{\Delta t}}{\partial t}
&\rightarrow {\mathcal E} \frac{\partial \straL}{\partial t}
 \qquad &&\mbox{weakly in }
L^{2}(0,T;[H^1(\D)]^{2 \times 2}_{\rm S}), \label{psitwconL2}\\
\stradh^{\Delta t (,\pm)} &\rightarrow
\straL
\qquad &&\mbox{strongly in }
L^{2}(0,T;[L^{r}(\D)]^{2 \times 2}), \label{psisconL2}
\\
\pi_h[\Bd(
\stradh^{\Delta t (,\pm)})] &\rightarrow
\straL
\qquad &&\mbox{strongly in }
L^{2}(0,T;[L^{2}(\D)]^{2 \times 2}),
\label{piBdLconv}
\\
\Lambda_{\delta,m,p}(\stradh^{\Delta t (,\pm)}) &\rightarrow
\straL\,\delta_{mp}
\qquad &&\mbox{strongly in }
L^{2}(0,T;[L^{2}(\D)]^{2 \times 2}),
\quad m,\,p =1,\,2,
\label{XXxLinf}
\end{alignat}
\end{subequations}
and
\begin{subequations}
\begin{alignat}{2}
\varrhoadh^{\Delta t (,\pm)} &\rightarrow
\varrho_\alpha
\quad &&\mbox{weak* in }
L^{\infty}(0,T;L^2(\D)), \label{psiwconL2rho}\\
\varrhoadh^{\Delta t(,\pm)}
&\rightarrow  \varrho_\alpha
\quad &&\mbox{weakly in }
L^{2}(0,T;H^1(\D)), \label{psiwconH1xrho}\\
{\mathcal E} \frac{\partial \varrhoadh^{\Delta t}}{\partial t}
&\rightarrow {\mathcal E} \frac{\partial \varrho_\alpha}{\partial t}
 \qquad &&\mbox{weakly in }
L^{2}(0,T;H^1(\D)), \label{psitwconL2rho}\\
\varrhoadh^{\Delta t (,\pm)} &\rightarrow
\varrho_\alpha
\qquad &&\mbox{strongly in }
L^{2}(0,T;L^{r}(\D)), \label{psisconL2rho}
\\
\pi_h\left[\Bd\left(1-\frac{\varrhoadh^{\Delta t (,\pm)}}{b}\right)\right] &\rightarrow
\left(1-\frac{\varrho_\alpha}{b}\right)
\qquad &&\mbox{strongly in }
L^{2}(0,T;L^{2}(\D)),
\label{piBdLconvrho}\\
\Lambda_{\delta,m,p}\left(1-\frac{\varrhoadh^{\Delta t (,\pm)}}{b}\right) &\rightarrow
\left(1-\frac{\varrho_\alpha}{b}\right)\,\delta_{mp}
\qquad &&\mbox{strongly in }
L^{2}(0,T;L^{2}(\D)),\quad m,\,p =1,\,2,
\label{XXxLinfrho}
\end{alignat}
\end{subequations}
where $\vartheta \in (2,4]$ 
and $r \in [1,\infty)$. 
\end{theorem}
\begin{proof}
The results  (\ref{uwconL2}--c) follow immediately from the bounds
(\ref{stab1c},f) on noting the notation (\ref{timeconaL}--c).
The denseness of $\bigcup_{h>0}{\rm Q}_h^1$ in $L^2(\D)$
and (\ref{Vh1}) yield that $\buaL \in L^2(0,T;\Uz)$.
Hence, the result (\ref{ualpreg}) holds on noting (\ref{Swnorm}) and (\ref{bu0hconv}),
where $\bua :[0,T] \rightarrow {\rm H}$ is weakly continuous.

The strong convergence result
(\ref{usconL2}) for $\buadh^{\Delta t}$ and $r=2$
follows immediately from (\ref{stab1c}) and the second bound in (\ref{stab3c}) and  
(\ref{compact2}) with ${\mathcal X}_0 = [H^1(\D)]^2$ and ${\mathcal X}={\mathcal X}_1=[L^2(\D)]^2$.
Here we note that $H^1(\D)$ is compactly embedded in $L^2(\D)$.
We note here also that one cannot 
appeal to (\ref{compact1}) for this strong convergence result
with 
$\mu_0=2$, $\mu_1= 4/\vartheta$,
${\mathcal Y}_0 = [H^1({\mathcal D})]^2$, 
${\mathcal Y}_1 = {\mathrm V}'$ with norm $\|{\mathcal S} \cdot\|_{H^1(\D)}$ and
${\mathcal Y} = [L^2({\mathcal D})]^2$ 
for the stated values of $\vartheta$,
since $[L^2({\mathcal D})]^2$ is not continuously embedded in ${\mathrm V}'$
as ${\mathrm V}$ is not dense in $[L^2({\mathcal D})]^2$.

The result (\ref{usconL2}) for
$\buadh^{\Delta t,\pm}$ and $r=2$ follows immediately from 
this result for $\buadh^{\Delta t}$ and the 
the bound on the last
term on the left-hand side of (\ref{stab1c}), which yields
\begin{eqnarray}
\|\buadh^{\Delta t}-\buadh^{\Delta t,\pm}\|_{L^{2}(0,T;L^{2}(\D))}^2
\leq C\,\Delta t.
\label{upmr1}
\end{eqnarray}
Finally, we note from (\ref{GN}), for $d=2$, 
that, for all $\eta \in
L^2(0,T;H^1(\D))$,
\begin{align}
\|\eta\|_{L^2(0,T;L^{r}(\D))} 
&\leq C\,\|\eta\|_{L^2(0,T;L^{2}(\D))}^{1-\theta} \,\|\eta\|_{L^2(0,T;H^{1}(\D))}^{\theta}
\label{upmr2}
\end{align}
for all $r \in [2,\infty)$ 
with $\theta = 1-\frac{2}{r} \in (0,1].$ 
Hence, combining (\ref{upmr2}) and
(\ref{uwconH1},d) for $\buadh^{\Delta t(,\pm)}$ with $r=2$ yields (\ref{usconL2}) for
$\buadh^{\Delta t(,\pm)}$ for the stated values of $r$.

Similarly, the results (\ref{psiwconL2}--c)
follow immediately from (\ref{stab2c},g).
The strong convergence result (\ref{psisconL2}) for $\stradh^{\Delta t}$
follows immediately from (\ref{psiwconH1x},h), (\ref{Echinorm}) 
and (\ref{compact1})
with $\mu_0=\mu_1=2$,
${\mathcal Y}_0 = [H^1({\mathcal D})]^{d \times d}$, 
${\mathcal Y}_1 = ([H^1({\mathcal D})]^{d \times d})'$ and
${\mathcal Y} = [L^r({\mathcal D})]^{d\times d}$ 
for the stated values of $r$. 
Here we note that $H^1(\D)$ is compactly embedded in $L^r(\D)$ for the stated values of $r$,
and $L^r(\D)$ is continuously embedded in $(H^{1}(\D))'$.
Similarly to (\ref{upmr1}) and (\ref{upmr2}), the last bound in (\ref{stab2c})
then yields that (\ref{psisconL2}) holds for  $\stradh^{\Delta t(,\pm)}$.

The results (\ref{psiwconL2rho}--d) follow analogously from noting (\ref{stab2crho},g). 
Hence, on noting (\ref{Echinorm})
and (\ref{bu0hconv}), the results (\ref{stralpreg},c) hold,
where $\stra :[0,T] \rightarrow [L^2(\D)]^{d\times d}_{\rm S}$ and
$\varrho_\alpha :[0,T] \rightarrow L^2(\D)$
are weakly continuous,
apart from
the claims 
on the non-negative definiteness of $\stra$ and the bound on $\varrho_\alpha$. 
It remains to prove these, (\ref{piBdLconv},f) and (\ref{piBdLconvrho},f).
It follows from (\ref{Lip}), (\ref{MXittxbdq}) and (\ref{stab2c},e) 
that 
\begin{align}
&\left\|[\stra]_{-}\right\|_{L^2(0,T;L^1(\D))}
\label{nonnegcon}\\
& \qquad 
\leq \left\|[\stra]_{-}-[\stradh^{\dt(,\pm)}]_{-}\right\|_{L^2(0,T;L^1(\D))}
+ \left\|
[\stradh^{\dt(,\pm)}]_{-}
-\pi_h\left[[\stradh^{\dt(,\pm)}]_{-}\right]
\right\|_{L^2(0,T;L^1(\D))}
\nonumber \\
& \hspace{2in}
+ \left\|
\pi_h\left[[\stradh^{\dt(,\pm)}]_{-}\right]
\right\|_{L^2(0,T;L^1(\D))}
\nonumber
\\
& \qquad 
\leq \left\|\stra-\stradh^{\dt(,\pm)}\right\|_{L^2(0,T;L^1(\D))}
+ C\left[h+\delta\right].
\nonumber 
\end{align}
The desired non-negative definiteness result on $\stra$ 
in (\ref{stralpreg})
then follows from (\ref{psisconL2}). 
The desired bound on $\varrho_\alpha$ in (\ref{rhoalpreg}) follows similarly from 
(\ref{Lip}), (\ref{MXittxbdq}), (\ref{stab2crho},e) 
and (\ref{psisconL2rho}).
The results (\ref{piBdLconv},f) follow immediately from
(\ref{MXittxbd}), (\ref{stab2c}), (\ref{Lip}), (\ref{psisconL2}), (\ref{eq:Bd})
and the non-negative definiteness result on $\stra$ in (\ref{stralpreg}).  
The results (\ref{piBdLconvrho},f) follow similarly
from the scalar version of  (\ref{MXittxbd}), (\ref{stab2crho}), 
(\ref{Lip}), (\ref{psisconL2rho}), (\ref{eq:Bd})
and the bound on $\varrho_\alpha$ in (\ref{rhoalpreg}).  
\end{proof}

\begin{lemma}\label{rhoathm}
Under all of the assumptions of Theorem \ref{dstabthmaL},
the subsequence of  
$\{(\stradh^{\dt},$\linebreak $\varrhoadh^{\dt})\}_{\delta>0,h>0,\Delta t>0}$ 
of Theorem \ref{convaL}
and the limiting functions $\straL$ and $\varrho_\alpha$,
satisfying (\ref{stralpreg},c), are such that,
as $\delta,\,h,\,\Delta t \rightarrow 0_+$,  
\begin{align}
\pi_h\left[\Bd^b(\varrhoadh^{\dt(,\pm)})\right]
\rightarrow \varrho_\alpha = \tr(\straL)
\qquad \mbox{strongly in } L^2(0,T;L^2(\D)). 
\label{rhosigeq}
\end{align} 
In addition, we have with $h=o(\delta)$, as $\delta \rightarrow 0_+$, 
that 
\begin{align}
&\straL \mbox{ is positive definite and } 
\tr(\straL) < b
\qquad \mbox{a.e.\ in } \D_T.
\label{sigPD}
\end{align}
\end{lemma}
\begin{proof}
Choosing $\bphi=\eta\,\I$
in (\ref{eqpsincon}) and subtracting from (\ref{eqvarrhocon}) yields that
\begin{align}
\label{eqvarrhocondif}
&\int_{\D_T} \pi_h\left[\frac{\partial}{\partial t}
\left(\varrhoadh^{\dt}-\tr(\stradh^{\dt})\right)
\,\eta
\right] \ddx\,\ddt
+ \alpha \int_{\D_T} \grad \left(\varrhoadh^{\dt,+} 
-\tr(\stradh^{\dt,+})\right)
\cdot \grad \eta \,\ddx\,\ddt
\\
& \hspace{0.1in} +\int_{\D_T}
\sum_{m=1}^d \sum_{p=1}^d [\buadh^{\dt,-}]_m \left[b\,\Lambda_{\delta,m,p}
\left(1-\frac{\varrhoadh^{\dt,+}}{b}\right)
+ \tr\left( \Lambda_{\delta,m,p}(\stradh^{\dt,+})\right)
\right] \frac{\partial \eta}{\partial \xx_p}\,
\ddx\,\ddt
=0 
\nonumber \\
& \hspace{3.8in}
\qquad \forall \eta \in L^2(0,T;\Qhone).
\nonumber
\end{align}
It follows from (\ref{usconL2}), (\ref{psiwconH1x}--d,f), (\ref{psiwconH1xrho}--e), 
(\ref{Echi}), 
(\ref{ualpreg}--c) and 
(\ref{interp1},b)
that we may pass to the limit $\delta, \,h,\, \Delta t \rightarrow 0_+$
in 
(\ref{eqvarrhocondif}) with $\eta \ = \pi_h\,\chi$
to obtain 
\begin{align}
\label{eqdifcon}
&\int_{0}^T \langle\frac{\partial}{\partial t}
\left(\varrho_\alpha-\tr(\straL)\right),\chi\rangle_{H^1(\D)}
\,\ddt
+ \alpha \int_{\D_T} \grad \left(\varrho_\alpha 
-\tr(\straL)\right)
\cdot \grad \chi \,\ddx\,\ddt
\\
& \hspace{1in} - \int_{\D_T}
\left(\varrho_\alpha-\tr(\straL)\right)\,\bua\cdot \grad \chi
\,\ddx\,\ddt
=0 
\qquad \forall \chi \in C^\infty_0(0,T;C^\infty(\overline{\D})),
\nonumber
\end{align}
where $[\varrho_\alpha-\tr(\straL)](0)=0$.
For example, in order to pass to the limit on the first term in 
(\ref{eqvarrhocondif}),
we note that
\begin{align}
\label{intpartstr} 
&\int_{\D_T} 
\pi_h\left[\frac{\partial}{\partial t}
\left(\varrhoadh^{\dt}-\tr(\stradh^{\dt})\right)
\,\pi_h\,\chi
\right] \ddx\,\ddt 
\\
&\quad
=
\int_{\D_T} \left\{
\left(\frac{\partial}{\partial t}
\left(\varrhoadh^{\dt}-\tr(\stradh^{\dt})\right) \right)
\, \pi_h 
\,\chi + (I-\pi_h)\left[\left(\varrhoadh^{\dt}-\tr(\stradh^{\dt})\right) 
\, \pi_h \left[ \frac{\partial \chi}{\partial t} \right]\right] \right\} 
\ddx\,\ddt.
\nonumber
\end{align}
Hence the desired first term in (\ref{eqdifcon}) 
follows from noting (\ref{psiwconH1x},c), (\ref{psiwconH1xrho},c),
(\ref{Echi}) and 
(\ref{interp1},b).
As
$ C^\infty_0(0,T;C^\infty(\overline{\D}))$
is dense in
$L^2(0,T;H^1(\D))$, we have, on noting (\ref{ualpreg}--c), that
(\ref{eqdifcon}) holds for all $\chi \in L^2(0,T;H^1(\D))$.
It then follows from (\ref{ualpreg}--c) that we can choose 
$\chi=\varrho_\alpha-\tr(\straL)$ in (\ref{eqdifcon})
to yield that $\varrho_\alpha=\tr(\straL)$
as $[\varrho_\alpha-\tr(\straL)](0)=0$.  
Recalling (\ref{stralpreg},c), we have that $\varrho_\alpha \in [0,b]$ a.e.\ in 
$\D_T$. The desired convergence result (\ref{rhosigeq}) then follows
from noting this, (\ref{BdLdef}), (\ref{eq:Bd}), (\ref{MXittxbdq}), (\ref{stab2crho})
and (\ref{psisconL2rho}).

We now improve on (\ref{rhoalpreg}) by establishing 
that $\varrho_\alpha = \tr(\straL) <b$ a.e.\ in $\D_T$.
Assuming that $\tr(\straL)=b$ a.e.\ in $\D_T^b \subset \D_T$, 
we have
\begin{align}
b\,|\D^b_T| &= \int_{\D_T^b} \tr(\straL)\,\ddx\,\ddt 
\label{weakrhob} \\
&= \int_{\D_T^b} \pi_h\left[\tr(\Bd(\stradh^{\dt,+}))\right] \,\ddx\,\ddt
+ \int_{\D^b_T} \tr\left(\straL-\pi_h\left[\Bd(\stradh^{\dt,+})\right]\right) \,\ddx\,\ddt  
=: T_1 + T_2.
\nonumber 
\end{align}
We deduce from (\ref{Addef}) and (\ref{eq:BGd}) that 
\begin{align}
T_1 &=
\int_{\D_T^b} \pi_h\left[\left(\tr\left(\Ad(\stradh^{\dt,+},\varrhoadh^{\dt,+})\,
\Bd(\stradh^{\dt,+})\right)+ 2\right)\,
\Bd\left(1-\frac{\varrhoadh^{\dt,+}}{b}\right)\right] \ddx\,\ddt.
\label{weakrhob1}
\end{align}
It follows from (\ref{ipmat}) and (\ref{normphi}) that 
\begin{align}
\tr\left(\Ad(\stradh^{\dt,+},\varrhoadh^{\dt,+})\,
\Bd(\stradh^{\dt,+})\right)
&\leq 
\left[\tr\left((\Ad(\stradh^{\dt,+},\varrhoadh^{\dt,+}))^2\,
\Bd(\stradh^{\dt,+})\right)
\tr\left(\Bd(\stradh^{\dt,+})\right)\right]^{\frac{1}{2}}.
\label{TrAB}
\end{align}
Combining (\ref{weakrhob1}) and (\ref{TrAB}) yields, on noting
a scalar version of (\ref{interprod}) over $D_T^b$, (\ref{stab2ctr}), (\ref{modphisq})
and (\ref{eq:Bd}) that 
\begin{align}
T_1 & \leq  
C\left( \int_{\D_T^b} \pi_h\left[
\left[\tr\left(\Bd(\stradh^{\dt,+})\right)+1\right]
\left[\Bd\left(1-\frac{\varrhoadh^{\dt,+}}{b}\right)\right]^2\right] \,\ddx\,\ddt \right)^{\frac{1}{2}}
\label{weakrhob2}\\
& \leq 
C\left(1+\int_{\D_T}\pi_h\left[\,\|\stradh^{\dt,+}\|^4\right]\ddx\,\ddt \right)^{\frac{1}{8}}
\left( \int_{\D_T^b} \pi_h\left[\Bd\left(1-\frac{\varrhoadh^{\dt,+}}{b}\right)\right]^{\frac{8}{3}} 
\ddx\,\ddt \right)^{\frac{3}{8}}.
\nonumber 
\end{align} 
It follows from 
(\ref{use1}), (\ref{GN}), for $d=2$, and (\ref{stab2c}) that
\begin{align}
\int_{\D_T} \pi_h\left[\,\|\stradh^{\dt,+}\|^4\right] \ddx\,\ddt 
\leq C\,\|\stradh^{\dt,+}\|_{L^4(\D_T)}^4 \leq C
\,\int_0^T \!\|\stradh^{\dt,+}\|_{L^2(\D)}^2\,  \|\stradh^{\dt,+}\|_{H^1(\D)}^2\,\ddt 
\leq C.
\label{weakrhob3}
\end{align}
Similarly to (\ref{use1}), it follows from (\ref{inverse}) with $K_k$ replaced by
$(K_k\times(t_{n-1},t_n))\cap\D^b_T$, $k=1,\ldots,N_K$ and $n=1,\ldots,N_T$, that
\begin{align}
\int_{\D_T^b} \pi_h\left[\Bd\left(1-\frac{\varrhoadh^{\dt,+}}{b}\right)\right]^{\frac{8}{3}} 
\ddx\,\ddt & \leq C\left\|
\pi_h\left[\Bd\left(1-\frac{\varrhoadh^{\dt,+}}{b}\right)\right]
\right\|_{L^{\frac{8}{3}}(\D_T^b)}^{\frac{8}{3}}
\leq C\,\|z\|_{L^{\frac{8}{3}}(\D_T^b)}^{\frac{8}{3}},
\label{weakrhob4}
\end{align}
where $z:= 
\pi_h\left[\Bd\left(1-\frac{\varrhoadh^{\dt,+}}{b}\right)\right]-\left(1-\frac{\varrho_\alpha}{b}
\right)$.
Here we have noted that $\varrho_\alpha=\tr(\straL)=b$ a.e.\ in $\D_T^b$.
Combining (\ref{weakrhob2})--(\ref{weakrhob4}) yields, on noting (\ref{GN}), for $d=2$, that 
\begin{align}
T_1 &\leq C\,\|z\|_{L^{\frac{8}{3}}(\D_T)}
\leq C \left(\int_0^T \|z\|_{L^2(\D)}^{2}
\,\|z\|_{H^1(\D)}^{\frac{2}{3}}\,\ddt\right)^{\frac{3}{8}}
\leq C\,\|z\|_{L^3(0,T;L^2(\D))}^{\frac{3}{4}}\, 
\|z\|_{L^2(0,T;H^1(\D))}^{\frac{1}{4}}
\label{weakrhob5}\\
& \leq C\,\|z\|_{L^2(0,T;L^2(\D))}^{\frac{1}{2}}\, 
\|z\|_{L^\infty(0,T;L^2(\D))}^{\frac{1}{4}}\,
\|z\|_{L^2(0,T;H^1(\D))}^{\frac{1}{4}}.
\nonumber
\end{align}
It follows from (\ref{weakrhob5}), (\ref{rhoalpreg}), (\ref{eq:inf-bound}), (\ref{eq:Bd}),
(\ref{eqnorm}), (\ref{stab2crho})
and (\ref{piBdLconvrho}) that $T_1 =0$.
In addition, it follows immediately from (\ref{piBdLconv}) that $T_2=0$.
Hence, we conclude from (\ref{weakrhob}) that $|\D^b_T|=0$, and so   
$\rho_\alpha = \tr(\straL) < b$ a.e.\ in $\D_T$; that is, the second desired result in
(\ref{sigPD}). 

We now establish the other result in (\ref{sigPD}) that $\straL$ is symmetric positive
definite a.e.\ in $\D_T$, which improves on (\ref{stralpreg}). 
This result requires the further assumption that $h=o(\delta)$, as $\delta \rightarrow 0_+$.
Assume that $\straL$ is not symmetric positive
definite a.e.\ in $\D_T^0 \subset \D_T$. 
Let $\bv \in L^\infty(0,T;[L^\infty(\D)]^2)$ be such that
$\straL \,\bv = \bzero$ 
with $\|\bv\|=1$ a.e.\ in $\D_T^0$ and $\bv=\bzero$ a.e.\ in $\D_T 
\setminus 
\D_T^0$. It then follows from 
(\ref{eq:BGd}) and (\ref{Addef}) that
\begin{align}
\label{weakSPD}
|\D^0_T| &= \int_{\D_T} \|\bv\|\,\ddx\,\ddt = \int_{\D_T} 
\left\|\pi_h\left[ \Gd'(\stradh^{\dt,+})\,\Bd(\stradh^{\dt,+})\right]\bv\right\| \ddx\,\ddt
\\
& \leq \int_{\D_T} 
\left\|\pi_h\left[ \Ad(\stradh^{\dt,+},\varrhoadh^{\dt,+})\,\Bd(\stradh^{\dt,+})
\right]\bv\right\| \ddx\,\ddt
\nonumber \\
&\qquad \qquad +
\int_{\D_T}
\left\|\pi_h\left[ \Gd'\left(1-\frac{\varrhoadh^{\dt,+}}{b}\right)\Bd(\stradh^{\dt,+})\right]\bv
\right\|  \ddx\,\ddt =: T_3 + T_4.
\nonumber
\end{align}
A simple variation of (\ref{interprod}), (\ref{ipmat}), (\ref{normv},b) and (\ref{stab2ctr})
yield that 
\begin{align}
T_3 &\leq \int_{\D_T} \left( \pi_h\left[ 
\left\|\Ad(\stradh^{\dt,+},\varrhoadh^{\dt,+})\,[\Bd(\stradh^{\dt,+})]^{\frac{1}{2}}
\right\|^2\right]\right)^{\frac{1}{2}} 
\left(
\pi_h\left[\Bd(\stradh^{\Delta t,+})\right] :: (\bv \,\bv^T)\right)^{\frac{1}{2}}
\ddx \,\ddt
\label{weakSPD1}\\
& \leq C\left(\int_{\D_T}
\pi_h\left[\Bd(\stradh^{\Delta t,+})\right] :: (\bv \,\bv^T)\,
\ddx \,\ddt\right)^{\frac{1}{2}}. 
\nonumber
\end{align}
Then (\ref{piBdLconv}) and the definition of $\bv$ yield that,
as $\delta,\,h,\,\Delta t \rightarrow 0_+$,  
\begin{align}
\int_{\D_T}
\pi_h\left[\Bd(\stradh^{\Delta t,+})\right] :: (\bv \,\bv^T)\,
\ddx \,\ddt \rightarrow \int_{\D_T}
\straL :: (\bv \,\bv^T)\,
\ddx \,\ddt =0,
\label{bvweak}
\end{align}
so we have that $T_3=0$.
Similarly to (\ref{weakSPD1}), on setting 
$\bchi_{\delta,h}^{\dt,+}=\Gd'\left(1-\frac{\varrhoadh^{\dt,+}}{b}\right)
\pi_h\left[\Bd(\stradh^{\dt,+})\right]$,
we have from 
(\ref{interprod})
that
\begin{align}
\label{T4}
T_4 & 
\leq 
\left(\int_{\D_T} \pi_h\left[ 
\left\|[\bchi_{\delta,h}^{\dt,+}]^{\frac{1}{2}}
\right\|^2\right]\ddx \,\ddt \right)^{\frac{1}{2}} 
\left(\int_{\D_T}
\pi_h\left[\bchi_{\delta,h}^{\dt,+}\right]
:: (\bv \,\bv^T)\,
\ddx \,\ddt\right)^{\frac{1}{2}}
\\
& \leq C\,
\left(\int_{\D_T}
\pi_h\left[\bchi_{\delta,h}^{\dt,+}\right]
:: (\bv \,\bv^T)\,
\ddx \,\ddt\right)^{\frac{1}{2}},
\nonumber 
\end{align}
where we have noted from 
(\ref{ipmat}), (\ref{normphi}), (\ref{modphisq}), (\ref{use1}), (\ref{Addef}) 
and (\ref{stabkdAdBda})
that
\begin{align}
\int_{\D_T} \pi_h\left[ 
\left\|[\bchi_{\delta,h}^{\dt,+}]^{\frac{1}{2}}
\right\|^2\right]\ddx \,\ddt & \leq C\,
\left\|\pi_h \left[\,
\left\|\bchi_{\delta,h}^{\dt,+}
\right\|\,\right] \right\|_{L^1(\D_T)}
\label{Aeq}
\\
&\leq C\,
\left\|\pi_h \left[\bchi_{\delta,h}^{\dt,+} \right]\right\|_{L^{1}(\D_T)}
\leq C\,
\left\|\pi_h \left[\bchi_{\delta,h}^{\dt,+} \right]\right\|_{L^{\frac{8}{5}}(\D_T)}
\nonumber \\
& 
\leq C + C\,
\left\|\pi_h\left[
\Ad(\stradh^{\dt,+},\varrhoadh^{\dt,+})
\,\Bd(\stradh^{\dt,+})\right]\right\|_{L^{\frac{8}{5}}(\D_T)}  \leq C.
\nonumber 
\end{align}

We will now show, on possibly extracting a further subsequence of
$\{(\stradh^{\dt},$\linebreak $\varrhoadh^{\dt})\}_{\delta>0,h>0,\Delta t>0}$,      
that   
\begin{align}
\pi_h\left[
\Gd'\left(1-\frac{\varrhoadh^{\dt,+}}{b}\right)
\Bd(\stradh^{\dt,+})
\right]
\rightarrow \left(1-\frac{\rho_\alpha}{b}\right)^{-1} \stra
\quad \mbox{weakly in } L^{\frac{8}{5}}(0,T;[L^{\frac{8}{5}}(\D)]^{2 \times 2}).
\label{keyres}
\end{align}
as $\delta,\,h,\,\Delta t \rightarrow 0_+$ with $h=o(\delta)$.   
It follows immediately from (\ref{Aeq}) and our definition of $\bchi_{\delta,h}^{\dt,+}$
that $\pi_h \left[\bchi_{\delta,h}^{\dt,+} \right]$ converges  
weakly in $L^{\frac{8}{5}}(0,T;[L^{\frac{8}{5}}(\D)]^{2 \times 2})$
to some limit for a subsequence. We just need to show it is the limit stated in 
(\ref{keyres}). We have from (\ref{piBdLconv}), (\ref{psisconL2rho}) and (\ref{eq:Gd})
that $\bchi_{\delta,h}^{\dt,+} \rightarrow 
\left(1-\frac{\rho_\alpha}{b}\right)^{-1} \stra$  
a.e.\ on $\D_T$, for a subsequence, 
as we have already established that $\rho_\alpha = \tr(\stra) < b$
a.e.\ on $\D_T$. So it remains to establish that 
$(I-\pi_h) \left[\bchi_{\delta,h}^{\dt,+} \right]$ converges
to zero a.e.\ on $\D_T$. 
As $\Gd' \in C^{0,1}(\R)$ is monotonic, it follows from 
(\ref{interp2}), (\ref{CME}), (\ref{inverse}), (\ref{eq:BGd}) and
(\ref{stab2c},c) that
\begin{align}
\left\| (I-\pi_h) \left[\bchi_{\delta,h}^{\dt,+} \right] \right\|_{L^1(\D_T)}
& \leq \left\| (I-\pi_h) \left[ \pi_h\left[\Gd'\left(1-\frac{\varrhoadh^{\dt,+}}{b}\right)\right]
\pi_h\left[\Bd(\stradh^{\dt,+})\right]\right] \right\|_{L^1(\D_T)}
\label{keyres1} 
\\
& \qquad + \left\| (I-\pi_h) \left[\Gd'\left(1-\frac{\varrhoadh^{\dt,+}}{b}\right)\right]
\right\|_{L^2(\D_T)}\left\|\pi_h\left[\Bd(\stradh^{\dt,+})\right] \right\|_{L^2(\D_T)}
\nonumber \\
& \leq C\,h \left\|\nabla \pi_h \left[\Gd'\left(1-\frac{\varrhoadh^{\dt,+}}{b}\right)\right]
\right\|_{L^2(\D_T)}\left\|\pi_h\left[\Bd(\stradh^{\dt,+})\right] \right\|_{L^2(\D_T)}
\nonumber \\
&\leq C\,\delta^{-1}\,h.
\nonumber
\end{align}
Hence, we have for a subsequence that $(I-\pi_h) \left[\bchi_{\delta,h}^{\dt,+} \right]$ converges
to zero a.e.\ on $\D_T$ as $\delta,\,h,\,\Delta t \rightarrow 0_+$ with $h=o(\delta)$.   
Therefore we have established (\ref{keyres}).

Similarly to (\ref{bvweak}), we have that (\ref{T4}), (\ref{keyres}) and our definitions of 
$\bchi_{\delta,h}^{\dt,+}$ and $\bv$ yield that $T_4=0$. 
Hence it follows from (\ref{weakSPD}) with $T_3=T_4=0$ that $|\D^0_T| =0$,
and so $\stra$ is positive definite a.e.\ in $\D_T$; that is, the first desired result in 
(\ref{sigPD}).
\end{proof}

\begin{lemma}\label{singconthm}
Under all of the assumptions of Lemma \ref{rhoathm},
a further subsequence of the subsequence of  
$\{(\stradh^{\dt},\varrhoadh^{\dt})\}_{\delta>0,h>0,\Delta t>0}$ 
of Lemma \ref{rhoathm}
and the limiting function $\straL$, 
satisfying (\ref{stralpreg}) 
and (\ref{sigPD}), are such that,
as $\delta,\,h,\,\Delta t \rightarrow 0_+$, with $h=o(\delta)$,  
\begin{subequations}
\begin{align}
&\pi_h\left[\kd(\stradh^{\dt,+},\varrhoadh^{\dt,+})\,
\Bd(\stradh^{\dt,+})\right]
\rightarrow \straL
\quad \hspace{0.45in}\mbox{strongly in } L^{2}(0,T;[L^{2}(\D)]^{2 \times 2}),
\label{kdBdconv}\\
&\pi_h\left[
\Ad(\stradh^{\dt,+},\varrhoadh^{\dt,+})
\,\Bd(\stradh^{\dt,+})\right]
\rightarrow A(\straL)\,\straL
\quad \mbox{weakly in } L^{\frac{8}{5}}(0,T;[L^{\frac{8}{5}}(\D)]^{2 \times 2}),
\label{Adconv}\\
&\pi_h\left[\kd(\stradh^{\dt,+},\varrhoadh^{\dt,+})\,
\Ad(\stradh^{\dt,+},\varrhoadh^{\dt,+})
\,\Bd(\stradh^{\dt,+})\right]
 \rightarrow A(\straL)\,\straL
\label{kdAdconv}\\
&\hspace{2.7in}\quad \mbox{weakly in } L^{2}(0,T;[L^{2}(\D)]^{2 \times 2}).
\nonumber
\end{align}
\end{subequations}
\end{lemma}
\begin{proof}
It follows from (\ref{interprod}), 
(\ref{ipmat}), (\ref{normphi}),
(\ref{kddef}), (\ref{modphisq}), (\ref{eq:Bd}), 
(\ref{eqnorm}) and (\ref{stab2c}) that
\begin{align}
&
\left\|\pi_h\left[\kd(\stradh^{\dt,+},\varrhoadh^{\dt,+})\,
\Bd(\stradh^{\dt,+})\right]- \pi_h\left[\Bd(\stradh^{\dt,+})\right]\right\|^2_{L^2(\D_T)}
\label{kdBdconv1}
\\
& \qquad \leq \int_{\D_T} \pi_h\left[
(\kd(\stradh^{\dt,+},\varrhoadh^{\dt,+})-1)^2 \,\tr(\Bd(\stradh^{\dt,+})) \right]
\,\pi_h \left[\tr(\Bd(\stradh^{\dt,+})) \right]\ddx \,\ddt
\nonumber \\
& \qquad \leq C\,\int_{\D_T} \pi_h \left[\,\left|\Bd^b(\varrhoadh^{\dt,+})-\tr(\Bd(\stradh^{\dt,+}) 
\right| \,\right]
\,\left(\pi_h \left[ \|\stradh^{\dt,+}\|\right] + \delta \right) \ddx \,\ddt
\nonumber \\
& \qquad \leq C\,\left\|\pi_h \left[
\Bd^b(\varrhoadh^{\dt,+})-\tr(\Bd(\stradh^{\dt,+})\right]\right\|_{L^2(\D_T)}.
\nonumber
\end{align}
The desired result (\ref{kdBdconv}) then follows immediately from
(\ref{kdBdconv1}),
(\ref{piBdLconv}) and (\ref{rhosigeq}).

The desired result (\ref{Adconv}) follows immediately from (\ref{Addef}), (\ref{eq:BGd}),
(\ref{keyres}) and (\ref{Adef}) as $\rho_\alpha = \tr(\stra)$.  

Similarly to (\ref{kdBdconv1}), it follows from
(\ref{interprod}), (\ref{ipmat}), (\ref{normphi}), (\ref{kddef}), 
(\ref{eqnorm}),  and (\ref{stab2ctr}) that
\begin{align} 
&\left\|\pi_h\left[\kd(\stradh^{\dt,+},\varrhoadh^{\dt,+})\,\Ad(\stradh^{\dt,+},\varrhoadh^{\dt,+})\,
\Bd(\stradh^{\dt,+})\right]- \pi_h\left[
\Ad(\stradh^{\dt,+},\varrhoadh^{\dt,+})\,
\Bd(\stradh^{\dt,+})\right]\right\|_{L^1(\D_T)}
\label{kdAdconv1}
\\
&\quad
\leq C\int_{\D_T} \left(\pi_h \left[\,
\left|\Bd^b(\varrhoadh^{\dt,+})-\tr(\Bd(\stradh^{\dt,+})\right|\, \right]\right)^{\frac{1}{2}}
\nonumber \\
& \hspace{1in} \times
\left(\pi_h \left[ \left\|\Ad(\stradh^{\dt,+},\varrhoadh^{\dt,+})\,
\left[\Bd(\stradh^{\dt,+})\right]^{\frac{1}{2}}\right\|^2 \right] \right)^{\frac{1}{2}} \ddx \,\ddt
\nonumber \\
& \quad 
\leq C \left\|\pi_h \left[\, 
\left|\Bd^b(\varrhoadh^{\dt,+})-\tr(\Bd(\stradh^{\dt,+})
\right|\,\right]\right\|_{L^1(\D_T)}^{\frac{1}{2}}
\leq C \left\|\pi_h \left[
\Bd^b(\varrhoadh^{\dt,+})-\tr(\Bd(\stradh^{\dt,+})\right]\right\|_{L^2(\D_T)}^{\frac{1}{2}}.
\nonumber
\end{align}
It follows immediately from
(\ref{kdAdconv1}),
(\ref{piBdLconv}), (\ref{rhosigeq}) and (\ref{stabkdAdBd})
that the weak limits in (\ref{Adconv},c) are the same.
Hence, the desired result (\ref{kdAdconv}).
\end{proof}

\begin{theorem}\label{exPathm}
Under all of the assumptions of Lemma \ref{rhoathm}, the limiting functions
$(\buaL,\straL)$ satisfying (\ref{ualpreg},b) 
and (\ref{sigPD})
solve the following problem:

{\rm ({\bf P}$_\alpha$)} Find $\buaL \in L^{\infty}(0,T;{\rm H})
\cap L^{2}(0,T;\Uz) \cap W^{1,\frac{4}{\vartheta}}(0,T;\Uz')$
and $\straL
\in L^{\infty}(0,T;[L^{2}(\D)]^{2 \times 2}_{{\rm S},> 0,b})\cap
L^{2}(0,T;[H^{1}(\D)]^{2 \times 2}_{\rm S})\cap H^{1}(0,T;([H^{1}(\D)]^{2 \times 2}_{\rm S})')$,
with $A(\stra)\,\stra \in L^2(0,T;([L^{2}(\D)]^{2 \times 2}_{\rm S}))$,
such that $\bua(0)=\bu^0$, $\stra(0)=\strs^0$ and 
\begin{subequations}
\begin{align}
\label{weak1d}
&{\rm Re} \displaystyle\int_{0}^{T}  \left\langle \frac{\partial \buaL}{\partial t},
\bv \right\rangle_{\Uz}
\ddt  
+ \int_{\D_T} \left[ 
(1-\e) \,\grad \buaL: \grad \bv +
{\rm Re} \left[ (\buaL \cdot \grad) \buaL
\right]\,\cdot\,\bv \right] \, \ddx\,\ddt
\\
&\hspace{0.5in} =
\int_0^T \langle \f , \bv \rangle_{H^1_0(\D)} \,\ddt
- \frac{\e}{\rm Wi} \int_{\D_T} A(\straL)\,\straL 
: \grad \bv \, \ddx\,\ddt  
\qquad \forall \bv \in L^{\frac{4}{4-\vartheta}}(0,T;\Uz),
\nonumber
\\
&\int_{0}^T  
\left \langle \frac{\partial \straL}{\partial t}
,\bphi
\right \rangle_{H^1(\D)} \ddt
+ \int_{\D_T} \left[
(\buaL \cdot \grad) \straL : \bphi +
\alpha\,\grad \straL :: \grad \bphi \right] \,  \ddx\,\ddt
\label{weak2d} 
\\
&\hspace{0.2in} = \int_{\D_T}
\left[2\,(\grad \buaL)\,\straL - \frac{1}{\rm Wi}
A(\straL)\,\straL \right] : \bphi \, \ddx\,\ddt
\qquad \forall
\bphi \in L^{2}(0,T;[H^1(\D)]^{2 \times 2}_{\rm S}),
\nonumber
\end{align}
\end{subequations}
where $\vartheta \in (2,4)$.
\end{theorem}
\begin{proof}
The function spaces and the initial conditions for $(\buaL,\straL)$  
follow immediately from \linebreak
(\ref{ualpreg},b), (\ref{sigPD}) and (\ref{kdAdconv}).
It remains to prove that $(\buaL,\straL)$ satisfy (\ref{weak1d},b).
It follows from (\ref{Vhconv}),
(\ref{uwconH1}--d), 
(\ref{kdAdconv}), 
(\ref{fnconv}), (\ref{Sw})
and (\ref{conv0c})
that we may pass to the limit, $\delta,\,h,\,\Delta t \rightarrow 0_+$, with $h=o(\delta)$, in
(\ref{equncon}) to obtain that $(\buaL, 
\straL)$ 
satisfy (\ref{weak1d}).

It follows from (\ref{psiwconH1x}--f), (\ref{uwconH1},d), 
(\ref{kdBdconv},b), (\ref{Echi}), 
(\ref{interp1},b), 
(\ref{eq:symmetric-tr}) and as $\buaL \in L^2(0,T;{\rm V})$
that we may pass to the limit $\delta,\,h,\,\Delta t \rightarrow 0_+$, with $h=o(\delta)$,
in 
(\ref{eqpsincon}) with $\bchi= \pi_h \bphi$
to obtain (\ref{weak2d}) for any
$\bphi \in C^\infty_0(0,T;[C^\infty(\overline{\D})]^{2 \times 2}_{\rm S})$.
For example, similarly to (\ref{intpartstr}), in order to pass to the limit on the first term in 
(\ref{eqpsincon}),
we note that
\begin{align}
\label{intparts} 
&\int_{\D_T} \pi_h \left[ \left(\frac{\partial \stradh^{\Delta t}}{\partial t}
+  \frac{\Ad(\stradh^{\dt,+},\varrhoadh^{\dt,+})\,
\Bd(\stradh^{\dt,+})}{\Wi} \right)
 : \pi_h\bphi \right]  \ddx\,\ddt \\
&\qquad 
=
\int_{\D_T} 
\left(\frac{\partial \stradh^{\Delta t}}{\partial t}
+ \frac{\pi_h\left[\Ad(\stradh^{\dt,+},\varrhoadh^{\dt,+})\,
\Bd(\stradh^{\dt,+})\right]
}{\Wi}\right) : \pi_h \bphi 
\,\ddx\,\ddt
\nonumber \\
& \qquad \qquad  + \int_{\D_T} 
(I-\pi_h) \left[
\stradh^{\Delta t} : 
\pi_h \left[\frac{\partial \bphi}{\partial t}\right]
-\frac{\pi_h \left[\Ad(\stradh^{\dt,+},\varrhoadh^{\dt,+})\,
\Bd(\stradh^{\dt,+})\right]}{\Wi} : 
 \pi_h \bphi
\right] \ddx\,\ddt.
\nonumber
\end{align}
The desired result (\ref{weak2d}) then follows from noting that
$ C^\infty_0(0,T;[C^\infty(\overline{\D})]^{2 \times 2}_{\rm S})$
is dense in
$L^2(0,T;$ $[H^1(\D)]^{2 \times 2}_{\rm S})$.

Of course passing to the limit 
$\delta,\,h,\,\Delta t \rightarrow 0_+$, with $h=o(\delta)$, in (\ref{eqvarrhocon}),
using in addition (\ref{psiwconH1xrho}--d,f), 
yields the weak formulation for $\tr(\stra)$ consistent with (\ref{weak2d}).  
\end{proof}

\begin{remark}\label{Lalphaindrem}
Choosing $\frac{1}{2}(\bphi +\bphi^T)$ as a test function 
in (\ref{weak2d}) for any $\bphi \in L^2(0,T;[H^1(\D)]^{2 \times 2})$ yields, on noting
the symmetry of $\stra$, (\ref{weak2d}) with the term $2 \,(\gbua)\stra$   
replaced by $(\gbua)\stra+\stra(\gbua)^T$, which is consistent with 
(\ref{eq:aoldroyd-b-sigma2}). 

Finally, it follows from (\ref{stab1c},f,h), (\ref{uwconL2}--c) and 
(\ref{kdAdconv}) 
that
\begin{align}
\sup_{t \in (0,T)} \|\buaL \|^2_{L^2(\D)}
+\int_{0}^T 
\left[ \|\grad \buaL \|^2_{L^2(\D)} + \left \|
{\mathcal S} \frac{\partial \buaL}{\partial t} \right\|_{H^1(\D)}^{\frac{4}{\vartheta}}
+ \|A(\stra)\,\stra\|_{L^2(\D)}^2 
\right]
\ddt
&\leq C,
\label{finalbd}
\end{align}
where $\vartheta \in (2,4)$ and  
$C$ is independent of the stress diffusion coefficient $\alpha$.
Of course, in addition, it follows from (\ref{sigPD}) and (\ref{modphisq}) 
that $\|\stra\|_{L^\infty(\D_T)} <b$.  
\end{remark}

\bibliographystyle{ws-m3as}
\def\cprime{$'$} \def\cprime{$'$}
  \def\ocirc#1{\ifmmode\setbox0=\hbox{$#1$}\dimen0=\ht0\advance\dimen0
  by1pt\rlap{\hbox to\wd0{\hss\raise\dimen0
  \hbox{\hskip.2em$\scriptscriptstyle\circ$}\hss}}#1\else {\accent"17 #1}\fi}
  \def\cprime{$'$} \def\cprime{$'$} \def\cprime{$'$}
  \def\soft#1{\leavevmode\setbox0=\hbox{h}\dimen7=\ht0\advance \dimen7
  by-1ex\relax\if t#1\relax\rlap{\raise.6\dimen7
  \hbox{\kern.3ex\char'47}}#1\relax\else\if T#1\relax
  \rlap{\raise.5\dimen7\hbox{\kern1.3ex\char'47}}#1\relax \else\if
  d#1\relax\rlap{\raise.5\dimen7\hbox{\kern.9ex \char'47}}#1\relax\else\if
  D#1\relax\rlap{\raise.5\dimen7 \hbox{\kern1.4ex\char'47}}#1\relax\else\if
  l#1\relax \rlap{\raise.5\dimen7\hbox{\kern.4ex\char'47}}#1\relax \else\if
  L#1\relax\rlap{\raise.5\dimen7\hbox{\kern.7ex
  \char'47}}#1\relax\else\message{accent \string\soft \space #1 not
  defined!}#1\relax\fi\fi\fi\fi\fi\fi}
  \def\ocirc#1{\ifmmode\setbox0=\hbox{$#1$}\dimen0=\ht0 \advance\dimen0
  by1pt\rlap{\hbox to\wd0{\hss\raise\dimen0
  \hbox{\hskip.2em$\scriptscriptstyle\circ$}\hss}}#1\else {\accent"17 #1}\fi}
  \def\ocirc#1{\ifmmode\setbox0=\hbox{$#1$}\dimen0=\ht0 \advance\dimen0
  by1pt\rlap{\hbox to\wd0{\hss\raise\dimen0
  \hbox{\hskip.2em$\scriptscriptstyle\circ$}\hss}}#1\else {\accent"17 #1}\fi}
  \def\ocirc#1{\ifmmode\setbox0=\hbox{$#1$}\dimen0=\ht0 \advance\dimen0
  by1pt\rlap{\hbox to\wd0{\hss\raise\dimen0
  \hbox{\hskip.2em$\scriptscriptstyle\circ$}\hss}}#1\else {\accent"17 #1}\fi}

\end{document}